\documentclass[11pt,a4paper]{article}

\usepackage{authblk}
\usepackage[margin=1in]{geometry}
\usepackage{amsmath,amsthm,amssymb,amsfonts,mathrsfs,mathtools}
\usepackage{bm}
\usepackage{enumitem}
\usepackage{microtype}
\usepackage{cite}
\usepackage[hidelinks]{hyperref}
\renewenvironment{abstract}
{\small\noindent\textbf{\abstractname.}}
{\par}
\allowdisplaybreaks
\numberwithin{equation}{section}

\newtheorem{theorem}{Theorem}[section]
\newtheorem{proposition}[theorem]{Proposition}
\newtheorem{lemma}[theorem]{Lemma}
\newtheorem{corollary}[theorem]{Corollary}
\theoremstyle{definition}
\newtheorem{definition}[theorem]{Definition}
\newtheorem{example}[theorem]{Example}
\theoremstyle{remark}
\newtheorem{remark}[theorem]{Remark}

\newcommand{\Bcal}{\mathcal B}
\newcommand{\N}{\mathbb N}
\newcommand{\Rplus}{[0,\infty)}
\newcommand{\U}{\mathfrak U}
\newcommand{\Ucom}{\mathfrak U_{\mathrm{com}}}

\newcommand{\supp}{\operatorname{supp}}
\newcommand{\essinf}{\operatorname*{ess\,inf}}
\newcommand{\esssup}{\operatorname*{ess\,sup}}
\newcommand{\diag}{\operatorname{diag}}
\newcommand{\sech}{\operatorname{sech}}
\newcommand{\Id}{I}
\newcommand{\dd}{\,\mathrm d}
\newcommand{\e}{\mathrm e}
\newcommand{\ind}{\mathbf 1}

\begin{document}

\title{Density and Spectral Criteria for Universal Nonuniform Sampling of Positive Operator Orbits}

\author{Jian Wu\thanks{Corresponding author: \texttt{xingxingwu2022@163.com}}}
\affil{}
\renewcommand*{\Affilfont}{\small\itshape}
\date{}

\maketitle

\begin{abstract}
We characterize locally finite sets \(T\subset[0,\infty)\) that preserve exact observability from nonnegative integer times for every bounded positive operator \(A\) and bounded observation operator \(C\). The Stein identity reduces universality to uniform upper and lower bounds on radial Stein--Pick cones over all finite spectral models. Writing \(N_T(R)=\#(T\cap[0,R])\), we prove that universal Bessel preservation is equivalent to \(N_T(R)=O(R)\). The conditions \(0\in T\) and \(N_T(R)\asymp R\) characterize universal observability preservation under \([A,C^*C]=0\), as well as uniform bound preservation for each fixed number of spectral values. However, we construct a clustered set satisfying these conditions whose optimal sampled lower bounds tend to zero as the spectral count grows, despite fixed integer observability bounds. The same set destroys the frame property of a singly generated positive diagonal Carleson frame while preserving its Bessel property. When natural density exists, universality is equivalent to \(0\in T\) and \(0<d(T)<\infty\). More generally, containing zero and a subset of positive finite natural density suffices under linear upper counting. For sets containing zero, universality is invariant under sublinear differences of counting functions. For general normal operators, universal Bessel preservation is instead equivalent to uniformly bounded counts on unit intervals; under a fixed logarithmic sector condition, it is again characterized by linear upper counting.
\end{abstract}



\section{Introduction}

Sampling theory asks when discrete measurements permit stable recovery.
Its classical foundations include Whittaker's interpolation theory and
Shannon's sampling theorem for bandlimited signals
\cite{WhittakerSampling,ShannonSampling}, while frame inequalities express
stability beyond regular sampling
\cite{DuffinSchaefferNonharmonic}.
Density conditions and reconstruction methods are central to irregular
sampling \cite{LandauDensity,GrochenigIrregular}, with extensions to
shift-invariant and reproducing kernel Hilbert spaces
\cite{AldroubiGrochenigSampling,FuhrEtAlDensity}.
Universal sampling asks whether one sampling set can serve a prescribed
class of signal spaces \cite{OlevskiiUlanovskiiUniversal}.
When measurements are taken during an evolution, a related question is
which changes of the observation times preserve an existing stable
sampling scheme.

Dynamical sampling studies recovery of an initial state from measurements
of its subsequent evolution
\cite{AldroubiDavisKrishtal,AldroubiEtAlDynamical}.
The associated theory of operator-orbit frames includes spectral
conditions for normal operators and results for finite index sets
\cite{AldroubiEtAlNormal,CabrelliEtAlFiniteIndex}, frame representations
by bounded operators and properties of their iterates
\cite{ChristensenHasannasabIterated,ChristensenHasannasabRashidi,ChristensenHasannasabPhilipp},
and Bessel criteria for normal operator orbits \cite{PhilippBessel}.
Related developments concern shift-preserving operators
\cite{AguileraEtAlShift}, commuting operator orbits
\cite{AguileraEtAlTwoOperators}, and hyperinvariant subspaces
\cite{BaileyEtAlDynamical}.
Continuous powers \cite{AldroubiHuangPetrosyan}, Kadec-type sampling
theorems \cite{KrishtalMillerKadec}, and functional calculus constructions
\cite{KrishtalMashreghiMiller} extend the study beyond integer iterates.
For a recent survey, see \cite{AldroubiSurvey}.

We consider positive operator evolutions. Unless stated otherwise,
$H,Y$ are separable complex Hilbert spaces, $A\in\Bcal(H)$ is positive,
and $C\in\Bcal(H,Y)$ is a bounded observation operator.
The output associated with $x\in H$ is $CA^tx$ for $t\geq0$, where
real powers are defined by continuous functional calculus and $A^0=\Id$.
Exact observability at the nonnegative integers means that
\[
a\|x\|^2\leq\sum_{n\geq0}\|CA^nx\|^2\leq b\|x\|^2,
\qquad x\in H,
\]
for some $0<a\leq b<\infty$.
This is also the $g$-frame property of $\{CA^n\}_{n\geq0}$
\cite{SunGFrames}; its upper and lower inequalities connect dynamical
sampling with admissibility and observability in systems theory
\cite{WeissAdmissible,RussellWeissExact,PartingtonPott,XuLiuYung,HaakOuhabaz,TucsnakWeiss}.

We ask which locally finite sets $T\subset\Rplus$ preserve exact
observability for every such system. Writing $T\in\U$ for this property,
the requirement is
\[
\{CA^n\}_{n\geq0}\text{ is exactly observable}
\quad\Longrightarrow\quad
\{CA^t\}_{t\in T}\text{ is exactly observable}
\]
for every bounded positive $A$ and bounded $C$, on arbitrary
separable Hilbert spaces.
Here local finiteness means that $T\cap[0,R]$ is finite for every $R>0$.
The same time set must work across arbitrary observation spaces and
spectral configurations. Positivity already permits substantial regular
subsampling: for each $N\in\mathbb N^+$, exact observability at the
integers is equivalent to exact observability at $N\N$.
Our aim is to determine how far counting and density conditions
characterize this universal preservation, and what additional
obstructions arise from interactions among spectral components.

We reduce this problem to finite spectral models through the integer
observability Gramian and its Stein identity,
\(
S=\sum_{n\geq0}A^nC^*CA^n,\quad C^*C=S-ASA.
\)
Integer exact observability forces $A$ to be a strongly stable contraction.
For finite spectra, the Stein identity expresses positivity of $C^*C$
as a matrix constraint and sampling as an entrywise multiplier.
Theorem~\ref{thm:complete-characterization} shows that $T\in\U$ precisely
when the resulting maps preserve upper and lower Gramian bounds on
radial Stein--Pick cones uniformly over all finite spectral models.
For each fixed pair of integer observability bounds, the sampled bounds
may depend on $T$, but must be independent of the number and location
of spectral points and their multiplicities.

Writing $N_T(R)=\#(T\cap[0,R])$, we obtain a complete counting
criterion for the upper bound:
universal Bessel preservation is equivalent to $N_T(R)=O(R)$
(Theorem~\ref{thm:universal-bessel}).
Scalar models show that universal observability preservation additionally
requires $0\in T$ and $N_T(R)\asymp R$.
These two conditions characterize universal preservation under
$[A,C^*C]=0$. They also characterize uniform bound preservation for
each fixed maximal number of spectral values, with arbitrary
multiplicities (Theorem~\ref{thm:fixed-spectral-count}).
In the latter statement, the constants may depend on that number.

The distinction between fixed and unrestricted spectral counts is
essential. We construct a clustered set satisfying $0\in T$ and
$N_T(R)\asymp R$ whose optimal uniform sampled lower bounds tend to
zero as the spectral count grows, while the integer observability
bounds remain fixed. Thus $\U\subsetneq\Ucom$, where $\Ucom$ denotes
the commuting universal class.
The same set destroys the frame property of a singly generated positive
diagonal Carleson frame while preserving its Bessel property, and the
construction extends to every prescribed finite number of generating
orbits (Theorem~\ref{thm:mellin-finite-orbit}).

This counterexample strengthens the obstruction established by
Gallardo-Guti\'errez and Partington
\cite[Example~3.4]{GallardoPartingtonOperatorOrbits} to the
M\"untz--Sz\'asz sampling conjecture discussed in
\cite{AldroubiSurvey}.
Their examples have vanishing scalar lower bounds, whereas our set satisfies
\[
0<c_T\leq(1-r^2)\sum_{t\in T}r^{2t}\leq C_T<\infty,
\qquad 0<r<1.
\]
Its failure therefore reflects cancellation between spectral scales
that scalar tests do not detect.

Positive results for Carleson frames provide useful benchmarks.
Christensen, Hasannasab, Philipp, and Stoeva
\cite[Theorem~2.1]{ChristensenEtAlMystery} proved preservation under
selection of every $N$th element, and Krishtal and Miller
\cite{KrishtalMillerDemystifying} established preservation under
choosing one iterate from each block of $N$ consecutive integers.
Krishtal and Miller \cite{KrishtalMillerBlock} also obtained a natural
density criterion for block diagonal Carleson frames with positive
spectrum. For arbitrary positive operator systems, we prove that,
whenever the natural density $d(T)=\lim_{R\to\infty}N_T(R)/R$ exists,
\[
T\in\U
\quad\Longleftrightarrow\quad
0\in T \text{ and }0<d(T)<\infty.
\]
The clustered counterexample shows that the density hypothesis cannot
be replaced by two-sided linear counting, even for a single diagonal
orbit, addressing the extension discussed in
\cite[Remark~4.2]{KrishtalMillerBlock}.

More generally, a subset of positive finite natural density supplies
the lower bound: if $T$ contains such a subset, then $T\in\U$ precisely
when $0\in T$ and $N_T(R)=O(R)$.
The corresponding subset observation in the block diagonal Carleson
setting appears in \cite[Remark~4.2]{KrishtalMillerBlock}.
Positive lower Beurling density supplies such a subset, while finite
upper Beurling density is unnecessary; local clusters of unbounded
size are allowed.
We also prove that under $N_T(R)\asymp R$, exact observability at $T$
implies exact observability at the nonnegative integers
(Proposition~\ref{prop:reverse-preservation}).

The density analysis also gives stability under changes in the
counting function. If $0\in T\cap T_0$ and
$N_T(R)-N_{T_0}(R)=o(R)$, then $T\in\U$ if and only if $T_0\in\U$
(Theorem~\ref{thm:counting-stability}).
This invariance is distinct from quantitative comparison with a
fixed lattice. For regular power perturbations $t_k=Nk+ck^\gamma$,
where $N\in\mathbb N^+$ and $c,\gamma>0$, the quadratic discrepancy
vanishes at the spectral boundary precisely when $\gamma<1/2$,
whereas universality holds precisely when $\gamma\leq1$.

For general normal operators, universal Bessel preservation is
equivalent to uniformly bounded counts on unit intervals.
A logarithmic sector condition
$|\arg\lambda|\leq M(-\log|\lambda|)$, with fixed $M\geq0$,
on the nonzero spectrum restores the linear upper counting criterion.
Under this sector condition, if the integer system is exactly
observable and the observation space is finite dimensional,
positive finite natural density implies that
$G_T-d(T)S$ is compact.
Consequently, exact observability of the sampled system is equivalent
to completeness (Corollary~\ref{cor:normal-finite-observation}).

The observability formulation translates these results into sampling
statements for operator-orbit frames. For a finite or countable Bessel
family $\mathcal F=(f_i)_{i\in I}$, taking
$C_{\mathcal F}x=(\langle x,f_i\rangle)_{i\in I}$ identifies
observability with the frame property of the corresponding adjoint
orbits, which are ordinary orbits when $A$ is positive.

Section~2 establishes the uniform Stein--Pick characterization.
Section~3 treats scalar and commuting systems, fixed spectral counts,
and the clustered obstruction.
Section~4 proves the counting, density, and stability criteria.
Section~5 studies quantitative perturbation scales.
Section~\ref{sec:normal-extension} develops the normal operator
extensions, and Section~7 gives the orbit formulations and the
single-orbit realization of the obstruction.

\section{Uniform Stein--Pick bounds and universal sampling}
We express sampled observability through observability Gramians and
the Stein identity, and prove that universal sampling is equivalent
to uniform preservation of upper and lower observability bounds
over all finite spectral models. The constants may depend on \(T\)
and the integer observability bounds, but must be independent of
the number and location of spectral points and their multiplicities.

We set \(\N=\{0,1,2,\ldots\}\) and
\(\mathbb N^+=\{1,2,\ldots\}\).
No finite time set can be universal. Whenever an infinite locally
finite set \(T\subset\Rplus\) is written in increasing order, we use
\(
T=\{\tau_k:k\in\N\},
0\leq\tau_0<\tau_1<\cdots .
\)
Local finiteness then implies that \(\tau_k\to\infty\).

We call the indexed operator family
\(\{CA^t\}_{t\in T}\)
the \(T\)-sampled observation family for \((A,C)\), or simply the
\(T\)-system when the pair \((A,C)\) is clear from the context.
When \(T=\N\), we refer to it as the \(\N\)-system.
For infinite \(T\), whenever the \(T\)-system is Bessel, its analysis
operator is
\[
\Gamma_T:H\longrightarrow\ell^2(\N;Y),
\qquad
\Gamma_Tx=\{CA^{\tau_k}x\}_{k\geq0}.
\]
The \(T\)-system is said to be exactly observable if it is a
\(g\)-frame for \(H\).

The following observation specializes
\cite[Lemma~3.1 and Theorem~3.3(i)]{CabrelliEtAlFiniteIndex}
to positive operators; we include the short proof for completeness.

\begin{lemma}
\label{lem:auto-contraction}
Suppose that the $\N$-system is exactly observable. Then
$0\leq A\leq\Id$ and
$A^n x\rightarrow0$
for every $x\in H$. Consequently, if $E_A$ denotes the spectral
measure of $A$, then $E_A(\{1\})=0$.
\end{lemma}

\begin{proof}
Let $0<a\leq b<\infty$ be observability bounds for the
$\N$-system. Let $B=C^*C$, and set
$S_m=\sum_{n=0}^{m}A^nBA^n$.
For every $x\in H$,
$\langle S_mx,x\rangle
 =\sum_{n=0}^{m}\lVert CA^nx\rVert^2
 \leq b\lVert x\rVert^2$.
Thus $0\leq S_m\leq b\Id$. Since $(S_m)_{m\geq0}$ is increasing, it
converges in the strong operator topology to a bounded operator
$S=\sum_{n\geq0}A^nBA^n$.
Moreover, $\langle Sx,x\rangle
 =\sum_{n\geq0}\lVert CA^nx\rVert^2$,
and hence the $g$-frame bounds give
$a\Id\leq S\leq b\Id$.
By the definition of $S$, $S-ASA=B\geq0$.

Set
$R=S^{1/2}AS^{-1/2}$.
Then
\[
 R^*R
 =S^{-1/2}ASAS^{-1/2}
 =\Id-S^{-1/2}BS^{-1/2}
 \leq\Id.
\]
Therefore $R$ is a contraction. Since $A$ and $R$ are similar,
$r(A)=r(R)\leq\lVert R\rVert\leq1,$
where $r(\cdot)$ denotes the spectral radius. Since $A$ is positive,
$\lVert A\rVert=r(A)$, and therefore
$0\leq A\leq\Id$.

Since $B=S-ASA$, we have
$A^nBA^n
 =A^nSA^n-A^{n+1}SA^{n+1}.$
Summing from $n=0$ to $n=m$ gives
$S_m=S-A^{m+1}SA^{m+1}$.
Since $S_m\to S$ strongly,
$A^{m+1}SA^{m+1}\longrightarrow0$
in the strong operator topology. Consequently, for every $x\in H$,
\[
 a\lVert A^{m+1}x\rVert^2
 \leq
 \langle SA^{m+1}x,A^{m+1}x\rangle
 =
 \langle (S-S_m)x,x\rangle
 \longrightarrow0.
\]
Thus $A^nx\to0$ for every $x\in H$.

Finally, if $x\in\operatorname{Ran}E_A(\{1\})$, then $A^nx=x$ for
every $n\geq0$. The strong stability of $A$ therefore implies $x=0$,
and hence $E_A(\{1\})=0$.
\end{proof}
\begin{remark}
Let $0\leq A\leq\Id$ and suppose that $A^n\to0$ strongly. Given a
positive invertible operator $S\in\Bcal(H)$, set
$B=S-ASA$ and $R=S^{1/2}AS^{-1/2}$.
Then $S^{-1/2}BS^{-1/2}=\Id-R^*R$.
Consequently, $B\geq0\Longleftrightarrow
 R\ \text{is a contraction}$.

If these conditions hold, let $C=B^{1/2}$. For every $m\geq0$,
\[
 \sum_{n=0}^{m}A^nC^*CA^n
 =\sum_{n=0}^{m}A^nBA^n
 =S-A^{m+1}SA^{m+1}.
\]
Moreover, for every $x\in H$,
$\lVert A^{m+1}SA^{m+1}x\rVert
 \leq \lVert SA^{m+1}x\rVert
 \rightarrow0$.
Hence
$\sum_{n\geq0}A^nC^*CA^n=S$
in the strong operator topology. Thus, for a fixed strongly stable positive contraction $A$, a positive
invertible operator $S$ can be represented as
$S=\sum_{n\geq0}A^nC^*CA^n$
for some $C\in\Bcal(H)$ if and only if $S-ASA\geq0$; see also
\cite[Theorem~3.2]{CabrelliEtAlFiniteIndex}.
\end{remark}

For $m\in\N$, define the partial observability Gramian associated with
$T$ by
$
G_{T,m}
:=
\sum_{k=0}^{m}A^{\tau_k}C^*CA^{\tau_k}.
$
Whenever the $T$-system is Bessel, $(G_{T,m})_{m\geq0}$ converges strongly
to the $T$-sampled observability Gramian
$
G_T
:=
\sum_{k\geq0}A^{\tau_k}C^*CA^{\tau_k}
=
\Gamma_T^*\Gamma_T.
$
Equivalently, $G_T$ is the $g$-frame operator of
$\{CA^{\tau_k}\}_{k\geq0}$. In particular, when $T=\N$, we write
$
G_{\N}
=
\sum_{n\geq0}A^nC^*CA^n
$.

\begin{lemma}
\label{lem:fixed-criterion}
The family $\{CA^{\tau_k}\}_{k\geq0}$ is exactly observable if and only if $G_{T,m}$
converges in the strong operator topology to a bounded positive invertible
operator $G_T$.
Its optimal lower and upper bounds are
$a_T=\lVert G_T^{-1}\rVert^{-1},\
b_T=\lVert G_T\rVert$.

Moreover, if $E_A$ denotes the spectral measure of $A$, then
\begin{equation}\label{eq:double-spectral-gramian}
G_T
=
\lim_{m\to\infty}
\int_{\sigma(A)}\!\int_{\sigma(A)}
K_{T,m}(\lambda,\mu)\,
E_A(\dd\lambda)C^*C E_A(\dd\mu),
\qquad
K_{T,m}(\lambda,\mu)
=
\sum_{k=0}^{m}(\lambda\mu)^{\tau_k},
\end{equation}
where the limit is taken in the strong operator topology.
\end{lemma}

\begin{proof}
Apply the standard characterization of $g$-frames
\cite[Section~2]{SunGFrames} to the operators
$
\Lambda_k=CA^{\tau_k}$ for $k\geq0.$
It follows that $\{CA^{\tau_k}\}_{k\geq0}$ is exactly observable if and only
if its $g$-frame operator
$
G_T=\Gamma_T^*\Gamma_T
$
is bounded, positive, and invertible. The optimal lower and upper bounds
are $\lVert G_T^{-1}\rVert^{-1}$ and
$\lVert G_T\rVert,$
respectively.

For each $m\geq0$, the spectral theorem gives
\[
G_{T,m}
=
\int_{\sigma(A)}\!\int_{\sigma(A)}
K_{T,m}(\lambda,\mu)\,
E_A(\dd\lambda)C^*C E_A(\dd\mu).
\]
Since $G_{T,m}\to G_T$ in the strong operator topology,
\eqref{eq:double-spectral-gramian} follows.
\end{proof}

In general, \eqref{eq:double-spectral-gramian} cannot be reduced to a
single spectral integral. Indeed, when $C^*C$ does not commute with $A$,
different spectral subspaces of $A$ may interact through the terms
$
E_A(\dd\lambda)C^*C E_A(\dd\mu).
$
We next use this representation to study universal preservation.

\begin{definition}
A locally finite set $T\subset\Rplus$ is called a
\emph{universal time set} if, for every pair of separable complex Hilbert
spaces $H$ and $Y$, every positive operator $A\in\Bcal(H)$, and every
$C\in\Bcal(H,Y)$, exact observability of the $\N$-system implies exact
observability of the $T$-system. The class of all universal time sets is
denoted by $\U$.
\end{definition}

\begin{proposition}
\label{prop:uniformization}
Let $T\in\U$. For every $0<a\leq b<\infty$, there exist constants
$
0<\alpha_T(a,b)\leq\beta_T(a,b)<\infty
$
such that, for every pair of separable Hilbert spaces $H,Y$, every positive
operator $A\in\Bcal(H)$, and every $C\in\Bcal(H,Y)$, the inequalities
$
a\Id\leq G_{\N}\leq b\Id
$
imply
$
\alpha_T(a,b)\Id
\leq G_T\leq
\beta_T(a,b)\Id.
$
\end{proposition}

\begin{proof}
Suppose first that no uniform upper bound exists. Then there is a sequence
of systems
$
A_j\in\Bcal(H_j),\ C_j\in\Bcal(H_j,Y_j),
$
such that
$
a\Id\leq G_{\N,j}\leq b\Id,
\
\lVert G_{T,j}\rVert\rightarrow\infty.
$
By Lemma~\ref{lem:auto-contraction}, $0\leq A_j\leq\Id$. Moreover,
$
C_j^*C_j\leq G_{\N,j}\leq b\Id,
$
so the operators $C_j$ are uniformly bounded. Hence
$
A=\bigoplus_{j\geq1}A_j,
\
C=\bigoplus_{j\geq1}C_j
$
define bounded operators from
$
H_\oplus=\bigoplus_{j\geq1}H_j$
to $H_\oplus$ and
$Y_\oplus=\bigoplus_{j\geq1}Y_j,
$
respectively, and $A$ is positive.

For the direct-sum system, the preceding definition gives
$
G_{\N}
=
\bigoplus_{j\geq1}G_{\N,j},
$
and therefore
$
a\Id\leq G_{\N}\leq b\Id.
$
On the other hand, $\lVert G_{T,j}\rVert\to\infty$ implies that the
$T$-system of the direct-sum system is not Bessel. This contradicts
$T\in\U$. Thus a finite uniform upper bound $\beta_T(a,b)$ exists.

Suppose next that no uniform positive lower bound exists. Then one can
choose systems such that
$
a\Id\leq G_{\N,j}\leq b\Id,
\
\lVert G_{T,j}^{-1}\rVert^{-1}\rightarrow0.
$
Form the corresponding direct-sum system as above. By the uniform upper
bound already proved,
$
G_{T,j}\leq\beta_T(a,b)\Id
$
for every $j$. Hence its $T$-sampled observability Gramian is the bounded operator
$
G_T=\bigoplus_{j\geq1}G_{T,j}.
$
However,
\[
\inf\sigma(G_T)
=
\inf_{j\geq1}\inf\sigma(G_{T,j})
=
\inf_{j\geq1}\lVert G_{T,j}^{-1}\rVert^{-1}
=0.
\]
Thus $G_T$ is not bounded below, and consequently the $T$-system of the
direct-sum system is not exactly observable. This again contradicts
$T\in\U$.

Therefore a uniform positive lower bound $\alpha_T(a,b)$ also exists.
\end{proof}

Fix $m\in\mathbb N^+$, points
$
 \lambda=(\lambda_1,\ldots,\lambda_m)\in[0,1)^m,
$
and positive integers $d_1,\ldots,d_m$. Set
$
 \mathcal E=\bigoplus_{i=1}^{m}\mathbb C^{d_i}.
$
For each $i$, let
$
 P_i:\mathcal E\to\mathbb C^{d_i}$ and
 $\iota_i:\mathbb C^{d_i}\to\mathcal E
$
denote the canonical projection and inclusion, respectively. Given
$S\in\Bcal(\mathcal E)$, define
\(
 S_{ij}=P_iS\iota_j
 \in\Bcal(\mathbb C^{d_j},\mathbb C^{d_i})\) for \(1\leq i,j\leq m\).
Thus
$
 S=[S_{ij}]_{i,j=1}^{m}.
$

We use the conventions $0^0=1$ and $0^t=0$ for $t>0$. For
$z\in[0,1)$, define
$
 \Phi_T(z)
 =
 (1-z)\sum_{t\in T}z^t
$
whenever the series converges. Define
\begin{equation}
 D_\lambda(S)
 =
 \bigl[(1-\lambda_i\lambda_j)S_{ij}\bigr]_{i,j=1}^{m}.
 \label{eq:D-lambda}
\end{equation}
Whenever
$
 \sum_{t\in T}(\lambda_i\lambda_j)^t<\infty$ for $1\leq i,j\leq m,$
define
\begin{equation}
 M_T^\lambda(S)
 =
 \bigl[\Phi_T(\lambda_i\lambda_j)S_{ij}\bigr]_{i,j=1}^{m}.
 \label{eq:M-lambda}
\end{equation}
Convergence of all the series above is included in the preservation
property introduced below.

The set
$
 \mathcal P_\lambda
 =
 \bigl\{
 S\in\Bcal(\mathcal E):
 S\geq0,\ D_\lambda(S)\geq0
 \bigr\}
$
is called the radial Stein--Pick cone associated with $\lambda$ and the
chosen multiplicities $d_1,\ldots,d_m$.
The terminology reflects the Stein positivity condition
\(D_\lambda(S)\geq0\) and the radial Pick kernel
\((1-\lambda_i\lambda_j)^{-1}\) underlying the corresponding finite
spectral model.
\begin{definition}
\label{def:complete-preservation}
We say that $T$ has the \emph{uniform bound preservation property on
radial Stein--Pick cones} if, for every $0<a\leq b<\infty$, there exist
constants $0<\alpha\leq\beta<\infty$, depending only on $T$, $a$, and $b$,
such that, for every $m\in\mathbb N^+$, every choice of positive integers
$d_1,\ldots,d_m$, every
$
 \lambda=(\lambda_1,\ldots,\lambda_m)\in[0,1)^m,
$
and every
$
 S\in\Bcal(\mathcal E),
 \
 \mathcal E=\bigoplus_{i=1}^{m}\mathbb C^{d_i},
$
we have
\begin{equation}
\label{eq:complete-preservation}
\begin{aligned}
a\Id\leq S\leq b\Id,
\quad
D_\lambda(S)\geq0
\quad\Longrightarrow\quad &
 \sum_{t\in T}(\lambda_i\lambda_j)^t<\infty
 \quad\text{for all }1\leq i,j\leq m,\\
 &\alpha\Id\leq M_T^\lambda(S)\leq\beta\Id.
\end{aligned}
\end{equation}
\end{definition}

In Definition~\ref{def:complete-preservation}, the constants must be
uniform over the number of spectral points. Arbitrary spectral
multiplicities introduce no additional condition. Indeed, for
\(A=\diag(\lambda_1\Id_{d_1},\ldots,\lambda_m\Id_{d_m})\), let
\(E_i\) be the coordinate orthogonal projections. For each unit vector
\(x\), the space
\(
K_x=\bigoplus_{i=1}^{m}\operatorname{span}\{E_i x\},
\)
with zero summands omitted, reduces \(A\) and has dimension at most
\(m\). Compressing \(S\) and its Stein defect to \(K_x\) preserves
admissibility and the sampled Gramian quadratic form at \(x\).
Thus it suffices to test matrices of arbitrary size with one-dimensional
spectral blocks, using the same constants; see
\eqref{eq:fixed-spectral-compression}.
Theorem~\ref{thm:fixed-spectral-count} and
Corollary~\ref{cor:no-fixed-spectral-test} show that bounds for each
fixed number of spectral points need not yield a common positive lower
bound. Lemma~\ref{lem:finite-spectral-identity} identifies the cones
with finite spectral models having a prescribed integer observability
Gramian.

\begin{lemma}
\label{lem:finite-spectral-identity}
Let
$
 A=\diag(\lambda_1\Id_{d_1},\ldots,\lambda_m\Id_{d_m}),
 \
 B=D_\lambda(S).
$
If $B\geq0$, then
$
 \sum_{n\geq0}A^nBA^n=S.
$
Moreover, whenever $M_T^\lambda(S)$ is defined,
$
 \sum_{t\in T}A^tBA^t=M_T^\lambda(S).
$
Both series converge in operator norm.
\end{lemma}

\begin{proof}
For every $1\leq i,j\leq m$,
\(
 \left(\sum_{n\geq0}A^nBA^n\right)_{ij}
 =
 (1-\lambda_i\lambda_j)S_{ij}
 \sum_{n\geq0}(\lambda_i\lambda_j)^n
 =
 S_{ij}.
\)
Whenever $M_T^\lambda(S)$ is defined, we also have
\(
 \left(\sum_{t\in T}A^tBA^t\right)_{ij}
 =
 (1-\lambda_i\lambda_j)S_{ij}
 \sum_{t\in T}(\lambda_i\lambda_j)^t
 =
 \Phi_T(\lambda_i\lambda_j)S_{ij}.
\)
The two identities follow by comparing all the blocks.
\end{proof}

The passage from finite spectra to arbitrary spectra uses the following
continuity statement.

\begin{lemma}
\label{lem:strict-continuity}
Let $0<r<r_0<1$, and let
$
 A_j,A,B_j,B\in\Bcal(H).
$
Suppose that $A_j$ and $A$ are positive,
$
 \lVert A_j\rVert,\lVert A\rVert\leq r,
 \
 A_j\rightarrow A,
 \
 B_j\rightarrow B
$
in operator norm. If
$
 \sum_{t\in T}r_0^{2t}<\infty,
$
then
$
 \sum_{t\in T}A_j^tB_jA_j^t
 \rightarrow
 \sum_{t\in T}A^tBA^t
$
in operator norm.
\end{lemma}

\begin{proof}
Continuous functional calculus gives
$
 A_j^t\rightarrow A^t
$
in operator norm for every fixed $t\in T$. Since $B_j\to B$ in operator
norm,
$
 M:=\sup_{j\geq1}\lVert B_j\rVert+\lVert B\rVert<\infty.
$
For every $j$ and every $t\in T$,
$
 \lVert A_j^tB_jA_j^t\rVert
 \leq M r^{2t}
 \leq M r_0^{2t},
$
and the same estimate holds for $A^tBA^t$. Thus both series converge in
operator norm, with tails bounded uniformly in $j$. Their finite partial
sums converge term by term, and the asserted convergence follows.
\end{proof}

\begin{theorem}
\label{thm:complete-characterization}
A locally finite set $T\subset\Rplus$ belongs to $\U$ if and only if it
has the uniform bound preservation property on radial Stein--Pick cones.
\end{theorem}

\begin{proof}
Suppose first that $T\in\U$. Fix $0<a\leq b<\infty$ and consider finite
data as in Definition~\ref{def:complete-preservation}. Let
\(
 A=\diag(\lambda_1\Id_{d_1},\ldots,\lambda_m\Id_{d_m}),
 \
 B=D_\lambda(S),
\)
and take $C=B^{1/2}$. By
Lemma~\ref{lem:finite-spectral-identity}, the $\N$-sampled observability
Gramian is
$
 G_{\N}=S.
$
Hence Proposition~\ref{prop:uniformization} shows that the $T$-sampled
observability Gramian exists and satisfies uniform upper and lower bounds.

Applying the same argument to the scalar case gives
$\sum_{t\in T}r^{2t}<\infty$ for $0\leq r<1$.
Taking $r=\sqrt{\lambda_i\lambda_j}$ shows that
$\sum_{t\in T}(\lambda_i\lambda_j)^t<\infty$ for all $i,j$.
Thus $M_T^\lambda(S)$ is well defined, and
Lemma~\ref{lem:finite-spectral-identity} gives
$G_T=M_T^\lambda(S)$.
Proposition~\ref{prop:uniformization} therefore yields the uniform bound
preservation property.

Conversely, suppose that $T$ has the uniform bound preservation property,
and fix $0<a\leq b<\infty$. Let $\alpha,\beta$ be the corresponding
constants in \eqref{eq:complete-preservation}. Consider an observation
system $(A,C)$ with $A\geq0$ such that
$
 a\Id\leq G_{\N}\leq b\Id.
$
Set
$
 B=C^*C=G_{\N}-AG_{\N}A\geq0.
$

We first establish the upper bound. Suppose that some finite partial sum is not bounded above by $\beta\Id$. Then there exist a finite set
$F\subset T$, a unit vector $x$, and $\varepsilon>0$ such that
$
 \sum_{t\in F}\langle A^tBA^t x,x\rangle
 >\beta+2\varepsilon.
$
For $0<c<1$ sufficiently close to $1$, set
$
 A_c=cA,
 \
 B_c=G_{\N}-A_cG_{\N}A_c.
$
Since
\[
 B_c=(1-c^2)G_{\N}+c^2B
 \geq a(1-c^2)\Id>0,
\]
the Stein defect remains positive under sufficiently small norm
perturbations of $A_c$. Moreover,
$\sum_{n\geq0}A_c^nB_cA_c^n=G_{\N}$.
By continuity,
$
 \sum_{t\in F}\langle A_c^tB_cA_c^t x,x\rangle
 >\beta+\varepsilon.
$
Approximate $A_c$ in norm by a positive operator $A_0$ with finite spectrum so
closely that
$
 \lVert G_{\N}^{1/2}A_0G_{\N}^{-1/2}\rVert<1,
 \
 B_0:=G_{\N}-A_0G_{\N}A_0\geq0,
$
and the preceding quadratic form remains strictly larger than $\beta$.

Write
\(
A_0=\sum_j\lambda_jP_j
\)
and let
\(
K=\bigoplus_j\operatorname{span}\{P_jx\},
\)
where zero summands are omitted. Then \(K\) is finite dimensional,
\(x\in K\), and \(K\) reduces \(A_0\). Set
\(
A_K=A_0|_K
\)
and
\(
G_{\N,K}=P_KG_{\N}|_K.
\)
We have
\[
G_{\N,K}-A_KG_{\N,K}A_K
=
P_K\bigl(G_{\N}-A_0G_{\N}A_0\bigr)|_K
=
P_KB_0|_K
\geq0.
\]
Compressing the bounds \(a\Id\leq G_{\N}\leq b\Id\) to \(K\) gives
\(
a\Id_K\leq G_{\N,K}\leq b\Id_K.
\)
Moreover, for every \(t\in T\),
\[
\left\langle
A_K^t
\bigl(G_{\N,K}-A_KG_{\N,K}A_K\bigr)
A_K^t x,x
\right\rangle
=
\langle A_0^tB_0A_0^t x,x\rangle.
\]
Hence the quadratic form of the finite sum over \(F\) is unchanged.
We therefore obtain an admissible finite spectral model for which that
quadratic form exceeds \(\beta\), contradicting
\eqref{eq:complete-preservation}.
Since these sums form an increasing family of positive operators, they
converge strongly to an operator $G_T$ satisfying
$
 0\leq G_T\leq\beta\Id.
$

It remains to prove the lower bound. Suppose that
$G_T\not\geq\alpha\Id$. Then there exist a unit vector $x$ and
$\varepsilon>0$ such that
$
 \langle G_Tx,x\rangle<\alpha-3\varepsilon.
$
By Lemma~\ref{lem:auto-contraction}, $E_A(\{1\})=0$, and hence
$E_A([0,r])\to\Id$ strongly as $r\uparrow1$. Therefore, for some $r<1$, the normalized vector
$
 y=\frac{E_A([0,r])x}{\lVert E_A([0,r])x\rVert}
$
satisfies
$
 \langle G_Ty,y\rangle<\alpha-2\varepsilon.
$
Compress $A$, $G_{\N}$, and $B$ to $E_A([0,r])H$. Since this spectral
subspace reduces $A$, the compressed operators still satisfy
$
 a\Id\leq G_{\N}\leq b\Id,
 \
 B=G_{\N}-AG_{\N}A\geq0,
$
while now $\lVert A\rVert\leq r<1$. We retain the same notation for the compressed operators.

The scalar case of the uniform bound preservation property implies
$
 \sum_{t\in T}r_0^{2t}<\infty$ for $0<r_0<1$.
Choose $r<\rho<\rho_0<1$. For $0<c<1$ sufficiently close to $1$, set
$A_c=cA$ and $B_c=G_{\N}-A_cG_{\N}A_c$. As above,
$B_c\geq a(1-c^2)\Id>0$; moreover,
$A_c\to A$ and $B_c\to B$ as $c\uparrow1$.
By Lemma~\ref{lem:strict-continuity},
$\sum_{t\in T}A_c^tB_cA_c^t\to\sum_{t\in T}A^tBA^t$
in operator norm.

Approximate $A_c$ in norm by a positive operator $A_0$ with finite spectrum so
closely that $\lVert A_0\rVert\leq\rho$ and
$B_0:=G_{\N}-A_0G_{\N}A_0\geq0$.
Since $B_0\to B_c$ as $A_0\to A_c$, another application of
Lemma~\ref{lem:strict-continuity} shows that the corresponding series
changes by an arbitrarily small amount. Hence $c$ and $A_0$ may be chosen
so that its quadratic form at $y$ is strictly smaller than
$\alpha-\varepsilon$.

Finally, write
\(
A_0=\sum_j\lambda_jP_j
\)
and let
\(
K=\bigoplus_j\operatorname{span}\{P_jy\},
\)
again omitting zero summands. Then \(K\) is finite dimensional,
\(y\in K\), and \(K\) reduces \(A_0\). With
\(
A_K=A_0|_K
\)
and
\(
G_{\N,K}=P_KG_{\N}|_K,
\)
we have
\(
a\Id_K\leq G_{\N,K}\leq b\Id_K,
G_{\N,K}-A_KG_{\N,K}A_K
=
P_KB_0|_K
\geq0.
\)
Moreover, for every \(t\in T\),
\(
\left\langle
A_K^t
\bigl(G_{\N,K}-A_KG_{\N,K}A_K\bigr)
A_K^t y,y
\right\rangle
=
\langle A_0^tB_0A_0^t y,y\rangle.
\)
Thus the quadratic form of the sampled Gramian at \(y\) is unchanged.
This yields an admissible finite spectral model whose sampled Gramian
has a quadratic form strictly smaller than \(\alpha\), again
contradicting \eqref{eq:complete-preservation}.
Thus
$
 \alpha\Id\leq G_T\leq\beta\Id.
$
Hence the $T$-system is exactly observable, and therefore
$T\in\U$.
\end{proof}

\section{Scalar bounds, commuting systems, and the noncommutative obstruction}
We now extract the scalar consequences of the uniform Stein--Pick
criterion and determine how far they characterize universal sampling.
The scalar bounds are equivalent to \(0\in T\) together with
two-sided linear counting. These conditions characterize the
commuting case. We construct a counterexample for general positive
systems, and then show that these conditions still give uniform
bounds for every fixed number of spectral points. The counterexample
exhibits the loss of a common positive lower bound as that number
increases.

\subsection{Scalar bounds and commuting systems}

For \(R\geq0\) and \(0<r<1\), set
\(
N_T(R)=\#(T\cap[0,R]),
F_T(r)=(1-r^2)\sum_{t\in T}r^{2t}.
\)
Here \(F_T(r)\) is allowed to be \(+\infty\).

\begin{theorem}
\label{thm:abel-counting}
If $T\in\U$, then $0\in T$ and there exist constants
$0<c_T\leq C_T<\infty$ such that
\(
 c_T\leq F_T(r)\leq C_T\) for \(0<r<1\).
Moreover, for a locally finite set $T\subset\Rplus$ containing $0$,
the existence of such two-sided bounds is equivalent to the existence
of constants $0<c_T'\leq C_T'<\infty$ and $R_0>0$ such that
\(
 c_T'R\leq N_T(R)\leq C_T'R\) for \(R\geq R_0\).
\end{theorem}

\begin{proof}
Suppose first that $T\in\U$. If $0\notin T$, take
$A=0$ and $C=\Id$. Then
\(
 \sum_{n\geq0}\lVert CA^n x\rVert^2=\lVert x\rVert^2,
\)
whereas $CA^t=0$ for every $t\in T$. Thus the $T$-system is not
exactly observable, a contradiction. Hence $0\in T$.

To prove the two-sided Abel bounds, consider
\(
H=L^2(0,1),
\quad
(Af)(r)=rf(r),
\quad
(Cf)(r)=\sqrt{1-r^2}\,f(r).
\)
Since
\(
\sum_{n\geq0}(1-r^2)r^{2n}=1
\)
for \(0<r<1\), Tonelli's theorem gives
\(
\sum_{n\geq0}\lVert CA^nf\rVert^2=\lVert f\rVert^2.
\)
Thus the \(\N\)-system is Parseval. Because \(T\in\U\), the
\(T\)-system is exactly observable. Hence there are constants
$0<c_T\leq C_T<\infty$ such that
\[
 c_T\lVert f\rVert^2
 \leq
 \int_0^1 F_T(r)|f(r)|^2\dd r
 \leq
 C_T\lVert f\rVert^2,
 \qquad f\in L^2(0,1).
\]
It follows that
\(
c_T\leq F_T(r)\leq C_T
\)
for almost every $r\in(0,1)$.

We claim that $F_T$ is continuous on $(0,1)$. Fix $0<\rho<1$.
Since $F_T$ is finite almost everywhere, there exists
$s\in(\rho,1)$ such that $F_T(s)<\infty$. Hence
\(
\sum_{t\in T}s^{2t}<\infty.
\)
Therefore the series
\(
\sum_{t\in T}r^{2t}
\)
converges uniformly for $0\leq r\leq\rho$, by the Weierstrass
$M$-test. It follows that $F_T$ is continuous on $(0,\rho)$.
Since $\rho<1$ was arbitrary, $F_T$ is continuous on $(0,1)$.
Consequently, the almost-everywhere bounds hold for every
$0<r<1$, proving the asserted Abel bounds.

We now prove the equivalence with linear counting. Put
$r=\e^{-u/2}$ and
$L_T(u)=\sum_{t\in T}\e^{-ut}$ for $u>0$. Then
\(
 F_T(\e^{-u/2})=(1-\e^{-u})L_T(u).
\)
Since
$(1-\e^{-1})u\leq1-\e^{-u}\leq u$ for $0<u\leq1$,
the two-sided Abel bounds imply that there exist
$0<a\leq b<\infty$ such that
\(
 \frac{a}{u}\leq L_T(u)\leq\frac{b}{u}\) for \(0<u\leq1\).

For later use, Tonelli's theorem gives the integral representation
\[
 L_T(u)=u\int_0^\infty \e^{-us}N_T(s)\dd s,
\]
because
\[
 u\int_0^\infty \e^{-us}N_T(s)\dd s
 =
 \sum_{t\in T}u\int_t^\infty \e^{-us}\dd s
 =
 \sum_{t\in T}\e^{-ut}.
\]

Assume first that the two-sided Abel bounds hold. For $R\geq1$,
\[
 \e^{-1}N_T(R)
 \leq
 \sum_{\substack{t\in T\\t\leq R}}\e^{-t/R}
 \leq
 L_T(1/R)
 \leq bR,
\]
so for \( R\geq1\), we have
\(
 N_T(R)\leq \e bR\).
Set $K=\e b$ and choose $q>1$ so large that
$K(q+1)\e^{-q}\leq a/2$.
Using the integral representation of $L_T$ and the preceding
counting bound, we obtain, for $R\geq1$,
\begin{align*}
 L_T(1/R)
 &=
 \frac1R\int_0^{qR}\e^{-s/R}N_T(s)\dd s
 +
 \frac1R\int_{qR}^\infty\e^{-s/R}N_T(s)\dd s
 \leq
 N_T(qR)
 +
 \frac K R\int_{qR}^\infty s\e^{-s/R}\dd s\\
 &=
 N_T(qR)+KR(q+1)\e^{-q}
 \leq
 N_T(qR)+\frac a2R.
\end{align*}
On the other hand, the lower estimate for $L_T$ gives
$L_T(1/R)\geq aR$.
Hence $N_T(qR)\geq aR/2$, and therefore
$N_T(R)\geq aR/(2q)$ for all sufficiently large $R$.
Together with the upper counting bound, this proves the asserted
two-sided linear counting condition.

Conversely, assume the two-sided linear counting condition in the
statement. Enlarging $R_0$ if necessary, suppose that $R_0\geq1$.
By local finiteness and the upper counting bound, there is
$K<\infty$ such that
\(
 N_T(s)\leq K(1+s)\) for \(s\geq0\).
The integral representation of $L_T$ then gives
\[
 L_T(u)
 \leq
 K u\int_0^\infty \e^{-us}(1+s)\dd s
 =
 K\left(1+\frac1u\right).
\]
If $0<u\leq R_0^{-1}$, then
\[
 L_T(u)
 \geq
 c_T'u\int_{R_0}^\infty s\e^{-us}\dd s
 =
 c_T'\e^{-uR_0}\left(R_0+\frac1u\right)
 \geq
 \frac{c_T'\e^{-1}}{u},
\]
whereas the upper estimate gives $L_T(u)\leq2K/u$.
Hence the relation between $F_T$ and $L_T$, together with
$(1-\e^{-1})u\leq1-\e^{-u}\leq u$, shows that
$F_T(\e^{-u/2})$ is bounded above and below by positive constants
for $0<u\leq R_0^{-1}$.

It remains to consider $u\geq R_0^{-1}$. Since $0\in T$,
$L_T(u)\geq1$, and hence
\[
 1-\e^{-1/R_0}
 \leq
 F_T(\e^{-u/2})
 \leq
 L_T(u)
 \leq
 K(1+R_0).
\]
Thus the two-sided Abel bounds hold for every $0<r<1$.
\end{proof}

Thus \(0\in T\) and \(N_T(R)\asymp R\) are already forced by the
scalar tests in the uniform Stein--Pick criterion, but they are not
sufficient in general. Their relation to uniform bounds at a fixed number
of spectral points is established after the counterexample below.

Here and below, ``commuting'' means $[A,C^*C]=0$, since commutation of
$A$ and $C$ is not defined in general when $C:H\to Y$.

Let \((A,C)\) be exactly observable at integer times, set
\(B=C^*C\), and assume \([A,B]=0\).
By Lemma~\ref{lem:auto-contraction}, \(0\leq A\leq\Id\) and
\(E_A(\{1\})=0\). In a spectral direct integral representation,
\[
H=\int_{[0,1)}^\oplus H_\lambda\dd\mu(\lambda),
\qquad
A=M_\lambda,
\qquad
B=M_{W(\lambda)},
\]
where \(W(\lambda)\geq0\) almost everywhere. Set
$S_T(\lambda)=\sum_{t\in T}\lambda^{2t}$ and
$F_T(\lambda)=(1-\lambda^2)S_T(\lambda)$.

\begin{definition}
Let $\Ucom$ denote the class of locally finite sets $T\subset\Rplus$
such that, for every positive operator $A$ and every observation operator
$C$ satisfying $[A,C^*C]=0$, exact observability at integer times implies
exact observability at the times in $T$.
\end{definition}

\begin{theorem}
\label{thm:commuting-universal}
Let $T\subset\Rplus$ be locally finite. Then
$
 T\in\Ucom$ if and only if
 $
 0\in T$
 and
 $N_T(R)\asymp R$
 as $R\to\infty.
$
\end{theorem}

\begin{proof}
Suppose first that $T\in\Ucom$. By the proof of
Theorem~\ref{thm:abel-counting}, we have
$0\in T$ and $N_T(R)\asymp R$, since the test systems used there
all satisfy the commuting condition $[A,C^*C]=0$.

Conversely, suppose that $0\in T$ and $N_T(R)\asymp R$. By
Theorem~\ref{thm:abel-counting}, there exist constants
$0<c_T\leq C_T<\infty$ such that
$c_T\leq F_T(r)\leq C_T$ for $0<r<1$.
Since $0\in T$ and we use the convention $0^0=1$, we also have
$F_T(0)=1$. Replacing $c_T$ by $\min\{c_T,1\}$ and $C_T$ by
$\max\{C_T,1\}$, for \(0\leq r<1\) we obtain
\(
c_T\leq F_T(r)\leq C_T\).

Let $(A,C)$ be any system that is exactly observable at integer times
and satisfies $[A,C^*C]=0$, and set $B=C^*C$. Let $a,b$ be observability bounds for the $\N$-system. By Lemma~\ref{lem:auto-contraction},
$0\leq A\leq\Id$ and $E_A(\{1\})=0$. Since $A$ and $B$ commute, we may
use a spectral direct integral representation
\[
 H=\int_{[0,1)}^\oplus H_\lambda\dd\mu(\lambda),
 \qquad
 A=M_\lambda,
 \qquad
 B=M_{W(\lambda)},
\]
where $W(\lambda)\geq0$ almost everywhere.

The $\N$-sampled observability Gramian is therefore multiplication by
$
 \sum_{n\geq0}\lambda^{2n}W(\lambda)
 =
 \frac{W(\lambda)}{1-\lambda^2}.
$
Hence the integer observability bounds are equivalent to
$
 a\Id_{H_\lambda}
 \leq
 \frac{W(\lambda)}{1-\lambda^2}
 \leq
 b\Id_{H_\lambda}
$
for almost every $\lambda$.

On the other hand, the $T$-sampled observability Gramian is multiplication
by \(S_T(\lambda)W(\lambda)\), where
\(
S_T(\lambda)=\sum_{t\in T}\lambda^{2t}.
\)
Since
$F_T(\lambda)=(1-\lambda^2)S_T(\lambda)$, we have
\[
 S_T(\lambda)W(\lambda)
 =
 F_T(\lambda)
 \frac{W(\lambda)}{1-\lambda^2}.
\]
Combining the preceding fiberwise bounds with
$c_T\leq F_T(\lambda)\leq C_T$ gives
$
 c_Ta\Id_{H_\lambda}
 \leq
 S_T(\lambda)W(\lambda)
 \leq
 C_Tb\Id_{H_\lambda}
$
almost everywhere. Consequently,
$
 c_Ta\Id\leq G_T\leq C_Tb\Id,
$
so the $T$-system is exactly observable. Thus $T\in\Ucom$.

Finally, $\U\subseteq\Ucom$ follows directly from the definitions, since
the commuting systems form a subclass of the systems with $A\geq0$ for
which the $\N$-system is exactly observable.
\end{proof}

\subsection{A clustered counterexample}

We next construct a time set \(T\) with \(0\in T\) and
\(N_T(R)\asymp R\) whose sampled lower bounds tend to zero along a
sequence of finite spectral models with fixed integer observability
bounds. Thus the scalar Abel bounds do not provide a common lower
bound over all such models, and
\(
\U\subsetneq\Ucom.
\)
The construction uses a self-similar scaling limit and cancellation
across geometrically separated spectral scales.

\begin{lemma}
\label{lem:tangent-measure}
Fix $q=16$. For $c>0$, define
\(
H_q(c):=\sum_{n\in\mathbb Z}(-1)^n q^{n/2}\e^{-cq^n}.
\)
Then the following assertions hold.
\begin{enumerate}[
 label=\textup{(\roman*)},
 leftmargin=*,
 topsep=2pt,
 itemsep=2pt,
 parsep=0pt
]
\item
The defining series converges absolutely and locally uniformly on
$(0,\infty)$, satisfies $H_q(qc)=-q^{-1/2}H_q(c)$, and has a zero
$c_0\in(1,q)$.

\item
For any such zero $c_0$, define
\begin{equation}\label{eq:clustered-T}
 B_m
 =
 \left\{
 c_0q^m+\ell q^{-2m}:1\leq\ell\leq q^m
 \right\},
 \qquad
 T=\{0\}\cup\bigcup_{m\geq0}B_m .
\end{equation}
Then $T$ is locally finite and $N_T(R)\asymp R$ as $R\to\infty$.
\item
For \(L\in\N\), set \(R_L=q^L\) and
\(\mu_L=R_L^{-1}\sum_{t\in T}\delta_{t/R_L}\).
Then
\begin{equation}\label{eq:tangent-nu}
 \mu_L
 \xrightarrow[L\to\infty]{\,v\,}
 \nu
 :=
 \sum_{n\in\mathbb Z}q^n\delta_{c_0q^n}
 \qquad\text{on }(0,\infty).
\end{equation}
Moreover, for every $a>0$,
\begin{equation}\label{eq:laplace-tangent}
 \frac1{R_L}\sum_{t\in T}\e^{-at/R_L}
 \longrightarrow
 \sum_{n\in\mathbb Z}q^n\e^{-ac_0q^n}.
\end{equation}
\end{enumerate}
\end{lemma}

\begin{proof}
For \textup{(i)}, fix $0<a<b<\infty$. For $c\in[a,b]$,
\(
 \sum_{n\geq0}q^{n/2}\e^{-cq^n}
 \leq
 \sum_{n\geq0}q^{n/2}\e^{-aq^n}
 <\infty,
 \
 \sum_{n<0}q^{n/2}\e^{-cq^n}
 \leq
 \sum_{n<0}q^{n/2}
 <\infty.
\)
Hence the defining series converges absolutely and uniformly on $[a,b]$,
and therefore locally uniformly on $(0,\infty)$.

Reindexing gives
\[
 H_q(qc)
 =
 \sum_{n\in\mathbb Z}(-1)^nq^{n/2}\e^{-cq^{n+1}}
 =
 \sum_{m\in\mathbb Z}(-1)^{m-1}q^{(m-1)/2}\e^{-cq^m}
 =
 -q^{-1/2}H_q(c).
\]

At $c=1$, put
$u_k=q^{-k/2}\e^{-q^{-k}}$ and
$v_k=q^{k/2}\e^{-q^k}$ for $k\geq1$. Since $q=16$,
\(
 \frac{u_{k+1}}{u_k}
 =
 q^{-1/2}
 e^{((1-q^{-1})q^{-k})}
 <1,
 \
 \frac{v_{k+1}}{v_k}
 =
 q^{1/2}
 e^{(-(q-1)q^k)}
 <1.
\)
Thus $(u_k+v_k)_{k\geq1}$ is decreasing, and the alternating series
estimate yields
\[
 H_q(1)
 =
 \e^{-1}
 +\sum_{k\geq1}(-1)^k(u_k+v_k)
 \geq
 \e^{-1}-u_1-v_1
 =
 \e^{-1}
 -q^{-1/2}\e^{-1/q}
 -q^{1/2}\e^{-q}
 >0.
\]
Consequently $H_q(q)=-q^{-1/2}H_q(1)<0$. By continuity, there exists
$c_0\in(1,q)$ such that $H_q(c_0)=0$.

For \textup{(ii)}, we have $\#B_m=q^m$ and
$B_m\subset[c_0q^m,c_0q^m+q^{-m}]$. Moreover,
\(
 c_0q^{m+1}
 -
 \bigl(c_0q^m+q^{-m}\bigr)
 =
 c_0(q-1)q^m-q^{-m}
 >0,
\)
so the clusters are pairwise disjoint and ordered. Since
$c_0q^m\to\infty$, the set $T$ is locally finite.

Let $b_m:=\min B_m=c_0q^m+q^{-2m}$. For sufficiently large $R$, choose
$M$ such that $b_M\leq R<b_{M+1}$. Then
\(
 1+\sum_{m=0}^{M-1}q^m
 \leq
 N_T(R)
 \leq
 1+\sum_{m=0}^{M}q^m,
\)
while $c_0q^M\leq R<c_0q^{M+1}+q^{-2M-2}$. Hence
$N_T(R)\asymp q^M\asymp R$.

For \textup{(iii)}, let $\varphi\in C_c(0,\infty)$ and choose
$0<\alpha<\beta<\infty$ such that
$\supp\varphi\subset[\alpha,\beta]$. Writing $m=L+n$, the contribution
of $B_{L+n}$ is
\(
 I_{L,n}
 :=
 \frac1{q^L}
 \sum_{\ell=1}^{q^{L+n}}
 \varphi\!\left(
 c_0q^n+\ell q^{-3L-2n}
 \right).
\)
For each fixed $n$,
$0\leq\ell q^{-3L-2n}\leq q^{-2L-n}\to0$ uniformly for
$1\leq\ell\leq q^{L+n}$, while $q^{L+n}/q^L=q^n$. Therefore
$I_{L,n}\to q^n\varphi(c_0q^n)$.

Since $n\geq-L$, we have $q^{-2L-n}\leq q^{-L}$. For all sufficiently
large $L$, the condition $I_{L,n}\neq0$ therefore implies
$\alpha/2\leq c_0q^n\leq\beta$. Thus only finitely many values of $n$
occur, independently of $L$. Hence
\(
 \int\varphi\,\dd\mu_L
 \longrightarrow
 \sum_{n\in\mathbb Z}q^n\varphi(c_0q^n)
 =
 \int\varphi\,\dd\nu,
\)
which proves \eqref{eq:tangent-nu}.

Finally,
\(
 \frac1{q^L}\sum_{t\in T}\e^{-at/q^L}
 =
 q^{-L}
 +
 \sum_{n\geq-L}J_{L,n},
\)
where
\(
 J_{L,n}
 :=
 \frac1{q^L}
 \sum_{\ell=1}^{q^{L+n}}
 e^{(
 -a(c_0q^n+\ell q^{-3L-2n})
 )}.
\)
For each fixed $n$, $J_{L,n}\to q^n\e^{-ac_0q^n}$. Moreover,
\(
 \sum_{\substack{n<-K\\ n\geq-L}}J_{L,n}
 \leq
 \sum_{n<-K}q^n,
 \
 \sum_{n>K}J_{L,n}
 \leq
 \sum_{n>K}q^n\e^{-ac_0q^n}.
\)
Thus
\(
 \lim_{K\to\infty}\sup_L
 \sum_{\substack{n<-K\\ n\geq-L}}J_{L,n}
 =0,
 \
 \lim_{K\to\infty}\sup_L
 \sum_{n>K}J_{L,n}
 =0.
\)
Since $q^{-L}\to0$, truncation in $n$ gives
\(
 \frac1{R_L}\sum_{t\in T}\e^{-at/R_L}
 \longrightarrow
 \sum_{n\in\mathbb Z}q^n\e^{-ac_0q^n},
\)
which proves \eqref{eq:laplace-tangent}.
\end{proof}

For $i\in\mathbb Z$, set
$a_i=q^{-i}$ and
$\varphi_i(s)=\sqrt{2a_i}\,\e^{-a_i s}$ for $s>0$.

\begin{lemma}
\label{lem:continuous-riesz}
The following assertions hold.
\begin{enumerate}[
 label=\textup{(\roman*)},
 leftmargin=*,
 topsep=2pt,
 itemsep=2pt,
 parsep=0pt
]
\item
There exist constants \(0<m_q\leq M_q<\infty\), depending only on \(q\),
such that every finitely supported scalar sequence \((x_i)\) satisfies
\[
m_q\sum_i|x_i|^2
\leq
\int_0^\infty
\left|\sum_i x_i\varphi_i(s)\right|^2\dd s
\leq
M_q\sum_i|x_i|^2.
\]

\item
Sampling this family against the limit measure \(\nu\) in
\eqref{eq:tangent-nu} defines, initially for finitely supported sequences
\(x=(x_i)_{i\in\mathbb Z}\),
\begin{equation}\label{eq:M-operator}
(\mathcal Mx)_n
=
q^{n/2}\sum_i x_i\varphi_i(c_0q^n)
=
\sqrt{2}\sum_i
x_iq^{(n-i)/2}\e^{-c_0q^{\,n-i}}.
\end{equation}
The operator \(\mathcal M\) extends boundedly to
\(\ell^2(\mathbb Z)\). Moreover, if
\[
x_i^{(M)}
=
\begin{cases}
\dfrac{(-1)^i}{\sqrt{2M+1}}, & |i|\leq M,\\[1ex]
0, & |i|>M,
\end{cases}
\qquad i\in\mathbb Z,
\]
then
\(
\|\mathcal Mx^{(M)}\|_{\ell^2(\mathbb Z)}
\rightarrow0\) for \(M\to\infty\).

\end{enumerate}
\end{lemma}

\begin{proof}
A direct computation gives
\(
\langle\varphi_j,\varphi_i\rangle
=
\frac{2\sqrt{a_i a_j}}{a_i+a_j}
=
\sech\!\left(\frac{(i-j)\log q}{2}\right).
\)
Hence the Gram operator is convolution on $\ell^2(\mathbb Z)$ by
$g_k=\sech(k\log q/2)$, where $g\in\ell^1(\mathbb Z)$.

Put $\alpha=(\log q)/2$. Using the Fourier transform convention
\(
\widehat f(\xi)
=
\int_{\mathbb R}f(x)\e^{-\mathrm i\xi x}\dd x,
\)
we have
\(
\int_{\mathbb R}\sech(\alpha x)\e^{-\mathrm i\xi x}\dd x
=
\frac{\pi}{\alpha}
\sech\!\left(\frac{\pi\xi}{2\alpha}\right).
\)
Poisson summation therefore gives the Fourier symbol
\[
G_q(\theta)
=
\sum_{k\in\mathbb Z}g_k\e^{-\mathrm i k\theta}
=
\frac{\pi}{\alpha}
\sum_{\ell\in\mathbb Z}
\sech\!\left(
\frac{\pi(\theta+2\pi\ell)}{2\alpha}
\right).
\]
The series on the right converges uniformly for
\(\theta\in[-\pi,\pi]\), so \(G_q\) is continuous and
\(2\pi\)-periodic. Since every summand is positive,
\(G_q(\theta)>0\) for all \(\theta\). Thus
\(
0<m_q:=\min_{\theta}G_q(\theta)
\leq
\max_{\theta}G_q(\theta)=:M_q
<\infty,
\)
and Parseval's identity gives the asserted upper and lower estimate.

By \eqref{eq:M-operator}, $\mathcal M$ is convolution by
\(
h_k=\sqrt{2}\,q^{k/2}\e^{-c_0q^k}.
\)
Since $h\in\ell^1(\mathbb Z)$, Young's inequality shows that
$\mathcal M$ extends boundedly to $\ell^2(\mathbb Z)$. Its Fourier symbol
satisfies
\[
\widehat h(\pi)
=
\sqrt{2}\sum_{k\in\mathbb Z}
(-1)^k q^{k/2}\e^{-c_0q^k}
=
\sqrt{2}\,H_q(c_0)
=
0.
\]

For $y\in\ell^1(\mathbb Z)$, write
\(
\widehat y(\theta)
=
\sum_{k\in\mathbb Z}y_k\e^{-\mathrm i k\theta}.
\)
Parseval's identity gives
\[
\|\mathcal Mx^{(M)}\|_2^2
=
\frac{1}{2\pi}
\int_{-\pi}^{\pi}
|\widehat h(\theta)|^2
|\widehat{x^{(M)}}(\theta)|^2
\dd\theta.
\]
Moreover,
\(
\frac{1}{2\pi}
|\widehat{x^{(M)}}(\theta)|^2
=
\frac{1}{2\pi(2M+1)}
\left|
\sum_{j=-M}^{M}
\e^{-\mathrm i j(\theta-\pi)}
\right|^2,
\)
which is the normalized Fej\'er kernel centered at $\pi$ on the circle.
Hence the measures
\(
\frac{1}{2\pi}
|\widehat{x^{(M)}}(\theta)|^2\dd\theta
\)
form an approximate identity at $\pi$. Since $\widehat h$ is continuous
and $\widehat h(\pi)=0$, it follows that
\(
\|\mathcal Mx^{(M)}\|_2\to0
\)
as \(M\to\infty\).
\end{proof}

\begin{theorem}
\label{thm:mellin-counterexample}
The set $T$ defined in \eqref{eq:clustered-T} satisfies
\(
0\in T,
\
N_T(R)\asymp R,
\
(1-r^2)\sum_{t\in T}r^{2t}\asymp1\) for \(0<r<1,\)
but $T\notin\U$.
\end{theorem}

\begin{proof}
Fix $M\in\mathbb N^+$. For $L\in\mathbb N^+$, let $H_M$ be a
$(2M+1)$-dimensional Hilbert space with orthonormal basis
$\{e_i:|i|\leq M\}$, and define
\[
A_{M,L}e_i=\e^{-a_i/q^L}e_i,
\qquad
C_{M,L}e_i=\sqrt{1-\e^{-2a_i/q^L}},
\qquad |i|\leq M,
\]
where $C_{M,L}:H_M\to\mathbb C$.

The integer observability Gramian $S_{M,L}$ has entries
\[
(S_{M,L})_{ij}
=
\frac{
\sqrt{1-\e^{-2a_i/q^L}}
\sqrt{1-\e^{-2a_j/q^L}}
}{
1-\e^{-(a_i+a_j)/q^L}
}.
\]
As $L\to\infty$,
\(
(S_{M,L})_{ij}
\rightarrow
\frac{2\sqrt{a_i a_j}}{a_i+a_j}.
\)
Since $H_M$ is finite dimensional, the entrywise convergence is convergence
in operator norm. The limit is the Gram matrix of
$\{\varphi_i:|i|\leq M\}$. Hence Lemma~\ref{lem:continuous-riesz} implies
that, for all sufficiently large $L$,
\(
\frac{m_q}{2}\Id
\leq
S_{M,L}
\leq
2M_q\Id.
\)

The \(T\)-sampled observability Gramian \(G_{M,L}\) has entries
\[
(G_{M,L})_{ij}
=
\sqrt{1-\e^{-2a_i/q^L}}
\sqrt{1-\e^{-2a_j/q^L}}
\sum_{t\in T}\e^{-(a_i+a_j)t/q^L}.
\]
By \eqref{eq:laplace-tangent},
\(
(G_{M,L})_{ij}
\rightarrow
2\sqrt{a_i a_j}
\sum_{n\in\mathbb Z}
q^n\e^{-c_0(a_i+a_j)q^n}
=:(K_M)_{ij}.
\)
Again, for fixed $M$, the convergence is in operator norm.

Identify $H_M$ with the coordinate subspace
$\operatorname{span}\{e_i:|i|\leq M\}$ of $\ell^2(\mathbb Z)$, and let
$P_M$ denote the corresponding orthogonal projection. By
\eqref{eq:M-operator},
\(
K_M=P_M\mathcal M^*\mathcal M P_M.
\)
Hence $\|K_M\|\leq\|\mathcal M\|^2$ for every $M$. Moreover,
Lemma~\ref{lem:continuous-riesz} gives
\(
\langle K_Mx^{(M)},x^{(M)}\rangle
=
\|\mathcal Mx^{(M)}\|_2^2
\rightarrow0.
\)

We may therefore choose a strictly increasing sequence $L_M\to\infty$
such that
\(
\frac{m_q}{2}\Id
\leq
S_{M,L_M}
\leq
2M_q\Id,
\
\|G_{M,L_M}-K_M\|
\leq
\frac1M.
\)
Consequently,
\(
\sup_{M\geq1}\|G_{M,L_M}\|
\leq
\|\mathcal M\|^2+1
<\infty.
\)
Since $\|x^{(M)}\|=1$,
\[
\left|
\langle G_{M,L_M}x^{(M)},x^{(M)}\rangle
-
\langle K_Mx^{(M)},x^{(M)}\rangle
\right|
\leq
\frac1M.
\]
It follows that
\(
\langle G_{M,L_M}x^{(M)},x^{(M)}\rangle
\rightarrow0.
\)

Now set
\(
H=\bigoplus_{M\geq1}H_M,
\
Y=\bigoplus_{M\geq1}\mathbb C,
\
A=\bigoplus_{M\geq1}A_{M,L_M},
\
C=\bigoplus_{M\geq1}C_{M,L_M}.
\)
Each $A_{M,L_M}$ is a positive contraction, and hence so is $A$.
Moreover,
\(
C_{M,L_M}^*C_{M,L_M}
\leq
S_{M,L_M}
\leq
2M_q\Id,
\)
so $C$ is bounded.

The integer Gramian is
$G_{\N}=\bigoplus_{M\geq1}S_{M,L_M}$ and therefore
\(
\frac{m_q}{2}\Id
\leq
G_{\N}
\leq
2M_q\Id.
\)
Thus the integer system is exactly observable.

For $x=(x_M)_{M\geq1}\in H$, Tonelli's theorem and the uniform bound on
the blocks $G_{M,L_M}$ give
\[
\sum_{t\in T}\|CA^t x\|^2
=
\sum_{M\geq1}
\langle G_{M,L_M}x_M,x_M\rangle
\leq
(\|\mathcal M\|^2+1)
\sum_{M\geq1}\|x_M\|^2.
\]
Hence the $T$-system is Bessel and
$G_T=\bigoplus_{M\geq1}G_{M,L_M}$.

Regard $x^{(M)}$ as a unit vector in the $M$th summand of $H$. Then
\(
\langle G_Tx^{(M)},x^{(M)}\rangle
=
\langle G_{M,L_M}x^{(M)},x^{(M)}\rangle
\rightarrow0.
\)
Thus $G_T$ is not bounded below, so the $T$-system fails the lower
observability bound. In particular, it is not exactly observable, and hence
$T\notin\U$.

Finally, Lemma~\ref{lem:tangent-measure} gives
$0\in T$ and $N_T(R)\asymp R$. Therefore
Theorem~\ref{thm:abel-counting} yields
\(
(1-r^2)\sum_{t\in T}r^{2t}\asymp1\) for \( 0<r<1.\)
\end{proof}

Combining Theorems~\ref{thm:commuting-universal} and
\ref{thm:mellin-counterexample}, we obtain
\(
\U\subsetneq\Ucom
=
\{T:0\in T,\ N_T(R)\asymp R\}.
\)
The obstruction exhibited by the counterexample comes from interactions
between different spectral components, which are absent from the commuting
theory. Scalar Abel bounds probe only the diagonal spectral weights, whereas
the self-similar scaling limit produces cancellation across geometrically
separated spectral scales, as detected by the identity
\(H_q(c_0)=0\). A finite orbit realization of the same obstruction, including the singly
generated case, is given in Theorem~\ref{thm:mellin-finite-orbit}.

\subsection{Uniform bounds for a fixed number of spectral values}

We now show that linear counting gives uniform bounds whenever the number
of spectral points is bounded in advance, even if their multiplicities
are unrestricted.

\begin{theorem}
\label{thm:fixed-spectral-count}
Let \(T\subset\Rplus\) be locally finite and fix \(p\in\mathbb N^+\).
The following assertions are equivalent:
\begin{enumerate}[label=\textup{(\roman*)},leftmargin=*]
\item
\(0\in T\) and \(N_T(R)\asymp R\) as \(R\to\infty\).
\item
For every \(0<a\leq b<\infty\), there exist constants
\(0<\alpha_{p,T}(a,b)\leq\beta_{p,T}(a,b)<\infty\) such that, for
every \(1\leq m\leq p\), every choice of positive integers
\(d_1,\ldots,d_m\), every \(\lambda\in[0,1)^m\), and every
\(S\in\Bcal(\mathcal E)\), where
\(\mathcal E=\bigoplus_{i=1}^m\mathbb C^{d_i}\), the conditions
\(a\Id\leq S\leq b\Id\) and \(D_\lambda(S)\geq0\) imply that
\(M_T^\lambda(S)\) is well defined and
\(
\alpha_{p,T}(a,b)\Id
\leq M_T^\lambda(S)
\leq\beta_{p,T}(a,b)\Id.
\)
\end{enumerate}
\end{theorem}

\begin{proof}
For necessity, take \(m=d_1=1\), \(S=1\), and \(a=b=1\).
The choice \(\lambda_1=0\) gives \(0\in T\). Varying
\(\lambda_1\in(0,1)\) gives the scalar Abel bounds, so
Theorem~\ref{thm:abel-counting} implies \(N_T(R)\asymp R\).

Assume \textup{(i)} and fix \(a,b\). By
Theorem~\ref{thm:abel-counting},
\(C_\Phi:=\sup_{0\leq z<1}\Phi_T(z)<\infty\).
In particular, all series defining \(M_T^\lambda(S)\) converge.
We first reduce arbitrary multiplicities to scalar matrices. Put
\(A=\diag(\lambda_1\Id_{d_1},\ldots,\lambda_m\Id_{d_m})\) and
\(B=S-ASA\geq0\). For a unit vector \(x\in\mathcal E\), let
\(E_i\) be the coordinate orthogonal projections on \(\mathcal E\)
and set
\(K=\bigoplus_{i=1}^m\operatorname{span}\{E_ix\}\), omitting zero
summands. Then \(x\in K\), \(\dim K\leq m\), and \(K\) reduces
\(A\). With \(A_K=A|_K\) and \(S_K=P_KS|_K\), we have
\(a\Id_K\leq S_K\leq b\Id_K\) and
\(S_K-A_KS_KA_K=P_KB|_K\geq0\). Moreover,
\begin{equation}\label{eq:fixed-spectral-compression}
\langle M_T^\lambda(S)x,x\rangle
=
\sum_{t\in T}
\langle A_K^t(S_K-A_KS_KA_K)A_K^tx,x\rangle.
\end{equation}
Thus uniform bounds for scalar matrices of all sizes at most \(p\)
give the same bounds for arbitrary multiplicities.

For scalar matrices of size \(m\leq p\), write
\(G=M_T^\lambda(S)\). Lemma~\ref{lem:finite-spectral-identity} gives
\(G\geq0\), and hence
\(\lVert G\rVert\leq\operatorname{tr}G
=\sum_{i=1}^m\Phi_T(\lambda_i^2)S_{ii}\leq pbC_\Phi\).
It remains to establish a uniform positive lower bound. We shall also use
the following consequence of \(B\geq0\):
\begin{equation}\label{eq:fixed-spectral-cross-bound}
|S_{ij}|
\leq b\,
\frac{\sqrt{(1-\lambda_i^2)(1-\lambda_j^2)}}
{1-\lambda_i\lambda_j},
\qquad
|G_{ij}|\leq C_\Phi|S_{ij}|.
\end{equation}
Indeed, \(|B_{ij}|^2\leq B_{ii}B_{jj}\), while
\(B_{ij}=(1-\lambda_i\lambda_j)S_{ij}\) and \(S_{ii}\leq b\).

Suppose that the lower bound fails. Passing to a subsequence with a fixed
size \(m\leq p\), choose scalar models
\(A_k=\diag(\lambda_{1,k},\ldots,\lambda_{m,k})\) and \(S_k\)
such that \(a\Id\leq S_k\leq b\Id\),
\(B_k=S_k-A_kS_kA_k\geq0\), and
\begin{equation}\label{eq:fixed-spectral-degeneration}
\lambda_{\min}(G_k)\longrightarrow0,
\qquad
G_k=M_T^{\lambda^{(k)}}(S_k),
\quad
\lambda^{(k)}=(\lambda_{1,k},\ldots,\lambda_{m,k}).
\end{equation}
By compactness, we may assume that \(S_k\to S_\infty\) and that every
\(\lambda_{i,k}\) has a limit \(\lambda_i^\infty\in[0,1]\).
Let \(I_0=\{i:\lambda_i^\infty<1\}\). For the remaining indices,
put \(h_{i,k}=-\log\lambda_{i,k}\), discarding finitely many terms so
that \(\lambda_{i,k}>0\). Extract a further subsequence on which all
ratios \(h_{i,k}/h_{j,k}\) converge in \([0,\infty]\). Partition these
indices into groups by declaring \(i\sim j\) when that limit is finite
and strictly positive.

If \(i,j\) belong to different boundary groups, then
\[
\frac{\sqrt{(1-\lambda_{i,k}^2)(1-\lambda_{j,k}^2)}}
{1-\lambda_{i,k}\lambda_{j,k}}
\sim
\frac{2\sqrt{h_{i,k}h_{j,k}}}{h_{i,k}+h_{j,k}}
\longrightarrow0.
\]
The same quotient tends to zero when exactly one of the two indices lies
in \(I_0\). Equation~\eqref{eq:fixed-spectral-cross-bound} therefore
shows that the entries of both \(S_k\) and \(G_k\) between different
groups tend to zero. It suffices to prove that the principal blocks of
\(G_k\) converge, along a common subsequence, to positive definite
matrices.

Consider first \(I_0\), when it is nonempty. Let \(A_{0,k},S_{0,k}\)
be the corresponding principal compressions and denote their limits by
\(A_0,S_0\). Then \(\lVert A_0\rVert<1\), \(S_0\geq a\Id\),
and \(B_0=S_0-A_0S_0A_0\geq0\). Since \(N_T(R)=O(R)\),
Lemma~\ref{lem:strict-continuity} gives
\(
(G_k)_{I_0,I_0}\rightarrow
G_0:=\sum_{t\in T}A_0^tB_0A_0^t
\)
in operator norm. If \(\langle G_0y,y\rangle=0\), then
\(B_0^{1/2}A_0^ty=0\) for every \(t\in T\).
Write the distinct positive eigenvalues of \(A_0\) as
\(\rho_1>\cdots>\rho_s>0\), with spectral projections \(P_j\).
Along any sequence of positive times in \(T\) tending to infinity,
\(\sum_{j=1}^s\rho_j^t B_0^{1/2}P_jy=0\). Dividing first by
\(\rho_1^t\) and taking a limit gives \(B_0^{1/2}P_1y=0\).
Repeating this argument eliminates every positive eigenvalue.
The identity at \(t=0\), which is available because \(0\in T\),
then gives the same conclusion for the zero eigenspace.
Thus \(B_0^{1/2}A_0^ny=0\) for every \(n\in\N\). Telescoping the
Stein identity yields
\(S_0=\sum_{n\geq0}A_0^nB_0A_0^n\), so
\(\langle S_0y,y\rangle=0\). Hence \(y=0\), proving that \(G_0\)
is positive definite.

Now fix a boundary group \(J\). Choose a representative scale
\(\varepsilon_k\downarrow0\) such that
\(h_{i,k}=\varepsilon_k\eta_{i,k}\) and
\(\eta_{i,k}\to\eta_i\in(0,\infty)\) for \(i\in J\).
Set \(D_k=\diag(\eta_{i,k})\) and
\(D=\diag(\eta_i)>0\). Then
\(A_{J,k}=\e^{-\varepsilon_kD_k}\), where the subscript \(J\)
denotes principal compression. Write \(S_{J,k}\to S_J\geq a\Id\).
The Stein identity gives
\[
Q_k:=\varepsilon_k^{-1}B_{J,k}
=\frac{S_{J,k}-\e^{-\varepsilon_kD_k}S_{J,k}
\e^{-\varepsilon_kD_k}}{\varepsilon_k}
\longrightarrow Q:=DS_J+S_JD\geq0.
\]
The convergence follows by expanding the exponentials, since \(D_k\)
and \(S_{J,k}\) are uniformly bounded.

Define positive measures on \([0,\infty)\) by
\(\mu_k=\varepsilon_k\sum_{t\in T}\delta_{\varepsilon_kt}\).
Linear upper counting gives
\(\mu_k([0,R])\leq C(\varepsilon_k+R)\) for all \(R\geq0\),
with \(C\) independent of \(k\). Compactness on bounded intervals and
a diagonal extraction give a subsequence converging vaguely to a
positive Radon measure \(\nu\) on \([0,\infty)\).
At every positive continuity point \(R\) of its distribution function,
\(\mu_k([0,R])=\varepsilon_kN_T(R/\varepsilon_k)\) and the linear
counting bounds imply \(cR\leq\nu([0,R])\leq CR\). Approximating an
arbitrary \(R>0\) from above and below by continuity points gives
\(
cR\leq\nu([0,R])\leq CR\) for \( R>0\).
In particular, \(\nu(\{0\})=0\) and \(\supp\nu\) is unbounded.
We may assume \(\varepsilon_k\leq1\). For each \(\eta>0\), the same
upper bound gives
\(
\sup_k\int_{[R,\infty)}\e^{-\eta s}\dd\mu_k(s)
\leq C\sum_{\ell\geq\lfloor R\rfloor}
(\ell+2)\e^{-\eta\ell}
\rightarrow0\) for \(R\to\infty\).
Since \(\min_{i\in J}\eta_{i,k}\) is bounded away from zero and
\(Q_k\to Q\), vague convergence, uniform convergence of the integrands
on compact intervals, and this tail estimate yield
\[
(G_k)_{J,J}
=\int_0^\infty\e^{-sD_k}Q_k\e^{-sD_k}\dd\mu_k(s)
\longrightarrow
G_J:=\int_0^\infty\e^{-sD}Q\e^{-sD}\dd\nu(s)
\]
in operator norm.

To prove that \(G_J\) is positive definite, suppose that
\(\langle G_Jy,y\rangle=0\). The continuous function
\(g(s)=Q^{1/2}\e^{-sD}y\) vanishes on \(\supp\nu\), hence along
a sequence \(s_\ell\to\infty\). If
\(0<\theta_1<\cdots<\theta_v\) are the distinct eigenvalues of \(D\),
then \(g(s)=\sum_{j=1}^v\e^{-\theta_js}v_j\) for suitable vectors
\(v_j\). Multiplying \(g(s_\ell)=0\) by
\(\e^{\theta_1s_\ell}\) and taking a limit gives \(v_1=0\).
Successive repetition gives \(v_j=0\) for all \(j\), so \(g\equiv0\).
On the other hand, \(Q=DS_J+S_JD\) and \(D>0\) imply
\(
\int_0^\infty\e^{-sD}Q\e^{-sD}\dd s=S_J.
\)
Consequently,
\(0=\int_0^\infty\lVert g(s)\rVert^2\dd s
=\langle S_Jy,y\rangle\geq a\lVert y\rVert^2\), and hence \(y=0\).

There are only finitely many groups, so all the preceding limits may be
taken along one subsequence. The cross blocks tend to zero and every
diagonal block has a positive definite limit. Thus \(G_k\) converges in
operator norm to a positive definite block diagonal matrix, contradicting
\eqref{eq:fixed-spectral-degeneration}. This proves a uniform lower
bound for all scalar matrices of sizes at most \(p\).
Equation~\eqref{eq:fixed-spectral-compression} extends both bounds to
arbitrary multiplicities, completing the proof.
\end{proof}

For \(0\in T\), \(N_T(R)\asymp R\), and fixed \(0<a\leq b<\infty\),
define the optimal lower bound over models with at most \(p\)
spectral values by
\[
\alpha_p^*(T;a,b)
=\inf\left\{
\lambda_{\min}\bigl(M_T^\lambda(S)\bigr):
\begin{array}{l}
1\leq m\leq p,\quad\lambda\in[0,1)^m,\\
S\in\Bcal(\mathbb C^m),\quad a\Id\leq S\leq b\Id,\\
D_\lambda(S)\geq0
\end{array}
\right\}.
\]
The compression argument above shows that the same infimum is obtained
if arbitrary spectral multiplicities are allowed.

\begin{corollary}
\label{cor:no-fixed-spectral-test}
Let \(T_*\) be the time set in \eqref{eq:clustered-T}, and set
\(a_0=m_q/2\) and \(b_0=2M_q\), with \(m_q,M_q\) as in
Lemma~\ref{lem:continuous-riesz}. Then
\begin{equation}\label{eq:fixed-spectral-lower-decay}
\alpha_p^*(T_*;a_0,b_0)>0
\quad\text{for every }p\in\mathbb N^+,
\qquad
\alpha_p^*(T_*;a_0,b_0)\downarrow0
\quad(p\to\infty).
\end{equation}
In particular, \(T_*\) satisfies the bound preservation property in
Theorem~\ref{thm:fixed-spectral-count} for every fixed \(p\), but
\(T_*\notin\U\).
\end{corollary}

\begin{proof}
Theorem~\ref{thm:fixed-spectral-count} gives positivity for each finite
\(p\). The quantities \(\alpha_p^*(T_*;a_0,b_0)\) are nonincreasing
in \(p\). The models \(S_{M,L_M}\) from the proof of
Theorem~\ref{thm:mellin-counterexample} have \(2M+1\) spectral points,
satisfy \(a_0\Id\leq S_{M,L_M}\leq b_0\Id\), and are admissible
because their Stein defects are \(C_{M,L_M}^*C_{M,L_M}\geq0\).
Therefore
\[
0\leq\alpha_{2M+1}^*(T_*;a_0,b_0)
\leq\langle G_{M,L_M}x^{(M)},x^{(M)}\rangle
\leq\lVert\mathcal Mx^{(M)}\rVert_2^2+\frac1M
\longrightarrow0.
\]
Monotonicity proves \eqref{eq:fixed-spectral-lower-decay}, and
\(T_*\notin\U\) was proved in
Theorem~\ref{thm:mellin-counterexample}.
\end{proof}

Thus separate bounds for every fixed number of spectral points cannot
replace one positive lower bound common to all finite spectral models.
This does not contradict Theorem~\ref{thm:complete-characterization}:
finite spectra suffice precisely when the bounds are uniform over all
such models. Theorem~\ref{thm:universal-bessel} below shows that linear
upper counting already supplies an upper bound independent of the number
of spectral points; the loss of uniformity concerns only the lower bound.

\section{Density criteria and stability under counting perturbations}

Linear upper counting controls universal Bessel bounds. For the lower
bound, we use natural density, positive density subsets, and a comparison
of Gramians under sublinear changes in the counting function.

Recall that the integer observation family
$\{CA^n\}_{n\geq0}$ is a $g$-Bessel sequence with bound $b>0$
if, for every $x\in H$,
\(
 \sum_{n\geq0}\|CA^n x\|^2\leq b\|x\|^2
\).
In this case, we write
\(
 S:=\sum_{n\geq0}(A^n)^*C^*CA^n\) and \(0\leq S\leq b\Id\),
where the series converges in the strong operator topology.
We first record a spectral reduction that also applies to normal
operators.

\begin{lemma}\label{lem:bessel-spectral-reduction}
Let $A\in\Bcal(H)$ be normal and let $C\in\Bcal(H,Y)$.
Suppose that $\{CA^n\}_{n\geq0}$ is a $g$-Bessel sequence
with bound $b>0$.
Then
\(
CE_{|A|}((1,\infty))=0,
CE_A(\{1\})=0.
\)
Let $P=E_{|A|}([0,1])$, and put
$\widehat A=A|_{PH}$ and $\widehat C=C|_{PH}$.
For every bounded Borel function $\varphi$ on $\sigma(A)$,
\(
C\varphi(A)=\widehat C\varphi(\widehat A)P.
\)
The operator $\widehat A$ is a normal contraction, and its
integer observation system has Bessel bound $b$.

If
\(
E_A\bigl(\{\lambda:|\lambda|=1,\ \lambda\ne1\}\bigr)=0,
\)
the same assertions hold with $P=E_{|A|}([0,1))$.
In this case, $\widehat A$ is strongly stable.
\end{lemma}

\begin{proof}
The integer Bessel bound gives $\|CA^n\|\leq\sqrt b$ for every
$n\geq0$.
For $q>1$, set $P_q=E_{|A|}([q,\infty))$ and
$A_q=A|_{P_qH}$.
The operator $A_q$ is invertible on $P_qH$, with
$\|A_q^{-n}\|\leq q^{-n}$.
Consequently,
\(
\|CP_q\|
=\|CA^nA_q^{-n}P_q\|
\leq\sqrt b\,q^{-n}\rightarrow0.
\)
Thus $CP_q=0$. Letting $q\downarrow1$ gives
$CE_{|A|}((1,\infty))=0$.

If $x\in E_A(\{1\})H$, then $A^nx=x$ for every $n\geq0$.
Hence
\(
\sum_{n\geq0}\|Cx\|^2
=\sum_{n\geq0}\|CA^nx\|^2
\leq b\|x\|^2,
\)
which forces $Cx=0$.

For either admissible choice of $P$, we therefore have $C=CP$.
Since $P$ reduces $A$ and its Borel functional calculus,
the asserted identity follows.
Restricting the integer Bessel inequality to $PH$ preserves
the bound $b$.
For $P=E_{|A|}([0,1))$, strong stability follows from the
spectral theorem and dominated convergence.
\end{proof}

Fix a system that is exactly observable at integer times and write
\begin{equation}\label{eq:S-bounds-density}
a\Id\leq S:=\sum_{n\geq0}A^nBA^n\leq b\Id,
\qquad
B=C^*C=S-ASA.
\end{equation}
For $0<r<1$, let $E_r=E_A([r,1])$. Given $x\in H$, set
\(
F_x(t)=\lVert CA^t x\rVert^2
\)
for \(t\geq0\), and define the continuous time observability Gramian
\(G_c\) by
\(
\langle G_cx,x\rangle
=
\int_0^\infty F_x(t)\dd t.
\)

\begin{lemma}\label{lem:continuous-bridge}
Let $A$ be a strongly stable positive contraction and let
$C\in\Bcal(H,Y)$.
Suppose that $\{CA^n\}_{n\geq0}$ is a $g$-Bessel sequence
with bound $b>0$.
Define $E_r$, $F_x$, and $G_c$ as above.
Then $G_c$ is bounded and positive, with
$0\leq G_c\leq b\Id$.
Moreover,
\(
\|E_r(G_c-S)E_r\|\longrightarrow0\) as \(r\uparrow1\).
\end{lemma}

\begin{proof}
Set
$
Q_x(t)=\langle SA^t x,A^t x\rangle.
$
Since $A^t$ commutes with $A$ and $B=S-ASA$,
\(
F_x(t)=Q_x(t)-Q_x(t+1).
\)
Hence, for every $R>0$,
\[
\int_0^R F_x(t)\dd t
=
\int_0^1 Q_x(s)\dd s
-
\int_R^{R+1}Q_x(s)\dd s.
\]
Since $A^t x\to0$ as $t\to\infty$, the second term tends to zero.
Therefore
\(
\int_0^\infty F_x(t)\dd t
=
\int_0^1Q_x(s)\dd s
\leq
b\lVert x\rVert^2.
\)
Thus the corresponding quadratic form defines a bounded positive operator
$G_c$ satisfying $0\leq G_c\leq b\Id$.

If $x\in E_rH$, then
\[
\langle(S-G_c)x,x\rangle
=
\int_0^1
\langle(S-A^sSA^s)x,x\rangle
\dd s.
\]
Using
\(
S-A^sSA^s
=
(\Id-A^s)S+A^sS(\Id-A^s)
\)
and
$
\lVert(\Id-A^s)E_r\rVert
\leq
1-r^s
\leq
s(-\log r),
$
we obtain
\[
\lVert E_r(S-G_c)E_r\rVert
\leq
2b\int_0^1(1-r^s)\dd s
\leq
2b(-\log r)\int_0^1s\dd s
=
b(-\log r).
\]
This proves the assertion.
\end{proof}

The next estimate controls both global upper bounds and boundary
comparisons, with constants independent of the dimension.

\begin{lemma}\label{lem:weighted-variation}
Under the assumptions and notation of
Lemma~\ref{lem:continuous-bridge}, for every $0<r<1$
and $x\in E_rH$,
\[
\int_0^\infty t|F_x'(t)|\dd t
\leq\sqrt{2}\,b\|x\|^2.
\]
\end{lemma}

\begin{proof}
On $E_rH$, the operator $H_A=-\log A$ is bounded and
$A^t=\e^{-tH_A}$. Set
$f_x(t)=CA^t x$,
$f_x'(t)=-CH_AA^t x$, and
$q_x(t)=\langle SH_AA^t x,H_AA^t x\rangle$.
The Stein identity gives
\(
q_x(t)-q_x(t+1)
=
\lVert CH_AA^t x\rVert^2
=
\lVert f_x'(t)\rVert^2.
\)
For every $R\geq1$, a change of variables on finite intervals gives
\[
\begin{aligned}
\int_0^R t^2\lVert f_x'(t)\rVert^2\dd t
&=\int_0^1 t^2q_x(t)\dd t
  +\int_1^R(2t-1)q_x(t)\dd t
-\int_R^{R+1}(t-1)^2q_x(t)\dd t\\
&\leq 2b\int_0^R t\lVert H_AA^t x\rVert^2\dd t
\leq \frac{b}{2}\lVert x\rVert^2.
\end{aligned}
\]
Here we used $q_x(t)\geq0$, $q_x(t)\leq
b\lVert H_AA^t x\rVert^2$, and the spectral identity
\[
\int_0^\infty th^2\e^{-2ht}\dd t=\frac14,\quad h>0,
\]
with the integrand equal to zero when $h=0$.
Letting $R\to\infty$ and applying monotone convergence yields
\[
\int_0^\infty t^2\lVert f_x'(t)\rVert^2\dd t
\leq\frac{b}{2}\lVert x\rVert^2.
\]
Lemma~\ref{lem:continuous-bridge} also gives
\(
\int_0^\infty\lVert f_x(t)\rVert^2\dd t
\leq
b\lVert x\rVert^2.
\)
Since
$|F_x'(t)|
\leq
2\lVert f_x(t)\rVert\lVert f_x'(t)\rVert$,
the Cauchy--Schwarz inequality yields
\[
\int_0^\infty t|F_x'(t)|\dd t
\leq
2
\left(
\int_0^\infty\lVert f_x(t)\rVert^2\dd t
\right)^{1/2}
\left(
\int_0^\infty t^2\lVert f_x'(t)\rVert^2\dd t
\right)^{1/2}
\leq
\sqrt{2}\,b\lVert x\rVert^2.
\]
\end{proof}

The weighted variation estimate characterizes the time sets
that preserve the Bessel property of integer observation systems
for positive operators. No lower observability bound is required,
and $0\in T$ is not needed.

\begin{theorem}
\label{thm:universal-bessel}
Let \(T\subset\Rplus\) be locally finite. The following assertions are
equivalent:
\begin{enumerate}[label=\textup{(\roman*)},leftmargin=*]
\item
For every bounded positive operator $A$ and every bounded
observation operator $C$, Besselness of the $\N$-system
implies Besselness of the $T$-system.
\item
\(N_T(R)=O(R)\) as \(R\to\infty\).
\end{enumerate}
When these conditions hold, set
\(
K_T=\sup_{R>0}\frac{\#(T\cap(0,R])}{R}<\infty.
\)
Every positive operator system $(A,C)$ whose integer observation
family $\{CA^n\}_{n\geq0}$ is a $g$-Bessel sequence
with bound $b>0$ then obeys
\begin{equation}\label{eq:universal-bessel-bound}
G_T\leq b\bigl(\ind_{\{0\in T\}}+\sqrt{2}\,K_T\bigr)\Id.
\end{equation}
Here \(\ind_{\{0\in T\}}\) equals \(1\) if \(0\in T\)
and \(0\) otherwise.
In particular, the bound is independent of the dimension, spectral
multiplicity, and observation space.
\end{theorem}

\begin{proof}
Assume \textup{(ii)}, and let $(A,C)$ be a positive operator
system such that $\{CA^n\}_{n\geq0}$ is a $g$-Bessel sequence
with bound $b>0$.
By Lemma~\ref{lem:bessel-spectral-reduction}, restriction to
$E_A([0,1))H$ preserves all observations and the integer
Bessel bound.
It therefore suffices to prove the estimate for a strongly
stable positive contraction.
We retain the notation $A$ and $C$ for the restricted system.

Local finiteness implies that the positive points of $T$,
if any, are bounded away from zero.
Together with the linear upper counting bound, this gives
$K_T<\infty$.
Write $N_+(s)=\#(T\cap(0,s])$.
For $0<r<1$ and $x\in E_rH$, put $F_x(t)=\|CA^tx\|^2$.
Since $F_x(t)\to0$, Lemma~\ref{lem:weighted-variation} gives
\(
tF_x(t)\leq\int_t^\infty s|F_x'(s)|\dd s\to0.
\)
Stieltjes integration by parts therefore yields
\[
\sum_{\substack{t\in T\\t>0}}F_x(t)
=-\int_0^\infty N_+(s)F_x'(s)\dd s
\leq K_T\int_0^\infty s|F_x'(s)|\dd s
\leq\sqrt{2}\,K_Tb\lVert x\rVert^2.
\]
Indeed, the boundary term at infinity vanishes because
\(N_+(s)\leq K_Ts\), and the boundary term at zero is zero.

Put \(P_0=E_A(\{0\})\). As \(r\downarrow0\),
\(E_rx\to(\Id-P_0)x\) for every \(x\in H\). Applying the preceding
estimate first to each finite sum and then passing to the limit gives
\[
\sum_{\substack{t\in T\\t>0}}\lVert CA^tx\rVert^2
\leq\sqrt{2}\,K_Tb\lVert(\Id-P_0)x\rVert^2,
\]
because \(A^tP_0=0\) for \(t>0\). If \(0\in T\), add
\(\lVert Cx\rVert^2\leq b\lVert x\rVert^2\). This proves
\eqref{eq:universal-bessel-bound}, including strong convergence of the
positive Gramian sums, and hence \textup{(i)}.

Conversely, assume \textup{(i)} and use the Parseval multiplication system
\(H=Y=L^2(0,1)\), \((Af)(r)=rf(r)\), and
\((Cf)(r)=\sqrt{1-r^2}\,f(r)\). Its \(T\)-system is Bessel, so
\(F_T(r)\leq M_T\) almost everywhere for some \(M_T<\infty\).
The continuity argument in the proof of
Theorem~\ref{thm:abel-counting} uses only this upper bound and shows that
it holds for every \(0<r<1\). With
\(L_T(u)=\sum_{t\in T}\e^{-ut}\), we obtain, for \(R\geq1\),
\[
\e^{-1}N_T(R)
\leq L_T(1/R)
=\frac{F_T(\e^{-1/(2R)})}{1-\e^{-1/R}}
\leq\frac{M_T}{1-\e^{-1}}R.
\]
Thus \(N_T(R)=O(R)\).
\end{proof}

In the finite spectral model of Lemma~\ref{lem:finite-spectral-identity},
\eqref{eq:universal-bessel-bound} becomes
\(M_T^\lambda(S)\leq
b(\ind_{\{0\in T\}}+\sqrt{2}\,K_T)\Id\).
Thus linear upper counting already controls the upper bound at every
matrix level. The additional obstruction in
Theorem~\ref{thm:mellin-counterexample} concerns only the lower bound.

We next consider the reverse implication. In the following proposition,
the standing assumption \eqref{eq:S-bounds-density} is not imposed.
In particular, integer time observability is a conclusion rather than
an assumption.

\begin{proposition}
\label{prop:reverse-preservation}
Let \(T\subset\Rplus\) be locally finite and satisfy
\(N_T(R)\asymp R\) as \(R\to\infty\).
Let \(A\in\Bcal(H)\) be positive and \(C\in\Bcal(H,Y)\).
If the \(T\)-system is exactly observable, then the \(\N\)-system
is exactly observable.
\end{proposition}

\begin{proof}
Suppose that
\[
 \alpha\lVert x\rVert^2
 \leq \sum_{t\in T}\lVert CA^t x\rVert^2
 \leq \beta\lVert x\rVert^2,
 \qquad x\in H,
\]
where \(0<\alpha\leq\beta<\infty\).

We first show that \(A\) is a strongly stable positive contraction.
For \(\varepsilon>0\), put
\(P_\varepsilon=E_A([1+\varepsilon,\infty))\).
For every positive \(t\in T\), spectral calculus gives
\[
 \lVert CP_\varepsilon\rVert
 \leq \lVert CA^tP_\varepsilon\rVert
       \lVert A^{-t}P_\varepsilon\rVert
 \leq \sqrt{\beta}(1+\varepsilon)^{-t},
\]
where \(A^{-t}P_\varepsilon\) is defined only on this spectral
subspace and extended by zero.
Since \(T\) is unbounded, \(CP_\varepsilon=0\).
The lower observability bound then gives \(P_\varepsilon=0\).
Also, on \(E_A(\{1\})H\), all observations coincide with \(C\).
Since \(T\) is infinite, the upper bound forces
\(CE_A(\{1\})=0\), and the lower bound gives \(E_A(\{1\})=0\).
Thus \(0\leq A\leq\Id\), and the spectral theorem yields
\(A^n x\to0\) for every \(x\in H\).

We next prove the integer Bessel bound.
Set \(h_j=2^{-j}\).
The lower counting estimate provides \(j_0\in\N\) and
\(\kappa>0\) such that
\(
 N_T\bigl((4h_j)^{-1}\bigr)\geq \kappa h_j^{-1}\) for \(j\geq j_0\).
For \(j\geq j_0\), put
\(
 P_j=E_A\bigl((\e^{-h_j},\e^{-h_j/2}]\bigr), M_j=\lVert CP_j\rVert, P_*=E_A([0,\e^{-h_{j_0}}]).
\)
These projections are mutually orthogonal, and
\(P_*+\sum_{j\geq j_0}P_j=\Id\) strongly.
If \(M_j>0\), choose a unit vector \(x\in P_jH\) with
\(\lVert Cx\rVert\geq3M_j/4\).
For \(0\leq t\leq(4h_j)^{-1}\),
\[
 \lVert CA^t x-Cx\rVert
 \leq M_j\lVert(A^t-\Id)P_j\rVert
 \leq M_j t h_j
 \leq M_j/4.
\]
Consequently,
\(
 \beta\geq \frac{M_j^2}{4}
 N_T\bigl((4h_j)^{-1}\bigr),
 M_j^2\leq \frac{4\beta}{\kappa}h_j.
\)
The latter inequality also holds when \(M_j=0\).

For \(L\in\N\), define the finite sum
\(S_L=\sum_{n=0}^{L}A^nC^*CA^n\).
For \(i,j\geq j_0\),
\[
 \lVert P_iS_LP_j\rVert
 \leq \frac{M_iM_j}{1-\e^{-(h_i+h_j)/2}}
 \leq K_1\frac{\sqrt{h_i h_j}}{h_i+h_j}
 \leq K_1 2^{-|i-j|/2},
\]
where \(K_1\) is independent of \(i,j,L\).
The scalar matrix \([2^{-|i-j|/2}]\) has uniformly bounded row
and column sums. Applying the Schur estimate first to finite
spectral sums and then passing to their strong limit gives
\(
 \sum_{n=0}^{L}
 \lVert CA^n(\Id-P_*)x\rVert^2
 \leq K_2\lVert(\Id-P_*)x\rVert^2
\)
with \(K_2\) independent of \(L\).
On the remaining spectral subspace,
\(
 \sum_{n\geq0}\lVert CA^nP_*x\rVert^2
 \leq
 \frac{\lVert C\rVert^2}{1-\e^{-2h_{j_0}}}
 \lVert P_*x\rVert^2.
\)
Combining the two estimates and letting \(L\to\infty\), we obtain
an integer Bessel bound \(b_{\N}<\infty\).
Let \(\Gamma_{\N}\) denote the resulting bounded analysis operator.

It remains to prove the integer lower bound.
Write
\(
 \mathcal E_c(x)=\int_0^\infty\lVert CA^t x\rVert^2\dd t.
\)
Tonelli's theorem and the integer Bessel bound give
\begin{equation}\label{eq:reverse-continuous-energy}
 \mathcal E_c(x)
 =\int_0^1\lVert\Gamma_{\N}A^s x\rVert^2\dd s
 \leq b_{\N}\lVert x\rVert^2.
\end{equation}
Thus the continuous Gramian \(G_c\) satisfies
\(0\leq G_c\leq b_{\N}\Id\).

We record a sampling estimate that uses only this upper bound.
For \(m\geq3\), put
\(
 R_m=E_A([m^{-1},1-m^{-1}]).
\)
Initially take \(z\in R_mH\), and let \(Q=-\log A\) on \(R_mH\).
For \(f_z(t)=CA^t z\), Tonelli's theorem yields
\[
 \int_0^\infty t^2\lVert f_z'(t)\rVert^2\dd t
 =2\int_0^\infty
 s\langle G_cQA^s z,QA^s z\rangle\dd s
 \leq 2b_{\N}\int_0^\infty s\lVert QA^s z\rVert^2\dd s
 =\frac{b_{\N}}{2}\lVert z\rVert^2.
\]
Here the last equality follows from
\(\int_0^\infty s q^2\e^{-2sq}\dd s=1/4\) for \(q>0\).
With \(F_z(t)=\lVert f_z(t)\rVert^2\), Cauchy--Schwarz therefore gives
\(
 \int_0^\infty t|F_z'(t)|\dd t
 \leq \sqrt{2b_{\N}}\,\lVert z\rVert\,
       \mathcal E_c(z)^{1/2}.
\)
Set \(N_+(R)=\#(T\cap(0,R])\) and
\(K_T=\sup_{R>0}N_+(R)/R<\infty\).
Local finiteness makes \(N_+\) vanish near zero, while
\(F_z\) decays exponentially at infinity.
Stieltjes integration by parts consequently gives
\[
 \sum_{t\in T\setminus\{0\}}\lVert CA^t z\rVert^2
 =-\int_0^\infty N_+(s)F_z'(s)\dd s
 \leq K_T\sqrt{2b_{\N}}\,\lVert z\rVert\,
       \mathcal E_c(z)^{1/2}.
\]
For general \(x\in H\), take \(z=R_mx\) and let \(m\to\infty\).
Since \(R_mx\to(\Id-E_A(\{0\}))x\), the boundedness of the
continuous analysis operator implies
\(\mathcal E_c(R_mx)\to\mathcal E_c(x)\).
Fatou's lemma, followed by adding the time \(0\) term if present, gives
\begin{equation}\label{eq:reverse-sampling-estimate}
 \sum_{t\in T}\lVert CA^t x\rVert^2
 \leq
 \ind_{\{0\in T\}}\lVert Cx\rVert^2
 +K_T\sqrt{2b_{\N}}\,\lVert x\rVert\,
       \mathcal E_c(x)^{1/2}.
\end{equation}

Suppose now that \(\Gamma_{\N}\) is not bounded below.
Choose unit vectors \(x_k\) with \(\Gamma_{\N}x_k\to0\).
If \(L\) denotes the backward shift on \(\ell^2(\N;Y)\), then
\(\Gamma_{\N}A=L\Gamma_{\N}\).
Hence
\(\Gamma_{\N}p(A)x_k=p(L)\Gamma_{\N}x_k\to0\)
for every polynomial \(p\).
Uniform polynomial approximation of \(u\mapsto u^s\) on \([0,1]\)
shows that
\(\Gamma_{\N}A^s x_k\to0\) for every \(s\in[0,1]\),
where \(s=0\) is immediate.
Dominated convergence in \eqref{eq:reverse-continuous-energy} gives
\(\mathcal E_c(x_k)\to0\).
Also, \(\lVert Cx_k\rVert\leq\lVert\Gamma_{\N}x_k\rVert\to0\).
Equation~\eqref{eq:reverse-sampling-estimate} now implies
\(
 \sum_{t\in T}\lVert CA^t x_k\rVert^2\rightarrow0,
\)
contradicting the lower \(T\)-observability bound.
Thus \(\Gamma_{\N}\) is bounded below, completing the proof.
\end{proof}

\begin{corollary}
\label{cor:universal-equivalence}
For a locally finite set \(T\subset\Rplus\), the following are equivalent:
\begin{enumerate}[label=\textup{(\roman*)},leftmargin=*]
\item \(T\in\U\).
\item For every positive \(A\in\Bcal(H)\) and
every \(C\in\Bcal(H,Y)\), the \(\N\)-system is exactly observable
if and only if the \(T\)-system is exactly observable.
\end{enumerate}
Consequently, the uniform Stein--Pick criterion in
Theorem~\ref{thm:complete-characterization} also characterizes
universal equivalence of exact observability at the two time sets.
\end{corollary}

\begin{proof}
Assertion \textup{(ii)} implies \textup{(i)} by definition.
If \(T\in\U\), Theorem~\ref{thm:abel-counting} gives
\(0\in T\) and \(N_T(R)\asymp R\).
The forward implication follows from \(T\in\U\), and
Proposition~\ref{prop:reverse-preservation} supplies the reverse one.
\end{proof}

\begin{remark}
Proposition~\ref{prop:reverse-preservation} applies to every locally
finite time set with two-sided linear counting, including the
nonuniversal set in Theorem~\ref{thm:mellin-counterexample}.
Thus, within this counting class, the additional obstruction concerns
the implication from integer time observability to \(T\)-observability.
\end{remark}

We now return to the standing assumption \eqref{eq:S-bounds-density}.
The next lemma isolates the passage from a lower bound near the spectral
boundary to exact observability on the whole space.

\begin{lemma}
\label{lem:high-spectrum-lower}
Assume \eqref{eq:S-bounds-density}, \(0\in T\), and
\(N_T(R)\asymp R\). If, for some \(0<r_0<1\) and \(\alpha>0\),
\(E_{r_0}G_TE_{r_0}\geq\alpha E_{r_0}\), then the \(T\)-system is
exactly observable.
\end{lemma}

\begin{proof}
By Theorem~\ref{thm:universal-bessel}, the analysis operator
\(\Gamma_T\) is bounded. Suppose that it is not bounded below, and choose
unit vectors \(x_j\) such that \(\Gamma_Tx_j\to0\). Write
\(T=\{0=\tau_0<\tau_1<\cdots\}\). Linear counting gives
\(\tau_k\asymp k\), and hence
\(\sum_{k\geq1}\tau_k/(1+\tau_k^2)=\infty\).
By the full M\"untz theorem \cite{BorweinErdelyi},
\(\operatorname{span}\{z^t:t\in T\}\) is dense in \(C[0,1]\).

For every \(t\in T\), we have \(CA^tx_j\to0\). Thus
\(Cp(A)x_j\to0\) for every
\(p\in\operatorname{span}\{z^t:t\in T\}\). Uniform approximation and
\(\lVert C(g-p)(A)x_j\rVert\leq
\lVert C\rVert\lVert g-p\rVert_\infty\)
then imply \(Cg(A)x_j\to0\) for every \(g\in C[0,1]\).

Choose \(r_0<r_1<1\) and \(\chi\in C[0,1]\) with
\(0\leq\chi\leq1\), \(\chi=1\) on \([0,r_0]\), and
\(\chi=0\) on \([r_1,1]\). Put
\(u_j=\chi(A)x_j\) and \(v_j=(\Id-\chi(A))x_j\).
For every fixed \(n\in\N\), we have \(CA^nu_j\to0\), while
\[
\sum_{n\geq N}\lVert CA^nu_j\rVert^2
\leq\lVert C\rVert^2\sum_{n\geq N}r_1^{2n}
\]
uniformly in \(j\). It follows that \(\Gamma_{\N}u_j\to0\), and the
integer lower bound gives \(u_j\to0\). Boundedness of \(\Gamma_T\) now
implies \(\Gamma_Tu_j\to0\). Hence
\(\lVert v_j\rVert\to1\) and \(\Gamma_Tv_j\to0\), whereas
\(v_j\in E_{r_0}H\) and
\(\lVert\Gamma_Tv_j\rVert^2\geq\alpha\lVert v_j\rVert^2\).
This contradiction proves the lower bound.
\end{proof}

\subsection{Natural density}

Assume that \(0\in T\) and
\begin{equation}\label{eq:natural-density-assumption}
N_T(R)=dR+o(R),\qquad 0<d<\infty.
\end{equation}
Set \(N_+(R)=\#(T\cap(0,R])\) and \(D(R)=N_+(R)-dR\).
Then \(D(R)=o(R)\), and Theorem~\ref{thm:universal-bessel} ensures that
\(G_T\) is bounded for every system satisfying
\eqref{eq:S-bounds-density}.

\begin{theorem}
\label{thm:high-spectrum-tauberian}
Under \eqref{eq:S-bounds-density} and
\eqref{eq:natural-density-assumption},
\(\lVert E_r(G_T-dS)E_r\rVert\to0\) as \(r\uparrow1\).
The convergence is uniform over systems with fixed integer bounds
\(a,b\).
\end{theorem}

\begin{proof}
Fix \(0<r<1\) and \(x\in E_rH\). Stieltjes integration by parts gives
\[
\sum_{\substack{t\in T\\0<t\leq M}}F_x(t)
-d\int_0^M F_x(t)\dd t
=D(M)F_x(M)-\int_0^M D(s)F_x'(s)\dd s.
\]
As in the proof of Theorem~\ref{thm:universal-bessel},
\(MF_x(M)\to0\). Since \(D(M)=o(M)\), the boundary term tends to zero.
Letting \(M\to\infty\) therefore yields
\begin{equation}\label{eq:density-variation-identity}
\sum_{\substack{t\in T\\t>0}}F_x(t)
-d\int_0^\infty F_x(t)\dd t
=-\int_0^\infty D(s)F_x'(s)\dd s.
\end{equation}

Given \(\varepsilon>0\), choose \(R_\varepsilon>0\) such that
\(|D(s)|\leq\varepsilon s\) for \(s\geq R_\varepsilon\), and put
\(M_\varepsilon=\sup_{0\leq s\leq R_\varepsilon}|D(s)|\).
Lemma~\ref{lem:weighted-variation} bounds the tail in
\eqref{eq:density-variation-identity} by
\(\sqrt{2}\,b\varepsilon\lVert x\rVert^2\).

For the bounded interval, set \(\delta_r=-\log r\). Since
\(B=S-ASA\) and \(E_r\) commutes with \(A\),
\(\lVert E_rBE_r\rVert\leq2b(1-r)\).
For \(f_x(t)=CA^tx\), this gives
\(\lVert f_x(t)\rVert^2\leq2b(1-r)\lVert x\rVert^2\) and
\(\lVert f_x'(t)\rVert^2
\leq2b(1-r)\delta_r^2\lVert x\rVert^2\).
Consequently,
\begin{equation}\label{eq:high-spectrum-pointwise-derivative}
|F_x'(t)|\leq4b(1-r)(-\log r)\lVert x\rVert^2,
\qquad t\geq0.
\end{equation}
Hence the integral over \([0,R_\varepsilon]\) is at most
\(4bM_\varepsilon R_\varepsilon(1-r)\delta_r\lVert x\rVert^2\).
The time \(0\) contributes
\(F_x(0)\leq2b(1-r)\lVert x\rVert^2\), so
\[
\lVert E_r(G_T-dG_c)E_r\rVert
\leq2b(1-r)
+4bM_\varepsilon R_\varepsilon(1-r)\delta_r
+\sqrt{2}\,b\varepsilon.
\]
Letting first \(r\uparrow1\) and then \(\varepsilon\downarrow0\), and
using Lemma~\ref{lem:continuous-bridge}, proves the assertion. All
constants are independent of the system except through its integer
upper bound \(b\).
\end{proof}

\begin{theorem}
\label{thm:natural-density}
Let \(T\subset\Rplus\) be locally finite and suppose that the natural
density \(d(T)=\lim_{R\to\infty}N_T(R)/R\) exists in \([0,\infty]\).
Then \(T\in\U\) if and only if \(0\in T\) and \(0<d(T)<\infty\).
\end{theorem}

\begin{proof}
Suppose that \(0\in T\) and \(d=d(T)\in(0,\infty)\), and fix an
arbitrary system satisfying \eqref{eq:S-bounds-density}.
Theorem~\ref{thm:high-spectrum-tauberian} gives \(0<r_0<1\) such that
\(E_{r_0}G_TE_{r_0}\geq(da/2)E_{r_0}\).
Since \(N_T(R)\asymp R\), Lemma~\ref{lem:high-spectrum-lower} shows
that the \(T\)-system is exactly observable. Thus \(T\in\U\).

Conversely, if \(T\in\U\), Theorem~\ref{thm:abel-counting} gives
\(0\in T\) and \(N_T(R)\asymp R\). Existence of the natural density
then implies \(0<d(T)<\infty\).
\end{proof}

The boundary comparison is uniform over all finite spectral models.

\begin{corollary}
\label{cor:complete-density}
Suppose that \(0\in T\) and \(N_T(R)=dR+o(R)\), where
\(0<d<\infty\). For every \(0<a\leq b<\infty\),
\[
\lim_{r\uparrow1}\sup
\left\{
\lVert M_T^\lambda(S)-dS\rVert:
\begin{array}{l}
m\in\mathbb N^+,\quad d_1,\ldots,d_m\in\mathbb N^+,\\
\mathcal E=\displaystyle\bigoplus_{i=1}^m\mathbb C^{d_i},\quad
\lambda\in[r,1)^m,\\
S\in\Bcal(\mathcal E),\quad a\Id\leq S\leq b\Id,\quad
D_\lambda(S)\geq0
\end{array}
\right\}=0.
\]
\end{corollary}

\begin{proof}
Apply Theorem~\ref{thm:high-spectrum-tauberian} to the model in
Lemma~\ref{lem:finite-spectral-identity}, where \(G_{\N}=S\) and
\(G_T=M_T^\lambda(S)\). The estimates are independent of the number of
spectral points and their multiplicities.
\end{proof}

\subsection{Subsets of positive natural density and lower Beurling density}

Once the upper bound is separated from the lower bound, the natural
density theorem also applies to time sets that contain a suitable subset.
This observation in the block diagonal Carleson setting, together with the interval
density used below, appears in
\cite[Remark~4.2]{KrishtalMillerBlock}.

\begin{proposition}
\label{prop:positive-density-subset}
Let \(T\subset\Rplus\) be locally finite and suppose that it contains a
subset \(T_0\) with natural density \(d(T_0)\in(0,\infty)\).
Then \(T\in\U\) if and only if \(0\in T\) and
\(N_T(R)=O(R)\).
\end{proposition}

\begin{proof}
Necessity follows from Theorem~\ref{thm:abel-counting}.
For sufficiency, \(T_0\cup\{0\}\subset T\) is universal by
Theorem~\ref{thm:natural-density}. Its lower observability bound is also
a lower bound for the \(T\)-system. Theorem~\ref{thm:universal-bessel}
provides the upper bound.
\end{proof}

For completeness, the existence of a subset as in
Proposition~\ref{prop:positive-density-subset} can be expressed in terms
of the macroscopic lower interval density
\[
d^-_{\mathrm{mac}}(T)
=\inf_{0<\eta<1}\liminf_{R\to\infty}
\inf_{0\leq s\leq(1-\eta)R}
\frac{N_T(R)-N_T(s)}{R-s}.
\]
We include the selection argument needed for the sampling criterion.

\begin{corollary}
\label{cor:macroscopic-density}
A locally finite set \(T\subset\Rplus\) satisfies
\(d^-_{\mathrm{mac}}(T)>0\) if and only if it contains a subset of
positive finite natural density. Consequently, within this class,
\(T\in\U\) if and only if \(0\in T\) and \(N_T(R)=O(R)\).
\end{corollary}

\begin{proof}
Suppose that \(T_0\subset T\) and
\(N_{T_0}(R)=dR+o(R)\), where \(d>0\).
The error is \(o(R)\) uniformly when evaluated at \(s\in[0,R]\).
Thus, for each fixed \(0<\eta<1\),
\[
\frac{N_T(R)-N_T(s)}{R-s}
\geq\frac{N_{T_0}(R)-N_{T_0}(s)}{R-s}
=d+o(1)
\]
uniformly for \(0\leq s\leq(1-\eta)R\).
It follows that \(d^-_{\mathrm{mac}}(T)\geq d>0\).

Conversely, choose \(0<d<d^-_{\mathrm{mac}}(T)\).
Construct a subset \(T_0\subset T\setminus\{0\}\) by scanning the
positive points of \(T\) in increasing order and retaining a point \(t\)
exactly when the number of previously retained points is smaller than
\(dt\). Writing \(M(R)=N_{T_0}(R)\), this rule gives
\(M(R)\leq dR+1\).

Fix \(0<\eta<1\). For all sufficiently large \(R\),
\[
N_T(R)-N_T(s)\geq d(R-s),
\qquad 0\leq s\leq(1-\eta)R.
\]
If a point of \(T\cap(0,R]\) was rejected, let \(s\) be the last such
point. Then \(M(s)\geq ds\), and every point in \(T\cap(s,R]\) was
retained. If \(s\leq(1-\eta)R\), this gives
\(M(R)=M(s)+N_T(R)-N_T(s)\geq dR\).
If \(s>(1-\eta)R\), it gives \(M(R)\geq d(1-\eta)R\).
If no point was rejected, use the preceding interval estimate with
\(s=0\) to obtain
\(M(R)=N_T(R)-N_T(0)\geq dR\).
Therefore \(\liminf_{R\to\infty}M(R)/R\geq d(1-\eta)\).
Letting \(\eta\downarrow0\), and using \(M(R)\leq dR+1\), proves
\(d(T_0)=d\). The final assertion follows from
Proposition~\ref{prop:positive-density-subset}.
\end{proof}

Density conditions of Beurling type have a long history in irregular
sampling and frame theory; see, for example,
\cite{LandauDensity,GrochenigIrregular,ChristensenDengHeil,
BalanCasazzaHeilLandauDensity}.
For a locally finite set \(T\subset\Rplus\), define its lower and upper
Beurling densities on the half line by
\[
D^-(T)
=
\liminf_{R\to\infty}
\inf_{x\geq0}
\frac{\#(T\cap[x,x+R])}{R},
\qquad
D^+(T)
=
\limsup_{R\to\infty}
\sup_{x\geq0}
\frac{\#(T\cap[x,x+R])}{R}.
\]
Then
\(
0<D^-(T)\leq D^+(T)<\infty
\)
is equivalent to the existence of \(L>0\), \(M\in\mathbb N^+\), and
\(k_0\in\mathbb N\) such that
\begin{equation}\label{eq:fixed-scale-occupancy}
1
\leq
\#\bigl(T\cap[kL,(k+1)L)\bigr)
\leq
M,
\qquad k\geq k_0.
\end{equation}
Indeed, if \(D^-(T)>0\) and \(D^+(T)<\infty\), choose \(L>0\)
sufficiently large that every closed interval \([x,x+L]\), \(x\geq0\),
contains at least two and at most \(M\) points of \(T\), for some
\(M\in\mathbb N^+\). Removing the right endpoint leaves at least
one point, so \eqref{eq:fixed-scale-occupancy} holds with \(k_0=0\).
Conversely, \eqref{eq:fixed-scale-occupancy} implies uniform linear upper
and lower bounds for the number of points in every sufficiently long
interval, since such an interval contains all but at most two of the
corresponding \(L\)-blocks. The finitely many blocks before \(k_0\) do not
affect the limiting densities.

The next theorem requires no uniform upper bound on the number of points
in translated intervals.

\begin{theorem}
\label{thm:occupancy}
Let \(T\subset\Rplus\) be locally finite and suppose that
\(D^-(T)>0\). Then
\[
T\in\U
\quad\Longleftrightarrow\quad
0\in T\ \text{ and }\ N_T(R)=O(R).
\]
In particular, \(0\in T\) and
\(0<D^-(T)\leq D^+(T)<\infty\) imply \(T\in\U\).
\end{theorem}

\begin{proof}
Necessity follows from Theorem~\ref{thm:abel-counting}.
Conversely, suppose that \(0\in T\) and \(N_T(R)=O(R)\).
Positive lower Beurling density gives \(L>0\) and
\(k_0\in\mathbb N^+\) such that
\(T\cap[kL,(k+1)L)\neq\varnothing\) for every \(k\geq k_0\).
Set \(\tau_0=0\), and choose
\(\tau_n\in T\cap[(k_0+n-1)L,(k_0+n)L)\) for \(n\geq1\).
Then \((\tau_n)\) is strictly increasing and \(\tau_n/n\to L\).
Hence \(T_0=\{\tau_n:n\in\N\}\subset T\) has natural density
\(1/L\). Proposition~\ref{prop:positive-density-subset} gives
\(T\in\U\). The last assertion follows because finite upper Beurling
density implies \(N_T(R)=O(R)\).
\end{proof}

\begin{example}
\label{ex:unbounded-upper-beurling}
Let \(T_*\) denote the set in \eqref{eq:clustered-T}, and put
\(T=\N\cup T_*\). Then \(N_T(R)=O(R)\) and \(D^-(T)\geq1\),
so \(T\in\U\) by Theorem~\ref{thm:occupancy}.
On the other hand, the clusters \(B_m\subset T_*\) contain \(q^m\)
points in intervals of length at most \(1\). Thus every fixed interval
length \(R\geq1\) has unbounded point counts over its translates, and
\(D^+(T)=\infty\). This shows that the linear upper counting assumption
in Theorem~\ref{thm:occupancy} is strictly weaker than finite upper
Beurling density.
\end{example}

The natural density condition and the condition
\(D^-(T)>0\), \(N_T(R)=O(R)\) do not imply one another.

\begin{example}
\label{ex:occupancy-no-density}
Choose integers
\(
0=N_0<N_1<N_2<\cdots
\)
such that
\(
N_{j+1}-N_j\geq jN_j\) for \(j\in\N\).
Define
\[
T
=
\{0\}
\cup
\bigcup_{j\geq0}
\{k+1/3:N_{2j}\leq k<N_{2j+1}\}
\cup
\bigcup_{j\geq0}
\{k+1/3,k+2/3:N_{2j+1}\leq k<N_{2j+2}\}.
\]
Then
\(
1
\leq
\#\bigl(T\cap[k,k+1)\bigr)
\leq
2\) for \(k\in\N,\)
and hence $T\in\U$ by Theorem~\ref{thm:occupancy}.

Moreover,
\(
\frac{N_j}{N_{j+1}}
\leq
\frac{1}{j+1}
\rightarrow0.
\)
At the end of a one point block,
\(
N_T(N_{2j+1})
=
N_{2j+1}-N_{2j}
+
O(N_{2j}),
\)
and therefore
\(
\frac{N_T(N_{2j+1})}{N_{2j+1}}
\rightarrow1.
\)
At the end of a two point block,
\[
N_T(N_{2j+2})
=
2\bigl(N_{2j+2}-N_{2j+1}\bigr)
+
O(N_{2j+1}),
\]
and hence
\(
\frac{N_T(N_{2j+2})}{N_{2j+2}}
\rightarrow2.
\)
Thus the natural density of $T$ does not exist.
\end{example}

\begin{example}
\label{ex:density-unbounded-gaps}
Fix $N\in\mathbb N^+$. Choose integers
\(
1<K_1<K_2<\cdots
\)
so rapidly increasing that, with
\(
H_m=\sqrt{K_m},
\
L_m=\left\lceil\frac{2H_m}{N}\right\rceil,
\)
for \(m\geq1,\) the intervals
\(
[K_m,K_m+L_m]\)
are pairwise disjoint and satisfy
\(
K_{m+1}>K_m+L_m+1.
\)
Define
\[
\delta_{K_m+j}
=
\max\left\{
H_m-\frac{N}{2}j,0
\right\},
\qquad
0\leq j\leq L_m,
\]
and set $\delta_k=0$ for all remaining $k\in\N$. Let
\(
\tau_k=Nk+\delta_k,
\
T=\{\tau_k:k\in\N\}.
\)

For indices inside one of the intervals,
\(
\delta_{k+1}-\delta_k\geq-\frac{N}{2},
\)
while outside these intervals the perturbation is zero. At each left
endpoint $K_m$, the perturbation jumps upward. Hence
\(
\tau_{k+1}-\tau_k
=
N+\delta_{k+1}-\delta_k
\geq
\frac{N}{2}>0,
\)
so $(\tau_k)$ is strictly increasing.

Moreover, if $K_m\leq k\leq K_m+L_m$, then
\(
0\leq
\frac{\delta_k}{k}
\leq
\frac{H_m}{K_m}
=
\frac{1}{\sqrt{K_m}},
\)
whereas $\delta_k=0$ outside these intervals. Therefore
\(
\frac{\delta_k}{k}\rightarrow0
\)
and hence
\(
\frac{\tau_k}{k}
=
N+\frac{\delta_k}{k}
\rightarrow N.
\)
It follows that
\(
\frac{N_T(R)}{R}
\rightarrow
\frac{1}{N}\) for \(R\to\infty.\)
Since $0=\tau_0\in T$, Theorem~\ref{thm:natural-density} gives
$T\in\U$.

At the left endpoint of the $m$th interval,
\(
\tau_{K_m}-\tau_{K_m-1}
=
N+H_m
=
N+\sqrt{K_m}
\rightarrow\infty.
\)
Thus the adjacent gaps of $T$ are unbounded. Consequently, for every
$L>0$ there exists an interval of length greater than $2L$ containing no
point of $T$, and such an interval contains some block
$[kL,(k+1)L)$. Hence
\(
T\cap[kL,(k+1)L)=\varnothing
\)
for arbitrarily large \(k\in\N\). Therefore \(T\) has natural density \(1/N\) but \(D^-(T)=0\),
and hence fails every fixed scale lower point count condition.
\end{example}

\subsection{Stability under sublinear counting perturbations}

The same integration by parts argument compares two time sets whose
counting functions differ by a sublinear term. This comparison does not
require a pairing of their sampling times.

\begin{theorem}
\label{thm:counting-stability}
Let \(T,T_0\subset\Rplus\) be locally finite, with
\(0\in T\cap T_0\), and suppose that
\(
N_T(R)-N_{T_0}(R)=o(R).
\)
Then \(T\in\U\) if and only if \(T_0\in\U\).
Moreover, if \(N_{T_0}(R)=O(R)\), every system satisfying
\eqref{eq:S-bounds-density} obeys
\(
\lim_{r\uparrow1}\lVert E_r(G_T-G_{T_0})E_r\rVert=0,
\)
uniformly over systems with fixed integer bounds \(a,b\).
\end{theorem}

\begin{proof}
We first prove the asserted norm convergence under the additional
assumption \(N_{T_0}(R)=O(R)\).
Then \(N_T(R)=O(R)\) as well, and
Theorem~\ref{thm:universal-bessel} ensures that both Gramians are
bounded. Put \(D(s)=N_T(s)-N_{T_0}(s)\). The contributions at zero
cancel. For \(x\in E_rH\), Stieltjes integration by parts gives
\[
\langle(G_T-G_{T_0})x,x\rangle
=-\int_0^\infty D(s)F_x'(s)\dd s.
\]
The boundary term vanishes because \(D(s)=o(s)\) and
\(sF_x(s)\to0\).

Given \(\varepsilon>0\), choose \(R_\varepsilon>0\) such that
\(|D(s)|\leq\varepsilon s\) for \(s\geq R_\varepsilon\), and set
\(M_\varepsilon=\sup_{0\leq s\leq R_\varepsilon}|D(s)|\).
Combining Lemma~\ref{lem:weighted-variation} with
\eqref{eq:high-spectrum-pointwise-derivative} gives
\[
\lVert E_r(G_T-G_{T_0})E_r\rVert
\leq4bM_\varepsilon R_\varepsilon(1-r)(-\log r)
+\sqrt{2}\,b\varepsilon.
\]
Letting first \(r\uparrow1\) and then \(\varepsilon\downarrow0\)
proves the asserted norm convergence. The constants depend only on
\(T,T_0\), and \(b\), not on the particular system.

Now suppose that \(T_0\in\U\), and fix an arbitrary system satisfying
\eqref{eq:S-bounds-density}. By Theorem~\ref{thm:abel-counting},
\(N_{T_0}(R)\asymp R\), so the assumption
\(N_T(R)-N_{T_0}(R)=o(R)\) gives \(N_T(R)\asymp R\).
Since \(T_0\) is universal,
\(G_{T_0}\geq\alpha_0\Id\) for some \(\alpha_0>0\).
The norm convergence proved above then implies
\(E_rG_TE_r\geq(\alpha_0/2)E_r\) for \(r\) sufficiently close to
\(1\). Lemma~\ref{lem:high-spectrum-lower} yields exact observability
of the \(T\)-system. Hence \(T\in\U\).
Interchanging \(T\) and \(T_0\) proves the converse.
\end{proof}

In particular, finite changes away from the time \(0\) do not change
membership in \(\U\). More generally, the signed counting difference
\(N_T-N_{T_0}\) need not control the counting function of the
symmetric difference.

\begin{corollary}
\label{cor:bounded-density-profile}
Let \(T\subset\Rplus\) be locally finite and contain \(0\).
Suppose that there is a measurable function \(w\) with
\(0<c\leq w(s)\leq C<\infty\) almost everywhere such that
\[
N_T(R)=\int_0^R w(s)\dd s+o(R).
\]
Then \(T\in\U\).
\end{corollary}

\begin{proof}
Set \(W(R)=\int_0^R w(s)\dd s\). The function \(W\) is continuous,
strictly increasing, and maps \(\Rplus\) onto itself. Define
\(s_n=W^{-1}(n)\) for \(n\in\N\) and
\(T_0=\{s_n:n\in\N\}\). The bounds on \(w\) give
\(1/C\leq s_{n+1}-s_n\leq1/c\), so
Theorem~\ref{thm:occupancy} implies \(T_0\in\U\).
Moreover, \(N_{T_0}(R)=\lfloor W(R)\rfloor+1\), and therefore
\(N_T(R)-N_{T_0}(R)=o(R)\). Apply
Theorem~\ref{thm:counting-stability}.
\end{proof}

The function \(w\) describes the counting function, not weights in the
sampling sum. Its averages need not converge, so
Corollary~\ref{cor:bounded-density-profile} does not require a natural
density. The complete characterization of \(\U\) remains
Theorem~\ref{thm:complete-characterization}; the results above give density
criteria for its upper bound, sufficient hypotheses for its lower bound,
and an invariance property under counting perturbations.

\section{Quantitative stability and critical perturbation scales}

This section quantifies comparison of sampling systems through the
quadratic discrepancy between their time powers. We first establish a
boundary transfer principle and then apply it to regular lattices.
For the cumulative perturbation condition below, universality already
follows from Theorem~\ref{thm:natural-density}; the additional result is
that the quadratic discrepancy tends to zero at the spectral boundary.
We also determine the critical scale of this discrepancy for regular
power perturbations and give examples with sparse larger deviations.

Let
\(
T=\{t_k:k\in\N\}
\)
and
\(
T_0=\{t_k^{(0)}:k\in\N\}
\)
be infinite locally finite time sets, listed in increasing order, with
\(t_k,t_k^{(0)}\to\infty\). Put \(\Sigma=\sigma(A)\subset[0,1]\). We write \(C(\Sigma)\) for the
space of continuous complex functions on \(\Sigma\), equipped with
the uniform norm. For \(z\in[0,1]\) and \(K\in\N\), define
the following nonnegative sums, which may be infinite:
\(
Q_{T,T_0}(z)
=
\sum_{k\geq0}|z^{t_k}-z^{t_k^{(0)}}|^2,
\)
\(
Q_{T,T_0}^{(K)}(z)
=
\sum_{k\geq K}|z^{t_k}-z^{t_k^{(0)}}|^2,
\)
and
\(
R_{T_0}^{(K)}(z)
=
\sum_{k\geq K}z^{2t_k^{(0)}}.
\)

\subsection{A transfer principle and a cumulative perturbation criterion}

The following result is the basic perturbative mechanism. It separates the
interior spectral region, controlled by M\"untz completeness and uniform
tail decay, from the boundary region, controlled by \(Q_{T,T_0}\).
Suprema of nonnegative functions over empty spectral sets are
understood to be zero.

\begin{theorem}
\label{thm:muntz-transfer}
Assume that the reference analysis operator
\(
\Gamma_{T_0}x=\{CA^{t_k^{(0)}}x\}_{k\geq0}
\)
is bounded and bounded below and that:
\begin{enumerate}[label=\textup{(\roman*)}]
\item
\(\operatorname{span}\{z^{t_k}:k\geq0\}\) is dense in \(C(\Sigma)\);

\item
for every \(0<\rho<1\),
\(
\sup_{z\in\Sigma\cap[0,\rho]}R_{T_0}^{(K)}(z)\rightarrow0,\
\sup_{z\in\Sigma\cap[0,\rho]}Q_{T,T_0}^{(K)}(z)\rightarrow0\) as \(K\to\infty\);

\item
\(\sup_{z\in\Sigma}Q_{T,T_0}(z)<\infty\) and
\(
\lim_{r\uparrow1}
\sup_{z\in\Sigma\cap[r,1]}Q_{T,T_0}(z)=0.
\)
\end{enumerate}
Then the \(T\)-sampled analysis operator
\(
\Gamma_Tx=\{CA^{t_k}x\}_{k\geq0}
\)
is bounded and bounded below.
\end{theorem}

\begin{proof}
The spectral theorem gives
\[
\lVert(\Gamma_T-\Gamma_{T_0})x\rVert^2
\leq
\lVert C\rVert^2
\int_\Sigma Q_{T,T_0}(z)\,
\dd\langle E_A(z)x,x\rangle .
\]
Hence \(\Gamma_T\) is bounded, and condition \textup{(iii)} yields
\(
\lVert(\Gamma_T-\Gamma_{T_0})E_A([r,1])\rVert\to0
\)
as \(r\uparrow1\).

Suppose that \(\Gamma_T\) is not bounded below. Choose unit vectors \(x_j\)
such that \(\Gamma_Tx_j\to0\). By \textup{(i)} and continuous functional
calculus, \(Cg(A)x_j\to0\) for every \(g\in C(\Sigma)\). Fix
\(0<r<\rho<1\), choose \(\chi\in C(\Sigma)\) such that
\(0\leq\chi\leq1\), \(\chi=1\) on \(\Sigma\cap[0,r]\), and
\(\chi=0\) on \(\Sigma\cap[\rho,1]\), and set
\(u_j=\chi(A)x_j\) and \(v_j=(\Id-\chi(A))x_j\).

For each fixed coordinate, both \(CA^{t_k^{(0)}}u_j\) and
\(CA^{t_k}u_j\) tend to zero. Condition \textup{(ii)} supplies the
corresponding uniform tail estimates, and therefore
\(
\lVert\Gamma_{T_0}u_j\rVert+\lVert\Gamma_Tu_j\rVert\to0.
\)
Hence \(\Gamma_Tv_j\to0\).

Let
\(
m_{T_0}:=\inf_{\lVert x\rVert=1}
\lVert\Gamma_{T_0}x\rVert>0.
\)
Then
\[
m_{T_0}
\leq
\lVert\Gamma_{T_0}u_j\rVert
+
\lVert\Gamma_Tv_j\rVert
+
\lVert(\Gamma_{T_0}-\Gamma_T)E_A([r,1])\rVert
\lVert v_j\rVert .
\]
Letting first \(j\to\infty\) and then \(r\uparrow1\) gives a contradiction.
Thus \(\Gamma_T\) is bounded below, so the \(T\)-system is
exactly observable.
\end{proof}

To verify the boundary condition in Theorem~\ref{thm:muntz-transfer}, we
use the following Abel estimate, which converts a cumulative bound of
order \(o(K^2)\) on the squared deviations into decay near the spectral
boundary.

\begin{lemma}
\label{lem:exponential-abel}
Let \(a_k\geq0\) and
\(
A_K=\sum_{k=0}^{K}a_k.
\)
Suppose that
\(
D=\limsup_{K\to\infty}\frac{A_K}{(K+1)^2}<\infty.
\)
Then, for every \(c>0\),
\[
\limsup_{u\downarrow0}
u^2\sum_{k\geq0}a_k\e^{-cuk}
\leq
\frac{2D}{c^2}.
\]
If \(A_K=o(K^2)\), then
\(u^2\sum_{k\geq0}a_k\e^{-cuk}\to0\) as \(u\downarrow0\).
\end{lemma}

\begin{proof}
Put \(q=\e^{-cu}\). Abel summation gives
\(\sum_{k\geq0}a_kq^k=(1-q)\sum_{K\geq0}A_Kq^K\). Since
\(\sum_{K\geq0}(K+1)^2q^K=(1+q)/(1-q)^3\) and
\(1-\e^{-cu}\sim cu\), the asserted estimate follows.

If \(A_K=o(K^2)\), split the sum into a finite initial part and a tail on
which \(A_K\leq\varepsilon(K+1)^2\), and then let
\(\varepsilon\downarrow0\).
\end{proof}

We now specialize the reference set to a regular lattice. The following
proposition justifies the choice \(T_0=N\mathbb N\).

\begin{proposition}
\label{prop:integer-step}
Let $N\in\mathbb N^+$. Then the $\N$-system is exactly observable if
and only if the $N\N$-system is exactly observable.
\end{proposition}

\begin{proof}
Suppose first that the $N\N$-system is exactly observable, and
let $b_N$ be an upper observability bound. Since
$N\N\subset\N$, its lower observability bound is also a lower
bound for the $\N$-system. Moreover,
\[
 \sum_{n\geq0}\lVert CA^n x\rVert^2
 =\sum_{r=0}^{N-1}\sum_{k\geq0}\lVert CA^{Nk}A^r x\rVert^2
 \leq b_N\sum_{r=0}^{N-1}\lVert A^r\rVert^2\lVert x\rVert^2.
\]

Conversely, assume that the $\N$-system is exactly observable. By
Lemma~\ref{lem:auto-contraction}, $0\leq A\leq\Id$. The operator
$\Gamma_Nx=\{CA^{Nk}x\}_{k\geq0}$ is bounded. If it were not bounded below, there would exist a sequence of unit vectors
$\{x_j\}_{j\geq1}\subset H$ such that
$\lVert\Gamma_Nx_j\rVert\to0$. Let $L$ be the
backward shift on $\ell^2(\N;Y)$. Then
$\Gamma_NA^N=L\Gamma_N,
 \
 \Gamma_Np(A^N)=p(L)\Gamma_N$
for every polynomial $p$. Since $A^r=(A^N)^{r/N}$ and $z^{r/N}$ can be approximated uniformly
on $[0,1]$ by polynomials, the identity
$\Gamma_Np(A^N)=p(L)\Gamma_N$
implies
$
 \Gamma_NA^rx_j\rightarrow0$ for $0\leq r<N.$ Consequently,
\[
 \sum_{n\geq0}\lVert CA^n x_j\rVert^2
 =\sum_{r=0}^{N-1}\lVert\Gamma_NA^r x_j\rVert^2\longrightarrow0,
\]
contrary to the lower observability bound for the $\N$-system.
\end{proof}

\begin{remark}
\label{rem:positivity-integer-step}
The positivity assumption in Proposition~\ref{prop:integer-step} is
structural. In its proof, positivity allows us to recover the intermediate
powers from $A^N$ through continuous functional calculus:
\(
A^r=(A^N)^{r/N}
\)
for $0\leq r<N$. This mechanism fails for general self-adjoint operators,
because the map $\lambda\mapsto\lambda^N$ may identify distinct spectral
points.

Indeed, let
\(
A=\diag(r,-r)
\)
with $0<r<1$, and let
\(
C(x_1,x_2)=x_1+x_2.
\)
Then
\(
CA^{2k}x=r^{2k}(x_1+x_2)
\)
and
\(
CA^{2k+1}x=r^{2k+1}(x_1-x_2).
\)
Consequently,
\[
\sum_{n\geq0}|CA^nx|^2
=
\frac{|x_1+x_2|^2+r^2|x_1-x_2|^2}{1-r^4},
\]
so the \(\N\)-system is exactly observable. On the other hand,
\(
CA^{2k}(1,-1)=0
\)
for every $k\geq0$, and hence the $2\N$-system is not exactly observable.

Thus Proposition~\ref{prop:integer-step} fails if positivity is replaced
by self-adjointness. The obstruction is that even powers identify the
spectral points \(r\) and \(-r\).
\end{remark}

We now take the reference set to be the lattice \(T_0=N\N\). In this case
we write
\(
Q_{T,N}:=Q_{T,N\N},
\)
so that
\(
Q_{T,N}(z)=\sum_{k\geq0}|z^{t_k}-z^{Nk}|^2.
\)
The following theorem gives a sufficient condition for
\(Q_{T,N}(r)\to0\) as \(r\uparrow1\).

\begin{theorem}
\label{thm:cumulative-transport}
Let
\(
T=\{t_k:k\in\N\}
\)
be listed in increasing order with \(t_0=0\). If, for some
\(N\in\mathbb N^+\),
\(
\sum_{k=0}^{K}|t_k-Nk|^2=o(K^2),
\)
then
\(
Q_{T,N}(r)
=
\sum_{k\geq0}|r^{t_k}-r^{Nk}|^2
\longrightarrow0
\)
as \(r\uparrow1\). Consequently, \(T\in\U\).
\end{theorem}

\begin{proof}
Put \(\delta_k=t_k-Nk\). The hypothesis
\(
\sum_{k=0}^{K}|t_k-Nk|^2=o(K^2)
\)
implies \(\delta_k=o(k)\), and hence \(t_k/k\to N\). Consequently,
\(\sum_{k\geq1}t_k/(1+t_k^2)=\infty\), so the full M\"untz theorem gives
density of \(\operatorname{span}\{z^{t_k}:k\geq0\}\) in \(C[0,1]\).

Put \(r=\e^{-u}\). By the mean value theorem,
\(
|\e^{-ut_k}-\e^{-uNk}|
\leq
u|\delta_k|\e^{-u\min\{t_k,Nk\}}.
\)
Fix \(0<\varepsilon<N\). Then
\(\min\{t_k,Nk\}\geq(N-\varepsilon)k\) for all sufficiently large \(k\).
Applying Lemma~\ref{lem:exponential-abel} to \(a_k=\delta_k^2\) therefore
gives
\(
\sum_{k\geq0}|\e^{-ut_k}-\e^{-uNk}|^2\to0
\)
as \(u\downarrow0\). Thus \(Q_{T,N}(r)\to0\) as \(r\uparrow1\).

Since \(t_k/k\to N\), there exist \(c_0>0\) and \(k_0\) such that
\(t_k\geq c_0k\) for \(k\geq k_0\). Hence, for every \(0<\rho<1\),
both
\(
\sup_{0\leq r\leq\rho}\sum_{k\geq K}r^{2Nk}\to0
\)
and
\(
\sup_{0\leq r\leq\rho}
\sum_{k\geq K}|r^{t_k}-r^{Nk}|^2\to0
\)
as \(K\to\infty\). Together with the boundary limit, this also gives
\(\sup_{0\leq r<1}Q_{T,N}(r)<\infty\).

Theorem~\ref{thm:muntz-transfer} now applies with \(T_0=N\N\). By
Proposition~\ref{prop:integer-step}, the \(N\N\)-system is exactly
observable whenever the integer system is. Hence \(T\in\U\).
\end{proof}

Bounded perturbations of a lattice are an immediate consequence.

\begin{corollary}
\label{cor:bounded-jitter}
Let
\(
T=\{t_k:k\in\N\}
\)
be listed in increasing order. If \(t_0=0\) and
\(
\sup_{k\geq0}|t_k-Nk|<\infty
\)
for some \(N\in\mathbb N^+\), then
\(
Q_{T,N}(r)\rightarrow0\) as \(r\uparrow1\),
and hence \(T\in\U\).
\end{corollary}

\begin{proof}
The hypothesis gives
\(\sum_{k=0}^{K}|t_k-Nk|^2=O(K)=o(K^2)\), so
Theorem~\ref{thm:cumulative-transport} applies.
\end{proof}

\subsection{Fixed lattice comparison and universal sampling}

The cumulative criterion allows highly nonuniform perturbations. We now
compare it with pointwise control relative to the fixed lattice \(N\N\).
For
\(
T=\{t_k:k\in\N\}
\)
write
\(
\delta_k=t_k-Nk
\)
and define
\(
q^*(T,N)=\limsup_{r\uparrow1}Q_{T,N}(r).
\)
The first result shows that \(\sqrt{k}\) is the critical scale for the
quadratic comparison function \(Q_{T,N}\).
\begin{theorem}
\label{thm:square-root-radius}
If
\(
D_\infty
=
\limsup_{k\to\infty}
\frac{|t_k-Nk|}{\sqrt{k}}
<\infty,
\)
then
\(
q^*(T,N)
\leq
\frac{D_\infty^2}{4N^2}.
\)
If \((t_k-Nk)/\sqrt{k}\to d\), then
\(\lim_{r\uparrow1}Q_{T,N}(r)=d^2/(4N^2)\).
\end{theorem}

\begin{proof}
Put \(\delta_k=t_k-Nk\). The hypothesis implies \(\delta_k=o(k)\) and
\[
\limsup_{K\to\infty}
\frac{\sum_{k=0}^{K}\delta_k^2}{(K+1)^2}
\leq
\frac{D_\infty^2}{2}.
\]
With \(r=\e^{-u}\), the same mean value estimate as above, followed by
Lemma~\ref{lem:exponential-abel}, gives
\(
q^*(T,N)
\leq
\frac{D_\infty^2}{4N^2}
\).

Suppose now that \(\delta_k/\sqrt{k}\to d\). Then
\[
Q_{T,N}(\e^{-u})
=
u^2\sum_{k\geq0}\delta_k^2\e^{-2Nku}
\left|
\frac{\e^{-u\delta_k}-1}{u\delta_k}
\right|^2,
\]
where the last factor is interpreted as \(1\) when \(\delta_k=0\).
On \(k\leq R/u\), this factor converges uniformly to \(1\), and the
corresponding Riemann sums converge to
\(d^2\int_0^R x\e^{-2Nx}\dd x\). The remaining tail is bounded by a
constant multiple of \(u^2\sum_{k>R/u}k\e^{-Nku}\). Letting first
\(u\downarrow0\) and then \(R\to\infty\) yields
\[
\lim_{r\uparrow1}Q_{T,N}(r)
=
d^2\int_0^\infty x\e^{-2Nx}\dd x
=
\frac{d^2}{4N^2}.
\]
\end{proof}

Thus \(\sqrt{k}\) is the critical pointwise scale for regular comparison
with the fixed reference lattice \(N\N\). This threshold, however, is not the universal sampling threshold. The distinction becomes explicit for regular power drifts.
Let
\(
T=\{t_k:k\in\N\},
\)
where \(t_0=0\) and
\begin{equation}\label{eq:power-drift}
t_k=Nk+ck^\gamma,
\qquad k\geq1,
\end{equation}
with \(c>0\) and \(\gamma>0\).

\begin{proposition}
\label{prop:power-kernel}
If \(0<\gamma<1\), then, as \(u\downarrow0\),
\(
Q_{T,N}(\e^{-u})
\sim
\frac{c^2\Gamma(2\gamma+1)}
{(2N)^{2\gamma+1}}
u^{1-2\gamma}.
\)
Consequently,
\[
q^*(T,N)
=
\begin{cases}
0, & 0<\gamma<\tfrac12,\\[1mm]
\dfrac{c^2}{4N^2}, & \gamma=\tfrac12,\\[2mm]
+\infty, & \tfrac12<\gamma<1.
\end{cases}
\]
\end{proposition}

\begin{proof}
We have
\(
Q_{T,N}(\e^{-u})
=
\sum_{k\geq1}\e^{-2Nku}(1-\e^{-ck^\gamma u})^2.
\)
Let \(\phi(y)=(1-\e^{-y})/y\) for \(y>0\) and \(\phi(0)=1\). Then
\[
u^{2\gamma-1}Q_{T,N}(\e^{-u})
=
c^2u\sum_{k\geq1}(ku)^{2\gamma}\e^{-2Nku}
\phi\!\left(c(ku)^\gamma u^{1-\gamma}\right)^2.
\]
The factor involving \(\phi\) tends uniformly to \(1\) on bounded
\(ku\)-intervals, while the tail is dominated by a constant multiple of
\(x^{2\gamma}\e^{-2Nx}\). Hence
\[
\lim_{u\downarrow0}
u^{2\gamma-1}Q_{T,N}(\e^{-u})
=
c^2\int_0^\infty x^{2\gamma}\e^{-2Nx}\dd x
=
\frac{c^2\Gamma(2\gamma+1)}
{(2N)^{2\gamma+1}}.
\]
The three alternatives follow.
\end{proof}

The preceding proposition identifies \(\gamma=1/2\) as the fixed-lattice
comparison threshold. Natural density gives the larger threshold for universal sampling.

\begin{corollary}
\label{cor:power-universal-threshold}
Let \(T\) be defined by \eqref{eq:power-drift}. Then
\(T\in\U\) if and only if \(0<\gamma\leq1\). There is no smallness
restriction on \(c\) in the universal range.
\end{corollary}

\begin{proof}
If \(0<\gamma<1\), then \(t_k/k\to N\), so \(T\) has natural density
\(1/N\). If \(\gamma=1\), it has natural density \(1/(N+c)\). Thus
Theorem~\ref{thm:natural-density} gives \(T\in\U\) for
\(0<\gamma\leq1\). If \(\gamma>1\), then
\(N_T(R)\asymp R^{1/\gamma}=o(R)\), contradicting
Theorem~\ref{thm:abel-counting}.
\end{proof}

Hence the threshold for comparison with a fixed reference lattice and the
threshold for universal sampling are genuinely different:
\[
\gamma=\frac12
\quad\text{for comparison with the fixed lattice},
\qquad
\gamma=1
\quad\text{for universal sampling}.
\]
In particular, if \(1/2<\gamma<1\), then
\(Q_{T,N}(r)\to\infty\) as \(r\uparrow1\), although \(T\in\U\).

The next example shows that \(Q_{T,N}(r)\to0\) as \(r\uparrow1\)
may hold even when
\(\limsup_{k\to\infty}|t_k-Nk|/\sqrt{k}=\infty\).

\begin{example}
\label{ex:sparse-super-root}
Fix \(1/2<\beta<2/3\). We construct a locally finite time set
\(
T=\{t_k:k\in\N\}
\)
with
\(
t_k=Nk+\delta_k,
\quad
t_0=0,
\)
such that
\(
\sup_k(t_{k+1}-t_k)=\infty,
\
\limsup_{k\to\infty}\frac{\delta_k}{\sqrt{k}}=\infty,
\)
while
\(
\sum_{k=0}^{K}\delta_k^2=o(K^2).
\)
Consequently, \(T\in\U\).

Indeed, choose rapidly increasing integers \(K_m\), set
\(H_m=K_m^\beta\) and
\(L_m=\lceil2H_m/N\rceil\), and define
\(
\delta_{K_m+j}
=
\max\{H_m-Nj/2,0\}\) for \(0\leq j\leq L_m,\)
with \(\delta_k=0\) outside these intervals. Choose the \(K_m\) so that
the intervals are pairwise disjoint and, if \(S_{m-1}\) denotes the total
squared deviation over the preceding intervals, then
\(
S_{m-1}\leq K_m^2/m
\).

Along each descending interval,
\(
\delta_{k+1}-\delta_k\geq-N/2,
\)
and hence
\(
t_{k+1}-t_k\geq N/2>0.
\)
Moreover, the contribution of the \(m\)th interval to the cumulative squared
deviation satisfies
\[
\sum_{j=0}^{L_m}\delta_{K_m+j}^2
\leq
(L_m+1)H_m^2
\leq
C_NK_m^{3\beta}
=
o(K_m^2),
\]
because \(3\beta<2\). Together with the choice of \(K_m\), this gives
\(\sum_{k=0}^{K}\delta_k^2=o(K^2)\). At the left endpoints,
\[
\frac{\delta_{K_m}}{\sqrt{K_m}}
=
K_m^{\beta-1/2}\longrightarrow\infty,
\qquad
t_{K_m}-t_{K_m-1}
=
N+H_m\longrightarrow\infty.
\]
Thus Theorem~\ref{thm:cumulative-transport} gives \(T\in\U\).
\end{example}

Thus pointwise square root control is not necessary for vanishing
quadratic comparison. The cumulative estimate controls perturbations at
the boundary scale \(k\asymp(1-r)^{-1}\).

\begin{remark}
\label{rem:wrong-reference-lattice}
Vanishing quadratic discrepancy relative to a fixed reference lattice
is not necessary for universal sampling. Indeed, let
\(M,N\in\mathbb N^+\) with \(M\neq N\), and take \(T=M\N\). By
Proposition~\ref{prop:integer-step}, \(T\in\U\). Relative to the reference
lattice \(N\N\), however,
\[
Q_{T,N}(r)
=
\frac{1}{1-r^{2M}}
+
\frac{1}{1-r^{2N}}
-
\frac{2}{1-r^{M+N}}.
\]
As \(r\uparrow1\), the coefficient of \((1-r)^{-1}\) is
\[
\frac{1}{2M}
+
\frac{1}{2N}
-
\frac{2}{M+N}
=
\frac{(M-N)^2}{2MN(M+N)}
>0.
\]
Hence \(Q_{T,N}(r)\to\infty\), even though \(T\in\U\).
\end{remark}

The estimates above distinguish quantitative comparison with a fixed
lattice from universal sampling. Regular power drifts remain universal up
to the linear scale, while sufficiently sparse perturbations may exceed
the square root scale and still give \(Q_{T,N}(r)\to0\).

\section{Nonuniform sampling for normal operators}
\label{sec:normal-extension}

We first characterize universal Bessel preservation for general normal operators in terms of uniformly bounded counts on unit intervals. Under a fixed logarithmic sector condition, we recover the linear cumulative counting criterion and establish the corresponding Gramian comparisons near the unit circle. The Bessel and boundary comparison results allow arbitrary observation spaces, and no commutativity between \(A\) and \(C^*C\) is assumed. For finite-dimensional observation spaces, positive finite natural density also yields a compact Gramian comparison and reduces exact observability to completeness.

Throughout this section, $A\in\Bcal(H)$ is normal,
$C\in\Bcal(H,Y)$, and $B=C^*C$.
The integer Bessel assumption is that
$\{CA^n\}_{n\geq0}$ is a $g$-Bessel sequence with bound $b>0$.
When exact observability is assumed, we write
\begin{equation}\label{eq:normal-input}
a\Id\leq S\leq b\Id,
\qquad 0<a\leq b<\infty.
\end{equation}
Under the integer Bessel assumption, the Stein and tail identities give
\[
B=S-A^*SA,\qquad
(A^n)^*SA^n=\sum_{k\geq n}(A^k)^*BA^k.
\]
If \eqref{eq:normal-input} holds, these identities imply
$A^nx\to0$ for every $x\in H$ and
$\sup_n\|A^n\|\leq\sqrt{b/a}$.
Normality then gives $\|A\|\leq1$, and the spectral projection
of $A$ on the unit circle is zero.

For $t>0$, define $A^t$ by Borel functional calculus using
\(
\lambda^t=\exp\bigl(t(\log|\lambda|+i\arg\lambda)\bigr), \arg\lambda\in(-\pi,\pi], \lambda\ne0,
\)
and put $0^t=0$. We retain $A^0=\Id$.

For a locally finite time set $T$, write
$G_T=\sum_{t\in T}(A^t)^*BA^t$ whenever this positive sum
defines a bounded operator.
Under the integer Bessel assumption with bound $b>0$, let
$\Pi=E_{|A|}([0,1])$.
Lemma~\ref{lem:bessel-spectral-reduction} gives
$C=C\Pi$ and $S=\Pi S\Pi$, while $A|_{\Pi H}$ is a contraction.
The semigroup identity and Tonelli's theorem therefore give,
without any angular restriction,
\begin{equation}\label{eq:normal-continuous-energy}
\begin{aligned}
\int_0^\infty\|CA^tx\|^2\dd t
=\int_0^1\sum_{n\geq0}\|CA^nA^sx\|^2\dd s
=\int_0^1\langle SA^sx,A^sx\rangle\dd s
\leq b\|x\|^2.
\end{aligned}
\end{equation}

In contrast to the positive case, universal Bessel preservation for general normal operators requires uniformly bounded counts on unit intervals.

\begin{theorem}\label{thm:normal-local-counting}
Let $T\subset\Rplus$ be locally finite. The following assertions are
equivalent:
\begin{enumerate}[label=\textup{(\roman*)},leftmargin=*]
\item
For every bounded normal operator $A$ and every bounded
observation operator $C$, Besselness of the $\N$-system
implies Besselness of the $T$-system.
\item
The local counts satisfy
\(
Q_T:=\sup_{n\in\N}\#\bigl(T\cap[n,n+1)\bigr)<\infty.
\)
\end{enumerate}
\end{theorem}

\begin{proof}
Assume first that $Q_T<\infty$, and let $(A,C)$ be a normal
operator system such that $\{CA^n\}_{n\geq0}$ is a
$g$-Bessel sequence with bound $b>0$.
By Lemma~\ref{lem:bessel-spectral-reduction}, we may restrict to
$E_{|A|}([0,1])H$ and assume that $A$ is a contraction.
This reduction preserves all observations and the integer
Bessel bound.
We retain the notation $A$ and $C$ for the restricted system.
Put $P=E_{|A|}([1/2,1])$ and
$\ell^2=(\log2)^2+\pi^2$.
On $PH$, set $L=-\operatorname{Log}(A|_{PH})$, using the prescribed
Borel functional calculus. Then $L$ is bounded and $\|L\|\leq\ell$.
For $x\in PH$, set $f(t)=CA^tx$. Since $f'(t)=-CA^tLx$,
\eqref{eq:normal-continuous-energy}, applied to $x$ and $Lx$, gives
\[
\int_0^\infty\|f(t)\|^2\dd t\leq b\|x\|^2,\qquad
\int_0^\infty\|f'(t)\|^2\dd t\leq b\ell^2\|x\|^2.
\]
On each unit interval, the fundamental theorem of calculus and
Cauchy--Schwarz give
\[
\sup_{n\leq t\leq n+1}\|f(t)\|^2
\leq
2\int_n^{n+1}
\bigl(\|f(s)\|^2+\|f'(s)\|^2\bigr)\dd s.
\]
Indeed, one may choose a point at which $\|f\|^2$ does not exceed
its average and integrate $f'$ from that point. Summing over the
unit intervals yields
\(
\sum_{t\in T}\|CA^tx\|^2
\leq2bQ_T(1+\ell^2)\|x\|^2\) for \(x\in PH\).

For $y\in(\Id-P)H$, the spectral theorem gives
$\|A^ty\|\leq2^{-t}\|y\|$ for $t>0$, and the same bound holds at
$t=0$ under our convention. Since $\|C\|^2\leq b$,
\[
\sum_{t\in T}\|CA^ty\|^2
\leq bQ_T\sum_{n\geq0}4^{-n}\|y\|^2
=\frac43bQ_T\|y\|^2.
\]
Decomposing an arbitrary vector $z\in H$ as $Pz+(\Id-P)z$ and
using $\|u+v\|^2\leq2\|u\|^2+2\|v\|^2$, we conclude that the
$T$-system is Bessel and its Gramian satisfies
\begin{equation}\label{eq:normal-local-bessel-bound}
G_T\leq4bQ_T\bigl(1+(\log2)^2+\pi^2\bigr)\Id.
\end{equation}

Conversely, suppose that $Q_T=\infty$. Subdividing unit intervals
into three equal parts, we can choose intervals
\(
J_k=[u_k-1/6,u_k+1/6)\subset\Rplus
\)
such that $q_k:=\#(T\cap J_k)\to\infty$.
For an integer $m\geq2$, put
\(
\rho_m=\e^{-1/m},
\theta_{m,j}=\frac{2\pi}{m}
\left(j-\frac{m-1}{2}\right)\) for \( 0\leq j<m\),
and define on $\mathbb C^m$
\[
A_m=\rho_m U_m,
\qquad
U_m=\diag(\e^{i\theta_{m,0}},\ldots,\e^{i\theta_{m,m-1}}),
\qquad
C_m=\frac1{\sqrt m}(1,\ldots,1).
\]
The rows $C_mU_m^s$, $0\leq s<m$, are orthonormal, and $U_m^m$
is a scalar unitary. Hence the integer Gramian $S_m$ has eigenvalues
\(
\frac{\rho_m^{2s}}{1-\rho_m^{2m}}
=\frac{\e^{-2s/m}}{1-\e^{-2}}\) for \(0\leq s<m\).
In particular,
\(
\frac{\e^{-2}}{1-\e^{-2}}\Id
\leq S_m\leq\frac1{1-\e^{-2}}\Id\) for \(m\geq2\).

Choose an integer $m_k\geq\max\{2,u_k+1\}$, and let
$x_k\in\mathbb C^{m_k}$ be the unit vector with coordinates
\(
x_{k,j}=m_k^{-1/2}\e^{-i\theta_{m_k,j}u_k}\) for \(0\leq j<m_k\).
All angles $\theta_{m,j}$ lie in $(-\pi,\pi)$, so the prescribed
Borel powers give, for $t\in J_k$,
\[
|C_{m_k}A_{m_k}^tx_k|
=\e^{-t/m_k}
\left|
\frac1{m_k}\sum_{j=0}^{m_k-1}
\e^{i\theta_{m_k,j}(t-u_k)}
\right|
\geq\e^{-1}\cos(\pi/6).
\]
Take
\(
H=\bigoplus_k\mathbb C^{m_k},
 Y=\ell^2,
 A=\bigoplus_k A_{m_k},
 C=\bigoplus_k C_{m_k}.
\)
The operator $A$ is bounded and normal, and $C$ is bounded with
norm one. The uniform bounds for $S_m$ above imply that the integer
system is exactly observable. Identifying $x_k$ with the unit vector
in $H$ supported in the $k$th block, we obtain
\(
\sum_{t\in T}\|CA^tx_k\|^2
\geq\frac{3}{4\e^2}q_k\rightarrow\infty.
\)
Thus the $T$-system is not Bessel.

Finally, every interval $[u,u+1)$ is contained in the union of at most
two consecutive integer cells. Consequently, $Q_T<\infty$ is also
equivalent to $\sup_{u\geq0}\#(T\cap[u,u+1))<\infty$.
\end{proof}

The following example shows that even natural density one does not ensure universal Bessel preservation for normal operators.

\begin{example}\label{ex:normal-local-clusters}
Let
\[
T=\N\cup\bigcup_{k\geq1}
\left\{2^{k^2}+\frac{j}{10k}:1\leq j\leq k\right\}.
\]
There are at most $\sum_{k\leq\sqrt{\log_2 R}}k=O(\log R)$
additional points below $R$, so $N_T(R)=R+O(\log R)$ as
$R\to\infty$. On the other hand, the cell
$[2^{k^2},2^{k^2}+1)$ contains at least $k+1$ points.
Thus $Q_T=\infty$ even though $0\in T$ and $d(T)=1$.
Theorem~\ref{thm:natural-density} gives $T\in\U$ for positive
operators, whereas Theorem~\ref{thm:normal-local-counting} shows that
$T$ does not universally preserve even the Bessel property for normal
operators.
\end{example}

A related sufficient local counting condition appears in
\cite[Remark~3.11]{KrishtalMashreghiMillerSectorial}
for invertible, finitely generated sectorial systems.
Theorem~\ref{thm:normal-local-counting} characterizes universal
Bessel preservation in the normal class, allowing arbitrary
observation spaces and noninvertible operators.

For the remainder of this section, fix $M\geq0$. On
$H_+=(\ker A)^\perp$, let
$L=-\operatorname{Log}(A|_{H_+})$ be the possibly unbounded normal
operator defined by the principal Borel logarithm. Assume, in the
spectral sense, that
\begin{equation}\label{eq:normal-log-sector}
|\operatorname{Im}L|\leq M\operatorname{Re}L.
\end{equation}
Equivalently, the nonzero spectral measure of $A$ is concentrated
where
$|\arg\lambda|\leq M(-\log|\lambda|)$.
This condition imposes no restriction on $\ker A$.
On the unit circle, it permits only the spectral value $1$.
Consequently, under the integer Bessel assumption,
Lemma~\ref{lem:bessel-spectral-reduction} permits reduction
to a strongly stable normal contraction whenever needed.

For $0<r<1$, write
\(
c_M=\sqrt{1+M^2}, E_r=E_{|A|}([r,1]), \delta_r=-\log r.
\)
The logarithm is bounded on $E_rH$, with
$\|LE_r\|\leq c_M\delta_r$.

\begin{lemma}\label{lem:normal-variation}
Under the integer Bessel assumption with bound $b>0$ and
the logarithmic sector condition \eqref{eq:normal-log-sector},
the continuous observability Gramian satisfies
\[
G_c:=\int_0^\infty(A^t)^*BA^t\dd t
=\int_0^1(A^s)^*SA^s\dd s\leq b\Id.
\]
For every $0<r<1$,
\(
\|E_r(G_c-S)E_r\|\leq bc_M\delta_r.
\)
Moreover, for $x\in E_rH$ and $F_x(t)=\|CA^tx\|^2$,
\[
\int_0^\infty t|F_x'(t)|\dd t
\leq\sqrt{2}\,c_Mb\|x\|^2.
\]
\end{lemma}

\begin{proof}
Identity \eqref{eq:normal-continuous-energy} gives the integral formula
and the upper bound for $G_c$. For $0\leq s\leq1$,
\(
\|(A^s-\Id)E_r\|
\leq s\|LE_r\|\leq sc_M\delta_r.
\)
Expanding $S-(A^s)^*SA^s$ therefore gives
\(
\|E_r\bigl(S-(A^s)^*SA^s\bigr)E_r\|
\leq 2bsc_M\delta_r.
\)
Integration over $s\in[0,1]$ yields
$\|E_r(S-G_c)E_r\|\leq bc_M\delta_r$.

Fix $x\in E_rH$ and set
$f_x(t)=CA^tx$ and
$q_x(t)=\langle SLA^tx,LA^tx\rangle$.
Since $L$ commutes with $A$ on this spectral subspace,
the Stein identity gives
\(
q_x(t)-q_x(t+1)
=\|CLA^tx\|^2
=\|f_x'(t)\|^2.
\)
For $R\geq1$, a change of variables gives
\[
\begin{aligned}
\int_0^R t^2\|f_x'(t)\|^2\dd t
&=\int_0^1t^2q_x(t)\dd t
 +\int_1^R(2t-1)q_x(t)\dd t\\
&\quad-\int_R^{R+1}(t-1)^2q_x(t)\dd t\\
&\leq2b\int_0^R t\|LA^tx\|^2\dd t.
\end{aligned}
\]
For $\operatorname{Re}z>0$ and
$|\operatorname{Im}z|\leq M\operatorname{Re}z$,
\[
\int_0^\infty t|z|^2\e^{-2t\operatorname{Re}z}\dd t
=\frac{|z|^2}{4(\operatorname{Re}z)^2}
\leq\frac{c_M^2}{4}.
\]
On $E_rH$, the only possible spectral value of $L$ with zero
real part is $z=0$, where the integrand vanishes.
The spectral theorem and monotone convergence consequently imply
\[
\int_0^\infty t^2\|f_x'(t)\|^2\dd t
\leq\frac{bc_M^2}{2}\|x\|^2.
\]
Finally, using
$|F_x'(t)|\leq2\|f_x(t)\|\|f_x'(t)\|$ and
\eqref{eq:normal-continuous-energy}, Cauchy--Schwarz yields
\[
\begin{aligned}
\int_0^\infty t|F_x'(t)|\dd t
&\leq
2\left(\int_0^\infty\|f_x(t)\|^2\dd t\right)^{1/2}
 \left(\int_0^\infty t^2\|f_x'(t)\|^2\dd t\right)^{1/2}\\
&\leq\sqrt{2}\,c_Mb\|x\|^2.
\end{aligned}
\]
\end{proof}

\begin{theorem}\label{thm:normal-bessel}
Fix $M\geq0$ and let $T\subset\Rplus$ be locally finite.
The following assertions are equivalent:
\begin{enumerate}[label=\textup{(\roman*)},leftmargin=*]
\item
For every bounded normal operator $A$ satisfying
\eqref{eq:normal-log-sector} and every bounded observation
operator $C$, Besselness of the $\N$-system implies
Besselness of the $T$-system.
\item
$N_T(R)=O(R)$ as $R\to\infty$.
\end{enumerate}
When these conditions hold, put
\(
K_T=\sup_{R>0}\frac{\#(T\cap(0,R])}{R}<\infty.
\)
Every normal operator system $(A,C)$ satisfying
\eqref{eq:normal-log-sector} whose integer observation family
$\{CA^n\}_{n\geq0}$ is a $g$-Bessel sequence with bound $b>0$
obeys
\begin{equation}\label{eq:normal-bessel-bound}
G_T\leq
b\bigl(\ind_{\{0\in T\}}+\sqrt{2}\,c_MK_T\bigr)\Id.
\end{equation}
\end{theorem}

\begin{proof}
Suppose that $N_T(R)=O(R)$, and let $(A,C)$ satisfy
\eqref{eq:normal-log-sector}.
Assume that $\{CA^n\}_{n\geq0}$ is a $g$-Bessel sequence
with bound $b>0$.
By Lemma~\ref{lem:bessel-spectral-reduction}, we may restrict to
$E_{|A|}([0,1))H$ and assume that $A$ is a strongly stable
normal contraction.
This reduction preserves the logarithmic sector condition,
all observations, and the integer Bessel bound.
We retain the notation $A$ and $C$ for the restricted system.
Local finiteness implies that the positive points of $T$
are bounded away from zero, so $K_T<\infty$.
Set $N_+(s)=\#(T\cap(0,s])$. For $x\in E_rH$, put
$F_x(t)=\|CA^tx\|^2$.
Strong stability gives $F_x(t)\to0$. By the weighted variation
estimate in Lemma~\ref{lem:normal-variation},
\[
tF_x(t)\leq\int_t^\infty s|F_x'(s)|\dd s\longrightarrow0.
\]
Stieltjes integration by parts therefore yields
\[
\sum_{t\in T\setminus\{0\}}F_x(t)
=-\int_0^\infty N_+(s)F_x'(s)\dd s
\leq\sqrt{2}\,c_MbK_T\|x\|^2.
\]
The boundary term at infinity vanishes since $N_+(s)\leq K_Ts$,
and $N_+$ vanishes near zero.
As $r\downarrow0$, $E_rx\to(\Id-P_0)x$, where
$P_0=E_A(\{0\})$. Since $A^tP_0=0$ for $t>0$, passing to the limit
first in each finite sum proves the same upper estimate for all
positive sampling times and every $x\in H$.
If $0\in T$, add $\|Cx\|^2\leq b\|x\|^2$.
This proves \eqref{eq:normal-bessel-bound}.

Conversely, the Parseval multiplication system used in the
proof of Theorem~\ref{thm:universal-bessel} satisfies
\eqref{eq:normal-log-sector} for every $M\geq0$.
Applying \textup{(i)} to this system and repeating the counting
argument from that proof gives $N_T(R)=O(R)$.
\end{proof}

\begin{theorem}\label{thm:normal-boundary-density}
Let $T\subset\Rplus$ be locally finite and suppose that
$N_T(R)=dR+o(R)$ with $0<d<\infty$.
Under \eqref{eq:normal-input} and \eqref{eq:normal-log-sector},
\(
\|E_r(G_T-dS)E_r\|\longrightarrow0\) as \(r\uparrow1\).
For fixed $T$, the convergence is uniform over systems with
fixed $M$ and $b$.
\end{theorem}

\begin{proof}
Theorem~\ref{thm:normal-bessel} ensures that $G_T$ is bounded.
Set $N_+(s)=\#(T\cap(0,s])$ and $D(s)=N_+(s)-ds$.
For $x\in E_rH$, put $F_x(t)=\|CA^tx\|^2$.
The preceding integration by parts argument gives
\[
\sum_{t\in T\setminus\{0\}}F_x(t)
-d\int_0^\infty F_x(t)\dd t
=-\int_0^\infty D(s)F_x'(s)\dd s.
\]
Here $D(0)=0$, and the boundary term at infinity vanishes
because $D(s)=o(s)$ and $sF_x(s)\to0$.

Given $\varepsilon>0$, choose $R_\varepsilon\geq1$ such that
$|D(s)|\leq\varepsilon s$ for $s\geq R_\varepsilon$, and put
\(
M_\varepsilon
=\sup_{0\leq s\leq R_\varepsilon}|D(s)|.
\)
By the weighted variation estimate in
Lemma~\ref{lem:normal-variation}, the integral over
$[R_\varepsilon,\infty)$ has absolute value at most
$\sqrt{2}\,c_Mb\varepsilon\|x\|^2$.

To estimate the integral over $[0,R_\varepsilon]$, the
compressed Stein identity gives
\[
\|CE_r\|^2
=\|E_rBE_r\|
\leq2b\|(\Id-A)E_r\|
\leq2bc_M\delta_r.
\]
Since $\|LE_r\|\leq c_M\delta_r$ and $A^t$ is a contraction,
\begin{equation}\label{eq:normal-pointwise-derivative}
\begin{aligned}
F_x(0)&\leq2bc_M\delta_r\|x\|^2,\\
|F_x'(t)|&\leq4bc_M^2\delta_r^2\|x\|^2,
\qquad t\geq0.
\end{aligned}
\end{equation}
Combining these estimates with the possible contribution at
time zero and the continuous Gramian comparison in
Lemma~\ref{lem:normal-variation}, we obtain
\[
\|E_r(G_T-dS)E_r\|
\leq(2+d)bc_M\delta_r
 +4bc_M^2M_\varepsilon R_\varepsilon\delta_r^2
 +\sqrt{2}\,c_Mb\varepsilon.
\]
Letting first $r\uparrow1$ and then $\varepsilon\downarrow0$
proves the convergence.
For fixed $T$, the right-hand side depends on the system
only through $M$ and $b$, which also proves the asserted
uniformity.
\end{proof}

For finitely many observation channels, the boundary comparison
has a global consequence: the sampled Gramian is a compact
perturbation of a positive invertible operator.

\begin{corollary}
\label{cor:normal-finite-observation}
Let $A\in\Bcal(H)$ be normal and let
$C\in\Bcal(H,\mathbb C^m)$, where $m\in\mathbb N^+$.
Assume \eqref{eq:normal-input} and \eqref{eq:normal-log-sector}.
Let $T\subset\Rplus$ be locally finite, with
$N_T(R)=dR+o(R)$ for some $0<d<\infty$.
Then
\(
G_T-dS\) is compact. Consequently,
\begin{equation}\label{eq:normal-completeness-observability}
\{CA^t\}_{t\in T}\text{ is exactly observable}
\quad\Longleftrightarrow\quad
\bigcap_{t\in T}\ker(CA^t)=\{0\}.
\end{equation}
Moreover, there exists a finite set $F\subset\N$ such that
the $(T\cup F)$-system is exactly observable.
\end{corollary}

\begin{proof} By Theorem~\ref{thm:normal-bessel}, the map
$\Gamma_Tx=(CA^tx)_{t\in T}$ defines a bounded operator
from $H$ into $\ell^2(T;\mathbb C^m)$.
Integer exact observability and normality imply that $A$ is a
strongly stable contraction. For $0<r<1$, put
\(
P_r=E_{|A|}([0,r)), E_r=\Id-P_r.
\)
Since $C$ has finite rank and $\|A^nP_r\|\leq r^n$, the series
\[
P_rSP_r=\sum_{n\geq0}(CA^nP_r)^*(CA^nP_r)
\]
converges in operator norm to a compact operator. Indeed, its
tail starting at $n=N$ has norm at most
$\|C\|^2r^{2N}/(1-r^2)$. Since $P_rSP_r\geq aP_r$,
the space $P_rH$ is finite dimensional.

Set $K=G_T-dS$. Theorem~\ref{thm:normal-boundary-density} gives
$\|E_rKE_r\|\to0$ as $r\uparrow1$, whereas
\[
K-E_rKE_r=P_rK+KP_r-P_rKP_r
\]
has finite rank. Thus $K$ is compact.

Since $dS\geq da\Id$, the positive operator $G_T=dS+K$
is Fredholm. Hence $\ker G_T$ is finite dimensional, and
there exists $\alpha>0$ such that
\(
\langle G_Tx,x\rangle\geq\alpha\|x\|^2\) for \(x\in(\ker G_T)^\perp\).
Using $G_T=\Gamma_T^*\Gamma_T$ and
$\ker G_T=\ker\Gamma_T$, we obtain
\eqref{eq:normal-completeness-observability}.

Finally, set $\mathcal N_T=\ker\Gamma_T$.
The decreasing sequence of subspaces
\(
\mathcal N_T\cap\bigcap_{n=0}^{N}\ker(CA^n)
\)
has zero intersection by integer exact observability.
Since $\mathcal N_T$ is finite dimensional, one of these
subspaces is zero. Choose such an $N$ and put
$F=\{0,1,\ldots,N\}$. The set $T\cup F$ has the same natural
density as $T$ and its observation family has trivial common
kernel. Applying the equivalence already proved gives exact
observability at $T\cup F$.
\end{proof}

No assumption that $0\in T$ is needed in this corollary.
The same proof applies to a compact observation operator
$C\in\Bcal(H,Y)$ with arbitrary observation space $Y$:
the summands defining $P_rSP_r$ are then compact, and the
same norm convergence argument shows that $P_r$ has finite rank.

\begin{corollary}\label{cor:normal-counting-stability}
Let $T,T_0\subset\Rplus$ be locally finite, with
$N_{T_0}(R)=O(R)$ and
$N_T(R)-N_{T_0}(R)=o(R)$.
Under \eqref{eq:normal-input} and \eqref{eq:normal-log-sector},
both sampled Gramians are bounded and
\(
\|E_r(G_T-G_{T_0})E_r\|\longrightarrow0\) as \(r\uparrow1\).
For fixed $T$ and $T_0$, the convergence is uniform over
systems with fixed $M$ and $b$.
\end{corollary}

\begin{proof}
Both time sets satisfy the linear upper counting bound, so
Theorem~\ref{thm:normal-bessel} applies.
Put
\(
D(s)=\#(T\cap(0,s])-\#(T_0\cap(0,s]).
\)
Then $D(s)=o(s)$.
For $x\in E_rH$, set $F_x(t)=\|CA^tx\|^2$.
As in the proof of Theorem~\ref{thm:normal-boundary-density},
integration by parts gives
\[
\langle(G_T-G_{T_0})x,x\rangle
=\bigl(\ind_{\{0\in T\}}-\ind_{\{0\in T_0\}}\bigr)F_x(0)
\quad-\int_0^\infty D(s)F_x'(s)\dd s.
\]
Given $\varepsilon>0$, choose $R_\varepsilon\geq1$ such that
$|D(s)|\leq\varepsilon s$ for $s\geq R_\varepsilon$.
Estimate the tail using the weighted variation estimate in
Lemma~\ref{lem:normal-variation}, and the contributions from
the bounded interval and time zero using
\eqref{eq:normal-pointwise-derivative}.
This yields
\[
\|E_r(G_T-G_{T_0})E_r\|
\leq2bc_M\delta_r
 +4bc_M^2\delta_r^2
   \int_0^{R_\varepsilon}|D(s)|\dd s
\quad+\sqrt{2}\,c_Mb\varepsilon.
\]
Letting first $r\uparrow1$ and then $\varepsilon\downarrow0$
proves the convergence.
For fixed $T$ and $T_0$, this estimate also gives the
asserted uniformity.
\end{proof}

\begin{remark}\label{rem:normal-aliasing}
Even for a single observation channel, the completeness
condition in Corollary~\ref{cor:normal-finite-observation}
cannot be omitted.
For $M>0$, put
\(
\rho=\exp\!\left(-\frac{\pi}{2M}\right),
A=\rho\diag(i,-i), C=(1,1), T=2\N,
\)
with $A$ acting on $\mathbb C^2$.
This system satisfies \eqref{eq:normal-log-sector}.
The integer system is Bessel since $0<\rho<1$, and the rows
$C=(1,1)$ and $CA=i\rho(1,-1)$ are linearly independent.
Hence the integer system is exactly observable.
However,
\(
A^{2k}=(-\rho^2)^k\Id,
CA^{2k}(1,-1)^{\mathsf T}=0\) for \(k\geq0\).
Thus $0\in T$ and $d(T)=1/2$ do not ensure exact observability.
See also \cite[Example~6.1]{KrishtalMashreghiMillerSectorial}
for this angular aliasing mechanism.
This example shows that Theorem~\ref{thm:natural-density}
cannot be extended to this normal class merely by imposing
a fixed logarithmic sector condition.
\end{remark}

\section{Operator orbit frames and sampling obstructions}

We conclude by relating the preceding observability results to
operator orbit frames and realizing the clustered sampling
obstruction with any prescribed finite number of generators,
including a single generator.
The positive operator results cover finite and countable families of generating vectors. The normal operator extension is recorded in Corollary~\ref{cor:normal-multiorbit}.

Let $I$ be finite or countable and let
\(
\mathcal F=(f_i)_{i\in I}
\)
be an indexed family in $H$. Define
\(
C_{\mathcal F}x=(\langle x,f_i\rangle)_{i\in I}.
\)
Whenever the integer orbit family is Bessel, its subfamily at time \(0\)
\(
\{f_i:i\in I\}
\)
is Bessel, so
\(C_{\mathcal F}:H\to\ell^2(I)\)
is bounded. Since $A=A^*$,
\begin{equation}\label{eq:orbit-observation-identity}
\sum_{n\geq0}\lVert C_{\mathcal F}A^nx\rVert_{\ell^2(I)}^2
=
\sum_{n\geq0}\sum_{i\in I}
|\langle x,A^nf_i\rangle|^2.
\end{equation}
For a time set $T$, write
\(
\mathcal O_T(A,\mathcal F)
=
\{A^tf_i:t\in T,\ i\in I\}.
\)

If $\mathcal F$ is Bessel, the $T$-version of
\eqref{eq:orbit-observation-identity} shows that
$\mathcal O_T(A,\mathcal F)$ is a frame for $H$ if and only if
\(
G_{T,\mathcal F}
=
\sum_{t\in T}
A^tC_{\mathcal F}^*C_{\mathcal F}A^t
\)
converges strongly to a bounded positive invertible operator. Thus the observability results obtained above apply directly to finite and
countable multi-orbit frames.

\begin{theorem}
\label{thm:universal-multiorbit}
Let $T\subset\Rplus$ be locally finite. Then the following are equivalent:
\begin{enumerate}[label=\textup{(\roman*)}]
\item
$T\in\U$;

\item
for every separable Hilbert space $H$, every positive
$A\in\Bcal(H)$, and every finite or countable indexed family
\(
\mathcal F=(f_i)_{i\in I}
\)
in $H$,
\(
\mathcal O_{\N}(A,\mathcal F)\text{ is a frame}
\Longrightarrow
\mathcal O_T(A,\mathcal F)\text{ is a frame}.
\)
\end{enumerate}
\end{theorem}

\begin{proof}
If $T\in\U$, then \eqref{eq:orbit-observation-identity} identifies the frame inequalities at integer times and at the times in $T$ with the
corresponding observability inequalities for $(A,C_{\mathcal F})$.
Hence \textup{(i)} implies \textup{(ii)}.

Conversely, let $C:H\to Y$ be an arbitrary observation operator and choose
an orthonormal basis
\(
(e_i)_{i\in I}
\)
of $Y$. Set
\(
f_i=C^*e_i
\)
and
\(
\mathcal F=(f_i)_{i\in I}.
\)
Then
\(
C_{\mathcal F}x=(\langle Cx,e_i\rangle)_{i\in I},
\)
so $C_{\mathcal F}$ is $C$ followed by the coordinate isometry from $Y$
into $\ell^2(I)$. Therefore \textup{(ii)} implies universal preservation
for every observation system, and hence $T\in\U$.
\end{proof}

The obstruction from Section~3 already occurs for any prescribed finite
number of generating orbits, including a single orbit. This also gives a
stronger form of the recent counterexamples to the conjecture of
Aldroubi, Cabrelli, Krishtal, and Molter \cite{AldroubiSurvey}, which
asserted that the M\"untz--Sz\'asz divergence condition, together with
the Bessel property, should preserve a positive diagonal Carleson frame.
Gallardo-Guti\'errez and Partington
\cite[Example~3.4]{GallardoPartingtonOperatorOrbits}
disproved this conjecture by Bessel examples for which the necessary
scalar lower bound fails. In contrast,
the time set constructed in Section~3 satisfies
\(0\in T\) and \(N_T(R)\asymp R\), and hence the corresponding
scalar Abel quantity is
uniformly bounded above and below, while the sampled orbit is still
Bessel but not a frame. The following result shows that this stronger
obstruction persists for every prescribed finite number of generating
orbits.

\begin{theorem}
\label{thm:mellin-finite-orbit}
Let $T$ be the set defined in \eqref{eq:clustered-T}. For every
$m\in\mathbb N^+$, there exist a separable Hilbert space $H^{(m)}$, a
positive diagonal contraction $A^{(m)}\in\Bcal(H^{(m)})$, and an
$m$-element family
\(
\mathcal F^{(m)}
=
\bigl(f_1^{(m)},\ldots,f_m^{(m)}\bigr)
\subset H^{(m)}
\)
such that
\(
\mathcal O_{\N}(A^{(m)},\mathcal F^{(m)})
\)
is a frame for $H^{(m)}$, whereas
\(
\mathcal O_T(A^{(m)},\mathcal F^{(m)})
\)
is Bessel but not a frame.

Moreover, every eigenvalue of $A^{(m)}$ has multiplicity $m$, and $m$ is
the minimal number of generating orbits for an integer time orbit frame
generated by $A^{(m)}$. In particular, $m=1$ gives a singly generated
positive diagonal Carleson frame whose $T$-sampled orbit is Bessel but
not a frame.
\end{theorem}

\begin{proof}
We first construct a singly generated system. Retain the notation
$q=16$, $a_i=q^{-i}$, $K_M$, $G_{M,L}$, and $x^{(M)}$ from the proof of
Theorem~\ref{thm:mellin-counterexample}. For $k\in\mathbb N^+$, set
\(
u_k=q^{-k},
\lambda_k=\e^{-u_k}.
\)

For sufficiently small $u,v>0$,
\(
\left(\frac{1-\e^{-2u}}{2u}\right)^{1/2}
=
1+O(u),
\frac{u+v}{1-\e^{-(u+v)}}
=
1+O(u+v).
\)
Hence there exist $K_1\in\mathbb N^+$ and $C_q<\infty$ such that
\[
\left|
\frac{
\sqrt{1-\e^{-2u_k}}
\sqrt{1-\e^{-2u_\ell}}
}{
1-\e^{-(u_k+u_\ell)}
}
-
\frac{2\sqrt{u_ku_\ell}}{u_k+u_\ell}
\right|
\leq
C_q\sqrt{u_ku_\ell},
\qquad
k,\ell\geq K_1.
\]
Choose $K_0\geq K_1$ so large that
\(
C_q\sum_{k\geq K_0}q^{-k}<\frac{m_q}{2}.
\)

Since, for each fixed $M$,
\(
G_{M,L}\rightarrow K_M
\)
in operator norm as $L\to\infty$, we may choose integers $L_M$ so large
that the intervals
\(
J_M
=
\{L_M-M,\ldots,L_M+M\}
\)
are pairwise disjoint, satisfy
\(
J_M\subset\{K_0,K_0+1,\ldots\},
\)
and
\begin{equation}\label{eq:multiorbit-bad-block}
\lVert G_{M,L_M}-K_M\rVert
\leq
\frac1M.
\end{equation}

Set
\(
J=\bigcup_{M\geq1}J_M,
H_0=\ell^2(J),
\)
and let $(e_k)_{k\in J}$ be the standard orthonormal basis of $H_0$.
Define
\(
De_k=\lambda_ke_k,
b=
\sum_{k\in J}
\sqrt{1-\lambda_k^2}\,e_k.
\)
Since
\(
1-\lambda_k^2
=
1-\e^{-2q^{-k}}
\leq
2q^{-k},
\)
we have $b\in H_0$, and $D$ is a positive diagonal contraction.

We first show that the integer orbit $\{D^nb:n\in\N\}$ is a frame.
For finitely supported vectors in $H_0$, summing the geometric series
shows that the quadratic form
$x\mapsto\sum_{n\geq0}|\langle x,D^nb\rangle|^2$ has matrix
\[
S_{k\ell}
=
\frac{
\sqrt{1-\e^{-2u_k}}
\sqrt{1-\e^{-2u_\ell}}
}{
1-\e^{-(u_k+u_\ell)}
},
\qquad
k,\ell\in J.
\]
Consider also
\[
R_{k\ell}
=
\frac{2\sqrt{u_ku_\ell}}{u_k+u_\ell}
=
\sech\!\left(
\frac{(k-\ell)\log q}{2}
\right).
\]
By Lemma~\ref{lem:continuous-riesz}, the corresponding operator on
$\ell^2(\mathbb Z)$ satisfies
\(
m_q\Id\leq R\leq M_q\Id.
\)
Hence its compression $R_J$ to $\ell^2(J)$ satisfies the same bounds.

By the choice of $K_0$,
\(
|S_{k\ell}-R_{k\ell}|
\leq
C_q\sqrt{u_ku_\ell}
\)
for all \(k,\ell\in J\). This estimate defines a bounded perturbation
of $R_J$, so the matrix $S$ defines a bounded operator satisfying
\(
\lVert S-R_J\rVert
\leq
C_q\sum_{k\in J}u_k
\leq
C_q\sum_{k\geq K_0}q^{-k}
<
\frac{m_q}{2}.
\)
Hence
\begin{equation}\label{eq:single-orbit-frame-bounds}
\frac{m_q}{2}\Id
\leq
S
\leq
\left(M_q+\frac{m_q}{2}\right)\Id.
\end{equation}
The upper bound extends the analysis map
$x\mapsto(\langle x,D^nb\rangle)_{n\geq0}$ from finitely supported
vectors to a bounded operator on $H_0$. Since each coordinate
functional is continuous, the extension is the orbit analysis operator.
By density, \eqref{eq:single-orbit-frame-bounds} gives the frame
inequalities on all of $H_0$, and $S$ is the frame operator.
Thus $\{D^nb:n\in\N\}$ is a frame for $H_0$.

Since \(N_T(R)=O(R)\), Theorem~\ref{thm:universal-bessel}, applied
to \(A=D\) and \(Cx=\langle x,b\rangle\), shows that
\(\{D^tb:t\in T\}\) is Bessel. Its frame operator \(G_T^{(0)}\) has
entries
\[
(G_T^{(0)})_{k\ell}
=\sqrt{1-\e^{-2u_k}}\sqrt{1-\e^{-2u_\ell}}
\sum_{t\in T}\e^{-(u_k+u_\ell)t},
\qquad k,\ell\in J.
\]

It remains to prove that this family has no positive lower frame bound.
For each $M\in\mathbb N^+$, define a unit vector $y^{(M)}\in H_0$,
supported on $J_M$, by
\(
y_{L_M+i}^{(M)}
=
\frac{(-1)^i}{\sqrt{2M+1}}\) for \(|i|\leq M\).
Identifying $e_i$ with $e_{L_M+i}$ for $|i|\leq M$, we have
\(
u_{L_M+i}
=
q^{-(L_M+i)}
=
\frac{a_i}{q^{L_M}}.
\)
Thus the compression of $G_T^{(0)}$ to $\ell^2(J_M)$ is exactly
$G_{M,L_M}$, and $y^{(M)}$ corresponds to $x^{(M)}$. By
\eqref{eq:multiorbit-bad-block},
\[
\langle G_T^{(0)}y^{(M)},y^{(M)}\rangle
=
\langle G_{M,L_M}x^{(M)},x^{(M)}\rangle
\leq
\langle K_Mx^{(M)},x^{(M)}\rangle+\frac1M
=
\lVert\mathcal Mx^{(M)}\rVert_2^2+\frac1M
\rightarrow0
\]
by Lemma~\ref{lem:continuous-riesz}. Hence $G_T^{(0)}$ is not bounded
below. Therefore $\{D^tb:t\in T\}$ is Bessel but not a frame.

We now fix $m\in\mathbb N^+$ and pass to $m$ generating orbits. Let
\(
H^{(m)}
=
H_0\otimes\mathbb C^m,
A^{(m)}
=
D\otimes\Id_{\mathbb C^m},
\)
and let $\varepsilon_1,\ldots,\varepsilon_m$ be the standard orthonormal
basis of $\mathbb C^m$. Define
\(
f_j^{(m)}
=
b\otimes\varepsilon_j\) for \(1\leq j\leq m,\)
and
\(
\mathcal F^{(m)}
=
\bigl(f_1^{(m)},\ldots,f_m^{(m)}\bigr).
\)
The associated observation operator is
\(
C_{\mathcal F^{(m)}}:H^{(m)}\longrightarrow\mathbb C^m,
C_{\mathcal F^{(m)}}x
=
\bigl(
\langle x,f_1^{(m)}\rangle,\ldots,
\langle x,f_m^{(m)}\rangle
\bigr).
\)

If
\(
x
=
\sum_{j=1}^m x_j\otimes\varepsilon_j
\in H^{(m)},
\)
then, for every $t\geq0$,
\(
\lVert
C_{\mathcal F^{(m)}}(A^{(m)})^t x
\rVert^2
=
\sum_{j=1}^m
|\langle x_j,D^tb\rangle|^2.
\)
Consequently,
\(
\sum_{n\geq0}
\lVert
C_{\mathcal F^{(m)}}(A^{(m)})^n x
\rVert^2
=
\sum_{j=1}^m
\sum_{n\geq0}
|\langle x_j,D^nb\rangle|^2.
\)
By \eqref{eq:single-orbit-frame-bounds}, the $\N$-system for
$(A^{(m)},C_{\mathcal F^{(m)}})$ is exactly observable. Equivalently,
\(
\left\{
(A^{(m)})^n f_j^{(m)}:
n\in\N,\ 1\leq j\leq m
\right\}
\)
is a frame for $H^{(m)}$, with bounds independent of $m$.

Similarly,
\(
\sum_{t\in T}
\lVert
C_{\mathcal F^{(m)}}(A^{(m)})^t x
\rVert^2
=
\sum_{j=1}^m
\sum_{t\in T}
|\langle x_j,D^tb\rangle|^2.
\)
Hence the $T$-system is Bessel. Equivalently,
\(
\left\{
(A^{(m)})^t f_j^{(m)}:
t\in T,\ 1\leq j\leq m
\right\}
\)
is Bessel.

On the other hand, let
\(
z^{(M)}
=
y^{(M)}\otimes\varepsilon_1.
\)
Then $\lVert z^{(M)}\rVert=1$ and
\[
\sum_{t\in T}
\lVert
C_{\mathcal F^{(m)}}(A^{(m)})^t z^{(M)}
\rVert^2
=
\sum_{t\in T}
|\langle y^{(M)},D^tb\rangle|^2
=
\langle G_T^{(0)}y^{(M)},y^{(M)}\rangle
\longrightarrow0.
\]
Thus the $T$-system is not exactly observable, and the corresponding
$T$-sampled orbit family is not a frame.

Finally, each eigenvalue $\lambda_k$ of $A^{(m)}$ has eigenspace
\(E_k=\operatorname{span}\{e_k\}\otimes\mathbb C^m\), with
\(\dim E_k=m\).
Suppose that $r<m$ vectors $g_1,\ldots,g_r\in H^{(m)}$ generated a
complete integer time orbit family. Let $P_k$ be the orthogonal projection
onto $E_k$. Since
\(
\dim\operatorname{span}
\{P_kg_1,\ldots,P_kg_r\}
\leq r<m,
\)
there exists a nonzero vector $v\in E_k$ orthogonal to
$P_kg_1,\ldots,P_kg_r$. Since
\(
A^{(m)}v=\lambda_kv,
\)
we have
\(
\langle v,(A^{(m)})^ng_j\rangle
=
\lambda_k^n\langle v,g_j\rangle
=
0
\)
for every $n\in\N$ and $1\leq j\leq r$. Hence no family generated by
fewer than $m$ orbits can be complete. Since
$\mathcal F^{(m)}$ contains exactly $m$ generators, $m$ is the minimal
number of generating orbits.

By the standard characterization of singly generated diagonal orbit
frames (see, for example, \cite{ChristensenEtAlMystery,
KrishtalMillerDemystifying}), the sequence
\((\lambda_k)_{k\in J}\) satisfies Carleson's interpolation condition.
Hence $\{D^nb:n\in\N\}$ is a singly generated positive diagonal
Carleson frame.
\end{proof}

\begin{remark}
\label{rem:mellin-finite-orbit}
The same time set $T$ works simultaneously for every
$m\in\mathbb N^+$. Thus the failure in
Theorem~\ref{thm:mellin-counterexample} is not caused by the infinite
dimensional observation space used in its direct sum construction. It
already occurs for a single generating orbit and, more generally, for
every prescribed finite number of generating orbits. In particular,
even in the singly generated positive diagonal Carleson setting, the
scalar condition
\(
N_T(R)\asymp R
\)
together with the Bessel property does not force the sampled orbit to
have a positive lower frame bound.
The generating operator here has infinitely many distinct spectral values.
This is consistent with Theorem~\ref{thm:fixed-spectral-count}: the
number of generating orbits is not the number of spectral points.

To compare directly with the conjecture in \cite{AldroubiSurvey}, set
\(
T_+=T\setminus\{0\}.
\)
Then
\(
N_{T_+}(R)\asymp R,
\)
and consequently
\(
\sum_{t\in T_+}\frac1t=\infty.
\)
Removing the time $0$ preserves the Bessel property and cannot restore a
positive lower frame bound. Hence the same construction gives a
counterexample with strictly positive sampling times satisfying the
M\"untz--Sz\'asz divergence condition.
\end{remark}

The density and stability results have the following orbit formulations.

\begin{corollary}
\label{cor:multiorbit-lifting}
Let \(T\subset\Rplus\) be locally finite.
Preservation below means preservation from \(\N\) to the given
sampling set for every bounded positive operator and every
finite or countable indexed family of generators.
\begin{enumerate}[label=\textup{(\roman*)},leftmargin=*]
\item
Besselness is preserved at \(T\) if and only if \(N_T(R)=O(R)\).
In this case, \eqref{eq:universal-bessel-bound} applies with
\(b\) any Bessel bound for the integer orbit family.

\item
If \(d(T)\) exists, frames are preserved at \(T\) if and only if
\(0\in T\) and \(0<d(T)<\infty\).

\item
If \(T\) contains a subset of positive finite natural density,
frames are preserved at \(T\) if and only if
\(0\in T\) and \(N_T(R)=O(R)\).
This applies in particular when
\(D^-(T)>0\) or \(d^-_{\mathrm{mac}}(T)>0\).

\item
If \(T_0\subset\Rplus\) is locally finite,
\(0\in T\cap T_0\), and
\(N_T(R)-N_{T_0}(R)=o(R)\), then frame preservation at \(T\)
is equivalent to frame preservation at \(T_0\).
\end{enumerate}
\end{corollary}

\begin{proof}
Whenever the integer orbit family is Bessel, its time zero
subfamily is Bessel. Hence $C_{\mathcal F}$ is bounded,
and \eqref{eq:orbit-observation-identity} applies.
For \textup{(i)}, apply Theorem~\ref{thm:universal-bessel}; conversely,
represent an arbitrary observation operator by coordinates in an
orthonormal basis of its observation space, as in the proof of
Theorem~\ref{thm:universal-multiorbit}.
The remaining assertions follow from
Theorems~\ref{thm:natural-density}, \ref{thm:occupancy}, and
\ref{thm:counting-stability},
Proposition~\ref{prop:positive-density-subset}, and
Corollary~\ref{cor:macroscopic-density}.
\end{proof}

The commuting multi-orbit case reduces further to scalar conditions on
the time set. Here, as in Section~3, commutativity means
\(
[A,C_{\mathcal F}^*C_{\mathcal F}]=0.
\)

Under this assumption, suppose in addition that
$\mathcal O_{\N}(A,\mathcal F)$ is a frame. Then the spectral direct integral argument used in the proof of
Theorem~\ref{thm:commuting-universal} shows that
$\mathcal O_T(A,\mathcal F)$ is a frame if and only if
\[
0<
\essinf_{\lambda\in\sigma(A)\cap[0,1)}
(1-\lambda^2)\sum_{t\in T}\lambda^{2t}
\leq
\esssup_{\lambda\in\sigma(A)\cap[0,1)}
(1-\lambda^2)\sum_{t\in T}\lambda^{2t}
<\infty,
\]
where the essential bounds are taken with respect to the spectral
representation of $A$. Consequently, universal preservation over all
multi-orbit systems satisfying
\(
[A,C_{\mathcal F}^*C_{\mathcal F}]=0
\)
is equivalent to
\(
0\in T
\)
and
\(
N_T(R)\asymp R.
\)

\begin{remark}
\label{rem:finite-noncommuting}
Suppose that $H$ is infinite dimensional and $I$ is finite. Then
$C_{\mathcal F}^*C_{\mathcal F}$ has finite rank, so
\(
\ker(C_{\mathcal F}^*C_{\mathcal F})
\)
is nontrivial. If
\(
[A,C_{\mathcal F}^*C_{\mathcal F}]=0,
\)
this kernel is $A$-invariant. Hence, for every
\(
x\in\ker(C_{\mathcal F}^*C_{\mathcal F}),
\)
we have
\(
C_{\mathcal F}A^nx=0
\)
for all $n\geq0$, and therefore the integer orbit family
\(
\mathcal O_{\N}(A,\mathcal F)
\)
cannot be complete. Thus frames generated by finitely many vectors
in infinite dimension necessarily lie outside this commuting subclass.
The commuting theory is instead naturally suited to countably generated
or finite dimensional systems.
\end{remark}

The normal extension in Section~\ref{sec:normal-extension} has a
corresponding formulation for adjoint orbits. For normal $A$, define
\(
\mathcal O_T^*(A,\mathcal F)
=\{(A^t)^*f_i:t\in T,\ i\in I\}.
\)
The adjoint is taken after forming the Borel power specified in that
section. For a Bessel family $\mathcal F$, one has
\begin{equation}\label{eq:normal-orbit-observation}
\sum_{t\in T}\|C_{\mathcal F}A^tx\|^2
=\sum_{t\in T}\sum_{i\in I}
|\langle x,(A^t)^*f_i\rangle|^2.
\end{equation}
Write
\(
\widetilde G_{T,\mathcal F}
=\sum_{t\in T}(A^t)^*C_{\mathcal F}^*C_{\mathcal F}A^t
\)
for the corresponding frame operator when the family is Bessel.

\begin{corollary}\label{cor:normal-multiorbit}
Let \(T\subset\Rplus\) be locally finite.
All operators below are bounded and normal, and all indexed
generating families are finite or countable.
\begin{enumerate}[label=\textup{(\roman*)},leftmargin=*]
\item
Besselness of adjoint-orbit families is universally preserved
from \(\N\) to \(T\) if and only if \(Q_T<\infty\), where
\(Q_T\) is defined in Theorem~\ref{thm:normal-local-counting}.

\item
For each fixed \(M\geq0\), the same preservation property
within the class satisfying \eqref{eq:normal-log-sector}
holds if and only if \(N_T(R)=O(R)\).

\item
Suppose that \(A\) satisfies \eqref{eq:normal-log-sector} and
\(\mathcal O_{\N}^*(A,\mathcal F)\) is a frame. Put
\(S_{\mathcal F}=\widetilde G_{\N,\mathcal F}\) and
\(E_r=E_{|A|}([r,1])\).
If \(N_T(R)=dR+o(R)\) with \(0<d<\infty\), then
\(
\|E_r(\widetilde G_{T,\mathcal F}-dS_{\mathcal F})E_r\|
\rightarrow0\) as \(r\uparrow1\).

\item
Under the operator and frame assumptions in \textup{(iii)},
if \(T_0\subset\Rplus\) is locally finite,
\(N_{T_0}(R)=O(R)\), and
\(N_T(R)-N_{T_0}(R)=o(R)\), then
\(
\|E_r(\widetilde G_{T,\mathcal F}
      -\widetilde G_{T_0,\mathcal F})E_r\|
\rightarrow0\) as \(r\uparrow1\).

\item
Suppose that $I$ is finite, $A$ satisfies
\eqref{eq:normal-log-sector}, and
$\mathcal O_{\N}^*(A,\mathcal F)$ is a frame.
If $N_T(R)=dR+o(R)$ with $0<d<\infty$, then
$\mathcal O_T^*(A,\mathcal F)$ is a frame if and only if
it is complete. Moreover, there exists a finite set
$F\subset\N$ such that
$\mathcal O_{T\cup F}^*(A,\mathcal F)$ is a frame.
\end{enumerate}
\end{corollary}

\begin{proof}
Whenever the integer adjoint-orbit family is Bessel, its
time zero subfamily is Bessel, so $C_{\mathcal F}$ is bounded.
Identity \eqref{eq:normal-orbit-observation} identifies the
observation Gramians with the corresponding frame operators.
The sufficient assertion in \textup{(i)} follows from
Theorem~\ref{thm:normal-local-counting}. Conversely, every bounded
observation operator $C:H\to Y$ is represented by the Bessel family
$(C^*e_i)_i$, where $(e_i)_i$ is an orthonormal basis of $Y$.
Thus necessity follows from the same theorem.
Both directions of \textup{(ii)} follow from
Theorem~\ref{thm:normal-bessel} by the same correspondence
between observation operators and indexed families.
Assertions \textup{(iii)} and \textup{(iv)} follow from
Theorem~\ref{thm:normal-boundary-density} and
Corollary~\ref{cor:normal-counting-stability}, respectively.
For \textup{(v)}, the observation space of $C_{\mathcal F}$
is finite dimensional, so
Corollary~\ref{cor:normal-finite-observation}, together with
\eqref{eq:normal-orbit-observation}, gives all assertions.
\end{proof}

For positive operators, the preceding results separate universal upper
bounds from universal exact observability. The former are characterized by \(N_T(R)=O(R)\), whereas
\[
\U\subsetneq\Ucom
=\{T:0\in T,\ N_T(R)\asymp R\}.
\]
Theorem~\ref{thm:fixed-spectral-count} shows that \(0\in T\) and
\(N_T(R)\asymp R\) also characterize uniform bound preservation for
each fixed number of spectral points, with arbitrary multiplicities.
Corollary~\ref{cor:no-fixed-spectral-test} shows that the optimal lower
bounds can nevertheless decrease to zero as the number of spectral
points grows. For each fixed pair of integer observability bounds,
universality requires a positive lower bound independent of the number
of spectral points.

For the lower bound, any subset of positive finite natural density
suffices when \(0\in T\) and \(N_T(R)=O(R)\).
This includes both the natural density criterion and the condition
\(D^-(T)>0\), \(N_T(R)=O(R)\).
Examples~\ref{ex:occupancy-no-density} and
\ref{ex:density-unbounded-gaps} show that these two sufficient conditions
do not imply one another. Example~\ref{ex:unbounded-upper-beurling}
shows that finite upper Beurling density is unnecessary.

Theorem~\ref{thm:counting-stability} further shows that, for time sets
containing \(0\), an \(o(R)\) difference of counting functions does not
change universality. This is distinct from quantitative comparison
with a prescribed lattice. For regular power perturbations,
Section~5 identifies the square root scale for this comparison.
The exact general criterion is uniform bound preservation on radial
Stein--Pick cones, as stated in
Theorem~\ref{thm:complete-characterization}. A remaining problem is to
characterize its lower bound directly in terms of the time set, without
testing finite spectral models of arbitrarily large size.

\section*{Statements and Declarations}

\noindent\textbf{Data availability.}
No data were used in this study.

\medskip
\noindent\textbf{Competing interests.}
The author declares that there is no conflict of interest.

\medskip
\noindent\textbf{Funding.}
The author received no funding for this work.

\end{document}